\documentclass[reqno]{amsart}
\usepackage[bookmarks=false]{hyperref} 
\usepackage{amssymb}
\usepackage{amsmath}
\usepackage{amsthm} 
\usepackage{mathtools}
\usepackage{color} 
\usepackage[svgnames]{xcolor}
\usepackage[utf8]{inputenc}
\usepackage[T1]{fontenc}
\usepackage{graphicx}
\usepackage{placeins}

\hypersetup{
  colorlinks=true,
  linkcolor=Blue  ,          
  citecolor=Red,        
  filecolor=Magenta,      
  urlcolor=Green,           
pdfpagemode=UseOutlines,
pdftitle={},
pdfauthor={Stephen Gustafson <gustaf@math.ubc.ca>, Stefan Le Coz
  <slecoz@math.univ-toulouse.fr>, Tai-Peng Tsai <ttsai@math.ubc.ca>},
pdfsubject={stability of periodic waves of 1d cubic nonlinear schrödinger equations},
pdfkeywords={stability, periodic waves, nonlinear schrödinger equations} 
}

\numberwithin{equation}{section}
\numberwithin{figure}{section}

\newtheorem{theorem}{Theorem}[section]
\newtheorem{proposition}[theorem]{Proposition}
\newtheorem{lemma}[theorem]{Lemma}
\newtheorem{observation}[theorem]{Observation}
\newtheorem{conjecture}[theorem]{Conjecture}
\newtheorem{corollary}[theorem]{Corollary}
\newtheorem*{acknowledgments}{Acknowledgments}

\theoremstyle{definition}

\theoremstyle{remark}
\newtheorem{remark}[theorem]{Remark}

\DeclarePairedDelimiter{\norm}{\lVert}{\rVert}
\DeclarePairedDelimiter{\abs}{\lvert}{\rvert}
\DeclarePairedDelimiter{\bket}{\lbrace}{\rbrace}

\newcommand{\scalar}[2]{\left( #1,#2 \right)}
\newcommand{\ps}[2]{\left( #1,#2 \right)}

\newcommand{\dual}[2]{\left\langle #1,#2 \right\rangle}

\newcommand{\eps}{\varepsilon}
\newcommand{\la}{\lambda}
\newcommand{\vp}{\varphi}

\newcommand{\N}{\mathbb{N}}
\newcommand{\R}{\mathbb{R}}
\newcommand{\C}{\mathbb C}

\DeclareMathOperator{\sech}{sech}
\DeclareMathOperator{\sn}{sn}
\DeclareMathOperator{\cn}{cn}
\DeclareMathOperator{\dn}{dn}
\DeclareMathOperator{\pq}{pq}

\renewcommand{\leq}{\leqslant}
\renewcommand{\geq}{\geqslant}

\DeclareMathAlphabet{\mathpzc}{OT1}{pzc}{m}{it}
\renewcommand{\Re}{\mathcal R\!\mathpzc{e}}
\renewcommand{\Im}{\mathcal I\!\mathpzc{m}}

\begin{document}

\title[Stability of periodic waves of 1D cubic NLS]{Stability of periodic waves of\\ 1D cubic nonlinear Schr\"odinger equations}

\author[S.~Gustafson]{Stephen Gustafson}
\thanks{The work of S. G. is partially supported by NSERC grant 251124-12}

\author[S.~Le Coz]{Stefan Le Coz}
\thanks{The work of S. L. C. is 
  partially supported by ANR-11-LABX-0040-CIMI within the
  program ANR-11-IDEX-0002-02 and  ANR-14-CE25-0009-01}

\author[T.-P.~Tsai]{Tai-Peng Tsai}
\thanks{The work of T. T. is partially supported by NSERC grant 261356-13}

\address[Stephen Gustafson and Tai-Peng Tsai]{
Department of Mathematics,
\newline\indent
University of British Columbia,
\newline\indent
Vancouver BC
\newline\indent
Canada V6T 1Z2}
\email[Stephen Gustafson]{gustaf@math.ubc.ca}
\email[Tai-Peng Tsai]{ttsai@math.ubc.ca}

\address[Stefan Le Coz]{Institut de Math\'ematiques de Toulouse,
  \newline\indent
  Universit\'e Paul Sabatier
  \newline\indent
  118 route de Narbonne, 31062 Toulouse Cedex 9
  \newline\indent
  France}

\email[Stefan Le Coz]{slecoz@math.univ-toulouse.fr}


\date{\today}

\begin{abstract}
We study the stability of the cnoidal, dnoidal and snoidal elliptic 
functions as spatially-periodic standing wave solutions of the 1D cubic nonlinear 
Schr\"odinger equations. First, we give global variational characterizations of 
each of these periodic waves, which in particular provide alternate proofs of
their orbital stability with respect to same-period
perturbations, restricted to certain subspaces. 
Second, we prove the spectral stability of the cnoidal waves (in
  a certain parameter range) and snoidal waves against 
same-period perturbations, thus providing 
an alternate proof of this (known) fact, which does
not rely on complete integrability.
Third, we give a rigorous version of a formal asymptotic calculation
of Rowlands 
to establish the instability of a class of 
real-valued periodic waves in 1D, which includes the cnoidal waves of
the 1D cubic focusing nonlinear Schr\"odinger equation, 
against perturbations with period a large multiple of their fundamental period.  
Finally, we develop a numerical method to compute the minimizers of the energy with fixed mass and momentum constraints. Numerical experiments support and complete our analytical results.
\end{abstract}

\maketitle
\tableofcontents


\section{Introduction}

We consider the cubic nonlinear Schr\"odinger equation
\begin{equation} \label{eq:nls}
  i \psi_t + \psi_{xx} + b |\psi|^2 \psi = 0, \qquad \psi(0,x) = \psi_0(x)
\end{equation}
in one space dimension, where $\psi : \R\times\R\to\mathbb C$ and 
$b\in\R\setminus\bket{0}$. Equation~\eqref{eq:nls} has well-known
applications in optics, quantum mechanics, and water waves, and serves
as a model for nonlinear dispersive wave phenomena more
generally \cite{Fi15,SuSu99}. It is said to be
\emph{focusing} if $b>0$ and \emph{defocusing} if $b<0$. Note that~\eqref{eq:nls} is invariant under
\begin{itemize}
\item spatial translation: $\psi(t,x) \mapsto \psi(t,x+a)$ for $a \in \R$
\item phase multiplication: $\psi(t,x) \mapsto e^{i \alpha} \psi(t,x)$ for
$\alpha \in \R$.
\end{itemize}

We are particularly interested in the spatially periodic setting
\[
  \psi(t,\cdot) \in H^1_{\mathrm{loc}}\cap P_T, \qquad
  P_T = \{ f \in L^2_{\mathrm{loc}}(\R): \, f(x+T) = f(x) \; \forall x\in\R \}.
\]
The Cauchy problem~\eqref{eq:nls} is globally well-posed in $H^1_{\mathrm{loc}}
\cap P_T$ \cite{Ca03}.  
We refer to \cite{Bo99} for a detailled analysis of nonlinear Schr\"odinger
equations with periodic boundary conditions. Solutions to~\eqref{eq:nls} conserve mass $\mathcal M$, energy $\mathcal E$,
and momentum $\mathcal P$:
\begin{gather*} 
  \mathcal M(\psi)=\frac12\int_0^T|\psi|^2dx,\quad
  \mathcal P(\psi)=\frac12\Im\int_0^T\psi\bar \psi_xdx,\\
  \mathcal E(\psi)=\frac12\int_0^T|\psi_x|^2dx-\frac
  b4\int_0^T|\psi|^4dx.
\end{gather*}
By virtue of its complete integrability,~\eqref{eq:nls} enjoys infinitely many
higher (in terms of the number of derivatives involved) conservation laws \cite{MaAb81}, but we do not use them here, in order to remain in the energy space
$H^1_{\mathrm{loc}}$, and with the aim of avoiding techniques which rely on integrability.

The simplest non-trivial solutions of~\eqref{eq:nls} are the 
\emph{standing waves},
which have the form
\[
  \psi(t,x) = e^{-i a t} u(x), \qquad  a \in \R
\]   
and so the profile function $u(x)$ must satisfy the ordinary differential equation
\begin{equation} \label{eq:ode}
  u_{xx} + b |u|^2 u + a u = 0.
\end{equation}
We are interested here in those standing waves $e^{-i a t} u(x)$ whose profiles
$u(x)$ are spatially periodic -- which we refer to as \emph{periodic waves}.
One can refer to the book~\cite{Pa09} for an overview of the
role and properties of periodic waves in nonlinear dispersive PDEs.

Non-constant, real-valued, periodic solutions of~\eqref{eq:ode} are well-known
to be given by the Jacobi elliptic functions: dnoidal ($\dn$), cnoidal ($\cn$) 
(for $b > 0$) and snoidal ($\sn$) (for $b < 0$) -- see 
Section~\ref{sec:preliminaries} for details. 
To make the link with Schr\"odinger equations set on the whole
  real line, one can see a periodic wave as a special case of infinite
train solitons \cite{LeLiTs15,LeTs14}. Another context in which
  periodic waves appear is when considering the nonlinear
  Schr\"odinger equation on a Dumbbell graph \cite{MaPe16}.
Our interest here is in the stability of 
these periodic waves against periodic perturbations whose period is a multiple of 
that of the periodic wave.

Some recent progress has been made on this stability question.
By Grillakis-Shatah-Strauss \cite{GrShSt87,GrShSt90} type methods,
orbital stability against energy ($H^1_{\mathrm{loc}}$)-norm perturbations of the same
period is known for dnoidal waves \cite{Pa07}, and for 
snoidal waves \cite{GaHa07-2} under the additional constraint that perturbations 
are anti-symmetric with respect to the half-period. In \cite{GaHa07-2},
cnoidal waves are shown to be orbitally stable with respect to
half-anti-periodic perturbations, provided some condition is
satisfied. This condition is verified analytically for small
amplitude cnoidal waves and numerically for larger amplitude. Remark
here that the results in \cite{GaHa07-2} are obtained in a broader
setting, as they are also considering \emph{non-trivially complex-valued} periodic
waves.
Integrable systems methods introduced in \cite{BoDeNi11} and
 developed in \cite{GalPel15} -- in particular conservation of a higher-order functional -- are used to obtain the 
orbital stability of the snoidal waves against $H^2_{\mathrm{loc}}$ perturbations of
period \emph{any} multiple of that of $\sn$.

Our goal in this paper is to further investigate the
  properties of periodic waves. We follow three lines of
  exploration. First, we give \emph{global} variational
  characterization of the waves in the class of periodic or
  half-anti-periodic functions. As a corollary, we obtain orbital
  stability results for periodic waves.
Second, we prove the spectral
  stability of cnoidal, dnoidal and snoidal waves within the class of
  functions whose period is the fundamental period of the
  wave. Third, we prove that cnoidal waves are linearly unstable if
  perturbations are periodic for a sufficiently large multiple of
  the fundamental period of the cnoidal wave.

Our first main results concern global variational characterizations
of the elliptic function periodic waves as constrained-mass energy minimizers
among (certain subspaces of) periodic functions, stated as a series of Propositions
in Section~\ref{sec:variational}. In particular, the following characterization 
of the cnoidal functions seems new. Roughly stated (see 
Proposition~\ref{prop:focusing-A} for a precise statement):

\begin{theorem}\label{thm:1}
Let $b > 0$. The unique (up to spatial translation and phase
multiplication) global minimizer of the energy, with fixed mass, among 
half-anti-periodic functions is a (appropriately
rescaled) cnoidal function.
\end{theorem}

Due to the periodic setting, existence of a minimizer for the
  problems that we are considering is easily obtained. The difficulty
  lies within the identification of this minimizer: is it  a plane
  wave, a (rescaled) Jacobi elliptic function, or something else? To answer
  this question, we first need to be able to decide whether the minimizer
  can be considered real-valued after a phase change. This is
 far from obvious in the half-anti-periodic setting of 
 Theorem~\ref{thm:1}, where we use
a Fourier coefficients rearrangement argument (Lemma
\ref{lem:stephen_trick}) to obtain this information. To identify the
minimizers, we use a combination of spectral and
Sturm-Liouville arguments. 

As a corollary of our global variational characterizations, we obtain
orbital stability results for the periodic waves. In particular,
Theorem~\ref{thm:1} implies the
orbital stability of all cnoidal waves in the space of
half-anti-periodic functions. Such orbital stability results for
periodic waves were already obtained in \cite{Pa07,GaHa07-2} as
consequences of  \emph{local} constrained minimization 
properties. Our global variational characterizations provide alternate proofs of 
these results -- see Corollary~\ref{cor:orbital} and Corollary~\ref{cor:sn-orbital}. 
The orbital stability of
  cnoidal waves was proved only
  for small amplitude in \cite{GaHa07-2}, and so we extend this result to all
  amplitude. Remark however once more that we are in this paper considering
  only real-valued periodic wave profiles, as opposed to
  \cite{GaHa07-2} in which truly complex valued periodic waves were
  investigated.

Our second main result proves the linear (more precisely, \emph{spectral})
stability of the snoidal and cnoidal (with some restriction on the parameter 
range in the latter case) waves against same-period perturbations, but \emph{without} 
the restriction of half-period antisymmetry:
\begin{theorem}\label{thm:2}
Snoidal waves and cnoidal waves (for a range of parameters) with fundamental period $T$
are spectrally stable against $T$-periodic perturbations.
\end{theorem}
See Theorem~\ref{thm:linear-stability} for a more precise statement.
For $\sn$, this is already a
consequence of~\cite{BoDeNi11,GalPel15}, whereas for $\cn$ the result was
obtained in \cite{IvLa08}. The works \cite{BoDeNi11,GalPel15} and \cite{IvLa08}
both exploit the integrable
structure,  so
our result could be considered an alternate proof which does not uses
integrability, but instead relies mainly on an invariant subspace decomposition
and an elementary Krein-signature-type argument. See also the recent
work \cite{GePe16} for related arguments.

The proof of Theorem \ref{thm:2} goes as follows.
The linearized operator around a periodic wave can be written as
$J\mathcal L$, where $J$ is a skew symmetric matrix and
$\mathcal L$ is the self-adjoint linearization of the action of the wave (see
Section~\ref{sec:stability} for details). The operator $\mathcal L$ is
made of two Lam\'e operators and we are able to calculate the bottom
of the spectrum for these operators. 
To obtain Theorem~\ref{thm:2}, we decompose the space of periodic
functions into invariant subspaces: half-periodic and
half-anti-periodic, even and odd. Then we analyse the linearized
spectrum in each of these subspaces. 
In the subspace of half-anti-periodic functions, we obtain spectral stability as a consequence of the analysis of the spectrum of $\mathcal L$
(alternately, as a consequence of the variational characterizations
of Section~\ref{sec:variational}). For the subspace of
half-periodic functions, a more involved argument is required. We give
in Lemma~\ref{lem:abstract-coercivity} an abstract argument relating coercivity
of the linearized action $\mathcal L$ with the number of eigenvalues
with negative Krein signature of $J\mathcal L$ (this is in fact a
simplified version of a more general argument \cite{HaKa08}).  Since
we are  able to find an eigenvalue with negative
Krein signature for $J\mathcal L$, spectral stability for
half-periodic functions follows from this abstract argument.

Our third main result makes rigorous a formal asymptotic calculation
of Rowlands~\cite{Row74} which establishes:
\begin{theorem}
Cnoidal waves are unstable against perturbations whose period is a 
sufficiently large multiple of its own.
\end{theorem}
This is stated more precisely in Theorem~\ref{thm:instability},
and is a consequence of a more
general perturbation result, Proposition~\ref{prop:perturbation},
which implies this instability for any real periodic wave for which a 
certain quantity has the right sign. In particular, the argument does not rely
on any integrability (beyond the ability to calculate the quantity in question
in terms of elliptic integrals).     

Perturbation argument were also used by \cite{GaHa07-1},
\cite{GalPel15}, but our strategy here is different. Instead of relying
on abstract theory to obtain the a priori existence of branches of
eigenvalues, we directly construct the branch in which we are
interested. This is done by first calculating the exact terms of the formal
expansion for the eigenvalue and eigenvector at the two first orders,
and then obtaining the rigorous existence for the rest of the expansion
using a contraction mapping argument. Note that the branch that we are
constructing was described in terms of Evans function in \cite{IvLa08}.

Finally, we complete our analytical results with some numerical
observations. Our motivation is to complete the variational
characterizations of periodic waves, which was only partial for snoidal
waves. We observe:
\begin{observation}\label{obs:4}
 Let $b<0$. For a given period, the unique (up to phase shift and
 translation) global minimizer of the energy with fixed mass and $0$
 momentum among half-anti-periodic functions is a (appropriately
 rescaled) snoidal function. 
\end{observation}
We have developed a numerical method to obtain the profile $\phi$ as
minimizer on two constraints, fixed mass and fixed (zero) momentum. We
use a heat flow algorithm, where at each time step the solution is
renormalized to satisfy the constraints. Mass renormalization is
simply obtained by scaling. Momentum renormalization is much
trickier. We define an auxiliary evolution problem for the momentum
that we solve explicitly, and plug back the solution we obtain to get
the desired renormalized solutions. 
We first have tested our algorithm in the cases where
our theoretical results hold and we have a good agreement between the
theoretical results and the numerical experiments. Then, we have
performed experiments on snoidal waves which led to Observation~\ref{obs:4}.

The rest of this paper is divided as follows. In Section
 ~\ref{sec:preliminaries}, we present the spaces of periodic
functions and briefly recall the main definitions and
properties of Jacobi elliptic functions and integrals. In Section
\ref{sec:variational}, we characterize the Jacobi elliptic functions
as global constraint minimizers and give the corresponding orbital
stability results. Section~\ref{sec:stability} is devoted to the proof
of spectral stability for cnoidal and snoidal waves, whereas in
Section~\ref{sec:instability} we prove the linear instability of
cnoidal waves. Finally, we present our numerical method in Section
\ref{sec:numerics}  and the numerical experiments in Section~\ref{sec:num-exp}.

\begin{acknowledgments}
We are grateful to Bernard Deconinck and Dmitri Pelinovsky for useful
remarks on a preliminary version of this paper. 
\end{acknowledgments}


\section{Preliminaries}
\label{sec:preliminaries}

This section is devoted to reviewing the classification of real-valued
periodic waves in terms of Jacobi elliptic functions.
 
\subsection{Spaces of Periodic Functions}
Let $T>0$ be a period. Denote by $\tau_T$  the translation operator
\[
  (\tau_T f)(x) = f(x+T),
\]
acting on $L^2_{\mathrm{loc}}(\R)$, and its eigenspaces 
\[
 P_{T}({\mu}) = \{ f \in L^2_{\mathrm{loc}}(\R): \tau_T f =  \mu f  \}
\]
for $\mu \in \mathbb C$ with $|\mu|=1$.
Taking $\mu=1$ yields the space of $T$-periodic functions
\[
  P_T = P_T(1) = \{ f \in L^2_{\mathrm{loc}}(\R): \, \tau_T f = f \},
\]
while for $\mu=-1$ we get the $T$-anti-periodic functions
\[
  A_T = P_T(-1) = \{ f \in L^2_{\mathrm{loc}}(\R): \, \tau_T f = - f\}.
\]
For $2 \le k \in \N$, letting $\mu$ run through the $k$th roots of unity: 
$\omega^k=1$, and  $\omega^j \not =1$ for $1 \le j<k$, we have
\[
  P_{kT} = \bigoplus_{j=0}^{k-1} P_{T}(\omega^j),
\]
where the decomposition of $f \in P_{kT}$ is given by
\[
f = \sum_{j=0}^{k-1} f_j,\quad f_j = \frac 1k \sum_{m=0}^{k-1} \omega^{-mj} \tau_{mT} f.
\]
Only the case $k=2$ is needed here:
\begin{equation} \label{eq:direct1}
  P_{2T} = P_{T} \oplus A_{T}, \qquad
  f = \frac 12 (f+\tau_T f) + \frac 12 (f-\tau_T f).
\end{equation}
Since the reflection $R : f(x) \mapsto f(-x)$
commutes with $\tau_{T}$ on $P_{2T}$,
we may further decompose into odd and even components in the usual way
\[
  f = f^+ + f^-, \quad
  f^{\pm} = \frac12 (f \pm Rf),
\]
to obtain
\[
  P_T = P_T^+ \oplus P_T^-, \;\; A_T = A_T^+ \oplus A_T^-, \quad
  P_T^{\pm} \; (A_T^{\pm} \;) 
  = \{ f \in P_T \; (A_T \; ) \; | \; f(-x) = \pm f(x) \}, 
\]  
and so
\begin{equation} \label{eq:direct2}
  P_{2T} = P_{T} \oplus A_{T} = P_T^+ \oplus P_T^- \oplus A_T^+ \oplus A_T^-.
\end{equation}
Each of these subspaces is invariant under~\eqref{eq:nls}, since
\[
  \psi \in P_T^{\pm} \;( A_T^{\pm} ) \; \implies \; |\psi|^2 \in P_T^+
  \; \implies \;
  \psi_{xx} + b |\psi|^2 \psi \in P_T^{\pm} \;( A_T^{\pm} ).
\]
When dealing with functions in $P_T$, we will denote norms such as 
$L^q(0,T)$ by
\[
  \norm{u}_{L^q} = \norm{u}_{L^q(0,T)} = \left( \int_0^T |u|^q \right)^{\frac{1}{q}},
\]
 and the \emph{complex} $L^2$ inner product by
  \begin{equation}\label{eq:inner-complex}
\ps{f}{g} =  \int _0^T f \bar g\, dx.
\end{equation}

\subsection{Jacobi Elliptic Functions}
Here we recall the definitions and main properties of the Jacobi elliptic functions. 
The reader might refer to treatises on elliptic functions
(e.g. \cite{La89}) or to the classical handbooks \cite{AbSt64,GrRy15} for more details.

Given $k\in(0,1)$, the \emph{incomplete elliptic integral of the first kind} 
in trigonometric form is 
\[
x = F(\phi,k):=\int_0^\phi \frac
{d \theta}{\sqrt{1-k^2 \sin^2(\theta)}},
\]
and the \emph{Jacobi elliptic functions} are defined through the inverse of $F(\cdot,k)$:
\[
\sn(x,k):= \sin (\phi), \quad \cn(x,k):=\cos (\phi), \quad
\dn (x,k):=\sqrt{1-k^2\sin^2 (\phi)}.
\]
The relations
\begin{equation}
\label{sn.sn}
1=\sn^2+\cn^2 = k^2 \sn^2 + \dn^2
\end{equation}
follow.  For extreme value $k=0$ we recover trigonometric functions,
\[
  \sn(x,0)=\sin (x), \quad \cn(x,0)=\cos( x), \quad \dn(x,0) = 1,
\]
while for extreme value $k=1$ we recover hyperbolic functions:
\[ 
  \sn(x,1)=\tanh (x), \quad \cn(x,1)=\dn(x,1)=\sech (x).
\]
The periods of the elliptic functions can be expressed in terms of the 
\emph{complete elliptic integral of the first kind}
\[
  K(k):=F\left(\frac\pi2,k\right), \quad K(k) \to \left\{ \begin{array}{cc}
  \frac{\pi}{2} & k \to 0 \\ \infty & k \to 1 \end{array} \right..  
\]
The functions $\sn$ and $\cn$ are $4K$-periodic while $\dn$
is $2K$-periodic. More precisely,
\[
  \dn\in P_{2K}^+, \quad \sn \in A_{2K}^- \subset P_{4K},
  \quad \cn\in A_{2K}^+ \subset P_{4K}.
\]

The derivatives (with respect to $x$) of elliptic functions can themselves
be expressed in terms of elliptic functions. For fixed $k\in(0,1)$, we have
\begin{equation}
  \label{dxsn}
\partial_x\sn = \cn \cdot \dn, \quad \partial_x\cn= -\sn\cdot \dn ,\quad \partial_x\dn= - k^2 \cn \cdot \sn,
\end{equation}
from which one can easily verify that $\sn$, $\cn$ and $\dn$
are solutions of 
\begin{equation}
  \label{eq:snls-1}
u_{xx}+a u+b|u|^2u=0,
\end{equation}
with coefficients $a,b\in\R$ for $k\in(0,1)$ given by
\begin{align}
  a&=1+k^2,& b&=-2k^2,&&\text{for }u=\sn,\label{eq:a-b-sn}\\
  a&=1-2k^2,& b&=2k^2,&&\text{for }u=\cn,\label{eq:a-b-cn}\\
  a&=-(2-k^2),& b&=2,&&\text{for }u=\dn. \label{eq:a-b-dn}
\end{align}


\subsection{Elliptic Integrals}
\label{ssec:elliptic}

For $k\in(0,1)$, the \emph{incomplete elliptic integral of the second
kind} in trigonometric form is defined by 
\[
E(\phi,k):=\int_0^\phi \sqrt{1-k^2 \sin^2(\theta)}
d \theta.
\]
The \emph{complete elliptic integral of the second kind} is defined as
\[
 E(k):=E\left(\frac\pi2,k\right) .
\]
We have the relations
(using $d \theta = \dn(z,k) dz$ and $x=F(\phi,k)$)
\begin{multline}
\label{eq:2.2}
E(\phi,k)
=\int_0^x\dn^2(z,k)dz
\\=x-k^2\int_0^x\sn^2(z,k)dz=(1-k^2)x+k^2\int_0^x\cn^2(z,k)dz,
\end{multline}
relating the elliptic functions to the elliptic integral of the second kind, and
\begin{equation}
\label{eq:2.1}
E(k) = K(k) - k^2 \int_0^K \sn^2(z,k)dz = (1-k^2)K(k) + k^2 \int_0^K \cn^2(z,k)dz,
\end{equation}
relating the elliptic integrals of first and second kind.
We can differentiate $E$ and $K$ with respect to $k$ and
express the derivatives in terms of $E$ and $K$:
\begin{align*}
\partial_k E(k)&=\frac{E(k)-K(k)}{k}<0,
\\
\partial_kK(k)&=\frac{E(k)-(1-k^2) K(k)}{k-k^3}
=\frac{k^2
\int_0^K\cn^2(x,k)dx}{k-k^3}>0.
\end{align*}
Note in particular  $K$ is increasing, $E$ is
decreasing. Moreover, 
\[
K(0)=E(0)=\frac\pi2,
\quad K(1-)=\infty, \quad E(1)=1.
\]

\subsection{Classification of Real Periodic Waves}
Here we make precise the fact that the elliptic functions provide the only (non-constant) 
real-valued, periodic solutions of~\eqref{eq:snls-1}.
Note that there is a two-parameter family of complex-valued, bounded, solutions 
for every $a,b\in \R$, $b \not = 0$ \cite{Ga94,GaHa07-1}.

\begin{lemma}[focusing case]\label{th:2-1}
Fix a period $T>0$, $a \in \R$, $b>0$ and $u \in P_T$ a non-constant real solution of 
\eqref{eq:snls-1}. By invariance under translation, and negation ($u \mapsto -u$),
we may suppose $u(0)=\max u>0$. 

\textup{(a)} If $0 \le \min u < u(0)$, then $a<0$, $|a|< b u(0)^2 < 2|a|$, and $u(x) = \frac 1\alpha \dn(\frac x \beta, k)$,

\textup{(b)} If $\min u<0$, then $\max(0,-2a)< b u(0)^2$, and  $u(x) = \frac 1\alpha \cn(\frac x \beta, k)$,

\noindent
for some $\alpha>0$, $\beta>0$, and $0<k<1$, uniquely determined by $T$, $a$, $b$ and $\max u$.
%
%
They satisfy the $a$-independent relations $b\beta ^2 =2 \alpha^2$ for \textup{(a)} and
 $b\beta^2 =2 k^2 \alpha^2  $ for \textup{(b)}.
\end{lemma}

Note that here $T$ may be any multiple of the fundamental period of $u$.
An $a$-independent relation is useful since $a$ will be the unknown Lagrange multiplier for our constrained minimization problems in 
 Section~\ref{sec:variational}.
 
\begin{proof} The first integral is constant: there exists $C_0\in \R$ such that
\[
u_x^2 + au^2 + \frac b2 u^4 = C_0.
\]
A periodic solution has to oscillate in the energy well $W(u)=au^2 + \frac b2 u^4$ with energy level  $C_0$.
If $0\le \min u$, then $a<0$ and $C_0<0$. If $\min u<0$, then $C_0>0$.
Let $u(x)=\frac 1\alpha v(\frac x \beta)$ with $\alpha = (\max u)^{-1}$. Then $v$ satisfies
\[
\max v = v(0)=1, \quad v'' + a \beta^2 v + \frac {b\beta^2 }{\alpha^2} v^3=0.
\]

(a) If $0\le \min u$, then $a<0$ and $C_0<0$. Let $0<y_1<y_2$ be the roots of $ ay + \frac b2 y^2 = C_0<0$. Then $u(0)^2 = y_2 \in (-a/b, -2a/b)$.

Let $\beta = \alpha \sqrt{2/b}$. Then $\frac {b\beta^2 }{\alpha^2} =2$ and $a \beta^2 \in (-2,-1)$, and there is a unique $k\in (0,1)$ so that $a \beta^2 = -2+k^2$. Thus 
\[
\max v = v(0)=1, \quad v'(0)=0, \quad v'' + (-2+k^2) v+ 2v^3=0.
\]
By uniqueness of the ODE, $v(x) = \dn(x,k)$ is the only solution. Hence $u(x) = \frac 1\alpha \dn(\frac x \beta, k)$.

(b) If $\min u<0$, then $C_0>0$. Let $y_1<0<y_2$ be the roots of $ ay + \frac b2 y^2 = C_0>0$. Then $u(0)^2 = y_2 > \max(0,  -2a/b)$ no matter $a<0$ or $a \ge 0$.
We claim we can choose unique $\beta>0$ and $k\in (0,1)$ so that
\[
a \beta^2 = 1-2k^2, \quad  \frac {b\beta^2 }{\alpha^2}  = 2k^2.
\]
The sum gives $(a+  \frac { b}{\alpha^2})\beta^2=1$, thus $\beta=(a+  \frac { b}{\alpha^2})^{-1/2}$ noting  $(a+  \frac { b}{\alpha^2})>0$, and
\[
k^2 = \frac {\beta^2 b}{2 \alpha^2} = \frac { b}{2(b+a \alpha^2)} \in (0,1)
\]
no matter $a<0$ or $a \ge 0$. 
Thus 
\[
\max v = v(0)=1, \quad v'(0)=0, \quad v'' + ( 1-2k^2) v+ 2k^2 v^3=0.
\]
By uniqueness of the ODE, $v(x) = \cn(x,k)$ is the only solution. Hence $u(x) = \frac 1\alpha \cn(\frac x \beta, k)$.
\end{proof}

\begin{lemma}[defocusing case]\label{th:2-2}
Fix a period $T>0$, $a \in \R$, $b<0$ and $u \in P_T$ a non-constant, real solution of 
\eqref{eq:snls-1}. By invariance under translation and negation, suppose $u(0)=\max u>0$. 
Then $0< |b| u(0)^2 < a$, and $u(x) = \frac 1\alpha \sn(K(k)+\frac x \beta, k)$,
for some $\alpha>0$, $\beta>0$, and $0<k<1$, uniquely determined by $T$, $a$, $b$ and $\max u$.
They satisfy the $a$-independent relation  $ b\beta^2 = -2 k^2 \alpha^2$.
\end{lemma}

\begin{proof} The first integral is constant: there exists $C_0\in \R$ such that
\[
u_x^2 + au^2 + \frac b2 u^4 = C_0.
\]
A periodic solution has to oscillate in the energy well $W(u)=au^2 + \frac b2 u^4$ with energy level  $C_0$.
Hence  $a>0$ and $0<C_0<\max W = \frac {a^2}{-2b}$. 
Let $u(x)=\frac 1\alpha v(\frac x \beta)$ with $\alpha = (\max u)^{-1}$. Then $v$ satisfies
\[
\max v = v(0)=1, \quad v'' + a \beta^2 v + \frac {b\beta^2 }{\alpha^2} v^3=0.
\]

Let $0<y_1<y_2$ be the roots of $ ay + \frac b2 y^2 = C_0$. Then $u(0)^2 = y_1 \in (0,-a/b)$.

Let $\beta = (\frac {2 \alpha^2} {2\alpha^2 a +b})^{1/2}$ and $k = (\frac{-b}{2\alpha^2 a +b})^{1/2}$,  noting ${2\alpha^2 a +b}>0$. Then $a\beta^2=1+k^2$, $\frac {b\beta^2 }{\alpha^2} = -2 k^2$, and $v$ satisfies 
\[
\max v = v(0)=1, \quad v'(0)=0, \quad v'' + (1+k^2) v  -2 k^2v^3=0.
\]
By uniqueness of the ODE, $v(x) = \sn(K(k)+x,k)$ is the only solution. Hence $u(x) = \frac 1\alpha \sn(K(k)+\frac x \beta, k)$.
\end{proof}


\section{Variational Characterizations and Orbital Stability}
\label{sec:variational}

Our goal in this section is to characterize the Jacobi elliptic functions as \emph{global} 
constrained energy minimizers. As a corollary, we recover some known results on orbital
stability, which is closely related to \emph{local} variational information.  

\subsection{The Minimization Problems}

Recall the basic conserved functionals for~\eqref{eq:nls} on $H^1_{\mathrm{loc}} \cap P_T$:
\begin{gather*}
  \mathcal M(u)=\frac12\int_0^T|u|^2dx,\quad
  \mathcal P(u)=\frac12\Im\int_0^Tu\bar u_xdx,\\
  \mathcal E(u)=\frac12\int_0^T|u_x|^2dx-\frac
  b4\int_0^T|u|^4dx.
\end{gather*}
In this section, we consider $L^2(0,T;\mathbb
  C)$ as a \emph{real} Hilbert space with scalar product
  $\Re\int_0^T f\bar gdx$. This way, the functionals $\mathcal E$,
  $\mathcal M$ and $\mathcal P$
  are $C^1$ functionals. This also ensures that the Lagrange
  multipliers are real. 
  Note that we see $L^2(0,T;\mathbb
  C)$ as a real Hilbert space only in the current section and in all
  the other sections it will be seen as a complex Hilbert space with the
  scalar product defined in \eqref{eq:inner-complex}.

Fix parameters $T>0$, $a,b\in \R$, $b \not = 0$. 
Since the Jacobi elliptic functions (indeed any standing wave profiles) 
are solutions of~\eqref{eq:snls-1}, they are critical points 
of the \emph{action} functional  $\mathcal S_a$ defined by
\[
\mathcal S_a(u)=\mathcal E(u)-a\mathcal M(u),
\]
where the values of $a$ and $b$ are given in
\eqref{eq:a-b-sn}-\eqref{eq:a-b-dn} and the fundamental periods
are $T=2K$ for $\dn$, $T=4K$ for $\sn,\cn$.
Given $m>0$, the basic variational problem is to minimize
the energy with fixed mass:
\begin{equation}
\label{eq:min-prob-m}
\min
\left\{
\mathcal E(u) \; | \; \mathcal M(u)=m
,u\in H_{\mathrm{loc}}^1\cap P_T
\right\},
\end{equation}
whose Euler-Lagrange equation 
\begin{equation} \label{eq:EL}
  u''+ b|u|^2u + {a} u = 0,
\end{equation}
with ${a} \in \R$ arising as Lagrange multiplier, is
indeed of the form~\eqref{eq:snls-1}.
Since the momentum is also conserved for~\eqref{eq:nls}, it is natural
to consider the problem with a further momentum constraint:
\begin{equation}
\label{eq:min-prob-m-p}
\min
\left\{
\mathcal E(u) \; | \; \mathcal M(u)=m,
\mathcal P(u)=0,u\in H_{\mathrm{loc}}^1\cap P_T
\right\}.
\end{equation}
\begin{remark}
Note that if a minimizer $u$ of~\eqref{eq:min-prob-m} is such that $P(u)=0$,
then it is real-valued (up to multiplication by a complex number of modulus $1$). 
Indeed, it verifies~\eqref{eq:EL} for some ${a}\in\R$.
It is well known (see e.g. \cite{GaHa07-2}) that the momentum density
$\Im(u_x \bar u )$ is therefore constant in $x$, and so it is identically $ 0$  
if $P(u)=0$. For $u(x)\neq 0$ we can
write $u$ as $u=\rho e^{i\theta}$, and express the momentum density as
$\Im(u_x\bar u)=\theta_x\rho^2$. Thus $\Im(u_x\bar u )=0$
implies $\theta_x=0$ and thus $\theta(x)$ is constant  as long as $u(x) \not =0$. If $u(x_0)=0$
and $ e^{\theta(x_0-)} \not = e^{\theta(x_0+)}$, 
we must have $u_x(x_0)=0$, and hence $u \equiv 0$ by uniqueness of the ODE.
\end{remark}
Since~\eqref{eq:nls} preserves the subspaces in the decomposition~\eqref{eq:direct1},
it is also natural to consider variational problems restricted to 
anti-symmetric functions,
\begin{gather}
\label{eq:min-prob-m-A}
\min
\left\{
\mathcal E(u) \; | \; \mathcal M(u)=m
,u\in H_{\mathrm{loc}}^1\cap A_{T/2}
\right\}, \\
\label{eq:min-prob-m-p-A}
\min
\left\{
\mathcal E(u) \; | \; \mathcal M(u)=m,\mathcal P(u)=0,u\in H_{\mathrm{loc}}^1\cap A_{T/2}
\right\},
\end{gather}
and in light of the decomposition~\eqref{eq:direct2}, further
restrictions to even or odd functions may also be considered.

In general, the difficulty does not lie in proving the existence of a minimizer,
but rather in identifying this minimizer with an elliptic function, since we are minimizing among \emph{complex valued} functions, and moreover restrictions to
symmetry subspaces prevent us from using classical variational methods 
like symmetric rearrangements.

We will first consider the minimization problems
\eqref{eq:min-prob-m} and~\eqref{eq:min-prob-m-p}  for periodic
functions in $P_{T}$. Then we will 
consider the minimization problems
\eqref{eq:min-prob-m-A} and~\eqref{eq:min-prob-m-p-A} for
half-anti-periodic functions in $A_{T/2}$. In both parts, we will treat separately the
focusing ($b>0$) and defocusing ($b<0$) nonlinearities. 
For each case, we will show the existence of a unique
(up to phase shift and translation)
minimizer, and we will identify it with either a plane wave or a
Jacobi elliptic function.

\subsection{Minimization Among Periodic Functions}

\subsubsection{The Focusing Case in $P_T$}

\begin{proposition}\label{prop:focusing-P_T}
  Assume $b>0$.
  The minimization problems~\eqref{eq:min-prob-m}
  and~\eqref{eq:min-prob-m-p}
  satisfy the following properties.
  \begin{itemize}
    \item[(i)] For all $m>0$,~\eqref{eq:min-prob-m}  and
     ~\eqref{eq:min-prob-m-p} share the same minimizers. The minimal energy is finite and negative.
  \item[(ii)]   For all $0<m\leq \frac{\pi^2}{bT}$ there exists a unique (up to
    phase shift) minimizer
  of~\eqref{eq:min-prob-m}. It is the constant function $u_{\min}\equiv \sqrt{\frac{2m}{T}}$.
  \item[(iii)] For all $\frac{\pi^2}{bT}<m<\infty$ there exists a unique (up
    to translations and phase shift) minimizer of
   ~\eqref{eq:min-prob-m}. It is the
    rescaled function
    $\dn_{\alpha,\beta,k}=\frac1\alpha\dn\left(\frac\cdot\beta,k\right)$
    where the parameters $\alpha$, $\beta$ and $k$ are uniquely
    determined. Its fundamental period is $T$. The map from $m \in (\frac{\pi^2}{bT},\infty)$ to $k \in (0,1)$ is one-to-one, onto and increasing.   
\item[(iv)] In particular, given $k\in(0,1)$, $\dn=\dn(\cdot,k)$, if $b=2$,
  $T=2K(k)$, and
  $m=\mathcal M(\dn)=E(k)$,  then the unique (up to translations and phase shift) minimizer of
 ~\eqref{eq:min-prob-m} is $\dn$.
  \end{itemize}
\end{proposition}

\begin{proof}
  Without loss of generality, we can restrict the minimization to
  real-valued non-negative functions. Indeed, if $u\in H^1_{\mathrm{loc}}\cap P_T$, then
  $|u|\in  H^1_{\mathrm{loc}}\cap P_T$ and we have 
\[
  \norm{\partial_x|u|}_{L^2}\leq \norm{\partial_xu}_{L^2}.
\]
This readily implies that~\eqref{eq:min-prob-m}  and
\eqref{eq:min-prob-m-p} share the same minimizers. 
Let us prove that 
  \begin{equation}\label{eq:4}
  -\infty<
  \min
\left\{
\mathcal E(u) \; | \; \mathcal M(u)=m, u\in H^1_{\mathrm{loc}}\cap P_T
\right\}<0.
  \end{equation}
 The last inequality in~\eqref{eq:4} is obtained using the constant
 function $\varphi_{m,0}\equiv \sqrt{\frac{2m}{T}}$ as a test function:
\[ 
\mathcal E(\varphi_{m,0})<0,\quad\mathcal M(\varphi_{m,0})=m.
\]
To prove the first inequality  in~\eqref{eq:4}, we observe that by
Gagliardo-Nirenberg inequality we have 
\[
\norm{u}_{L^4}^4\lesssim \norm{u}_{L^2}^3\norm{u_x}_{L^2} + \norm{u}_{L^2}^4.
\]
Consequently, for $u\in H^1_{\mathrm{loc}}\cap P_T$ such that $\mathcal M(u)=m$, we have
\[
\mathcal E(u)\gtrsim \norm{u_x}_{L^2}\left( \norm{u_x}_{L^2}-m^{3/2}\right) - m^2,
\]
and $\mathcal E$ has to be bounded from below. The above shows (i).

Consider now a minimizing sequence
  $(u_n)\subset H^1_{\mathrm{loc}}\cap P_T$ for~\eqref{eq:min-prob-m}. It is bounded in $H^1_{\mathrm{loc}}\cap P_T$ and
  therefore, up to a subsequence, it converges weakly in $H^1_{\mathrm{loc}}\cap P_T$ and
  strongly in $L^2_{\mathrm{loc}}\cap P_T$ and $L^4_{\mathrm{loc}}\cap P_T$ towards $u_\infty\in H^1_{\mathrm{loc}}\cap P_T$. Therefore
  $\mathcal E(u_\infty)\leq \mathcal E(u_n)$ and $\mathcal M(u_\infty)=m$. This implies that
  $\norm{\partial_xu_\infty}_{L^2}=\lim_{n\to\infty }\norm{\partial_x
    u_n}_{L^2}$ and therefore the convergence from $u_n$ to $u_\infty$ is
  also strong in $H^1_{\mathrm{loc}}\cap P_T$. Since $u_\infty$ is a minimizer of
 ~\eqref{eq:min-prob-m}, there exists a Lagrange multiplier ${a}\in\R$
  such that
  \[
 -
\mathcal E'(u_\infty)+{a} \mathcal M'(u_\infty)=0,
\]
that is 
\begin{equation*}
\partial_{xx}u_\infty+b u_\infty^3+{a} u_\infty=0.
\end{equation*} 
Multiplying by $u_\infty$ and integrating, we find that
\[
{a}=\frac{\norm{\partial_xu_\infty}_{L^2}^2-b\norm{u_\infty}_{L^4}^4}{\norm{u_\infty}_{L^2}^2}.
\] 
Note that
\[
\norm{\partial_xu_\infty}_{L^2}^2-b\norm{u_\infty}_{L^4}^4=2\mathcal E(u_\infty)-\frac
b2 \norm{u_\infty}_{L^4}^4<0,
\]
therefore 
\[
{a}<0.
\]

We already have $u_\infty \in \R$, and we may assume $\max u = u(0)$ by translation.
By Lemma~\ref{th:2-1} (a),  either $u_\infty$ is constant or there exist
$\alpha,\beta\in(0,\infty)$ and $ k\in(0,1)$ such that $\beta = \alpha \sqrt{2/b}$ and
\[
u_\infty(x)=\dn_{\alpha,\beta,k}(x)=\frac{1}{\alpha}\dn\left(\frac{x}{\beta},k\right).
\]

We now show that the minimizer $u_\infty$ is of the form $\dn_{\alpha,\beta,k}$ if $m>\frac{\pi^2}{bT}$.
Indeed,
assuming by contradiction that $u_\infty$ is a constant,  we necessarily have
$u_\infty\equiv\sqrt{\frac {2m}{T}}$. The Lagrange multiplier can
also be computed and we find ${a}=-bu_\infty^2=-\frac{2bm}{T}$. Since $u_\infty$ is supposed to be a constrained minimizer for
\eqref{eq:min-prob-m}, the operator 
\[
-\partial_{xx}-{a}-3bu_\infty^2=-\partial_{xx}-\frac{4bm}{T}
\]
must have Morse index at
most $1$, i.e. at most $1$ negative eigenvalue. The eigenvalues are
given for $n\in\mathbb Z$ by the formula 
\[
\left(\frac{2\pi n}{T}\right)^2-\frac{4bm}{T}.
\]
Obviously $n=0$ gives a negative eigenvalue. For $n=1$, the eigenvalue
is non-negative if and only if 
\[
m\leq \frac{\pi^2}{bT},
\]
which gives the contradiction. Hence when $m>\frac{\pi^2}{bT}$ the
minimizer $u_\infty$ must be of the form
$\dn_{\alpha,\beta,k}$.

There is a positive integer $n$ so that the fundamental period of
$u_\infty=\dn_{\alpha,\beta,k}$ is $ {2K(k)}{\beta} = Tn^{-1 }$.
As already
mentioned, since $u_\infty$ is a minimizer for
\eqref{eq:min-prob-m}, the operator 
\[
-\partial_{xx}- {a} - 3b u_\infty ^2 
\]
can have at most one negative eigenvalue. The function
$\partial_x u_\infty$ is in its kernel and has $2n$
zeros. By Sturm-Liouville
theory (see e.g. \cite{Ea73,ReSi78}) we have at least $2n-1$ eigenvalues below $0$. Hence $n=1$ and $ {2K(k)}{\beta} =T$.

Using $2 \alpha^2 = b\beta ^2$ (see Lemma~\ref{th:2-1}), the mass
verifies, 
\[
m=\frac12\int_0^T
|\dn_{\alpha,\beta,k}(x)|^2dx=\frac{\beta}{\alpha^2}\frac12\int_0^{2K(k)}
|\dn(y,k)|^2dy=\frac{2}{b\beta}E(k)
\]
where $E(k)$ is given in Section~\ref{ssec:elliptic}. 
Using $ {2K(k)}{\beta} =T$,
\begin{equation}\label{eq:mk}
m = \frac 4{bT} E(k) K(k).
\end{equation}
Note
\[
\frac{\partial}{\partial k}  EK(k)=\frac{  E(k)^2-(1-k^2)K(k)^2}{(1-k^2)k}>0,
\]
where the positivity of the numerator is because it vanishes at $k=0$ and 
\[
\frac{\partial}{\partial k} (E^2-(1-k^2)K^2 )= \frac 2k (E-K)^2, \quad (0<k<1).
\]
Thus
 $ EK(k)$ varies from $\frac{\pi^2}{4}$ to $\infty$
when $k$ varies from $0$ to $1$. Thus~\eqref{eq:mk} defines $m$ as a strictly increasing function of $k\in (0,1)$ with range $(\frac {\pi^2} {bT},\infty)$ and hence has an inverse function.  For fixed $b,m,T$, the value $k \in (0,1)$ is uniquely determined by ~\eqref{eq:mk}.
We also have $\beta = \frac T{2K(k)}$ and $\alpha = \beta \sqrt {b/2}$.
The above shows (iii). 

The above calculation also shows that $m > \frac{\pi^2} {bT}$ if $u_\infty = \dn_{\alpha,\beta,k}$. Thus $u_\infty$ must be a constant when $0 < m \le 
\frac{\pi^2} {bT}$.
This shows (ii).

In the case we are given  $k\in(0,1)$, 
$T=2K(k)$, $b=2$ and $m=\mathcal M(\dn)=E(k)$, we want to show that $u_\infty (x)= \dn(x,k)$. In this case 
$m>\frac {\pi^2}{bT}$ since $EK > \frac {\pi^2}{4}$. Thus, by Lemma~\ref{th:2-1} (a),  $u_\infty = \dn_{\alpha,\beta,s}$ for some $\alpha,\beta>0$ and $s \in (0,1)$, up to translation and phase.
By the same Sturm-Liouville
theory argument, 
the fundamental period of $u_\infty$ is $T= {2K(s)}{\beta} $. The same calculation leading to~\eqref{eq:mk} shows
\[
m = \frac 4{bT} E(s) K(s).
\]
Thus $E(k)K(k)=E(s) K(s)$.
Using the monotonicity of $EK(k)$ in $k$, we have $k=s$.  Thus $\alpha=\beta=1$ and  $u_\infty (x)= \dn(x,k)$. 
This gives (iv) and finishes the proof.
\end{proof}

\subsubsection{The Defocusing Case in $P_T$}

\begin{proposition}\label{prop:defocusing-P}
Assume $b<0$. For all $0<m<\infty$, the constrained minimization problems   
\eqref{eq:min-prob-m} and 
~\eqref{eq:min-prob-m-p} have the same unique  (up to phase
shift) minimizers, which is the constant function
$u_{\min}\equiv \sqrt{\frac{2m}{T}}$.
\end{proposition}

\begin{proof}
  This is a simple consequence of the fact that functions with constant modulus are
  the optimizers of the injection $L^4(0,T)\hookrightarrow
  L^2(0,T)$. More precisely,  for every $f\in L^4(0,T)$ we have by H\"older's inequality,
  \[
  \norm{f}_{L^2}\leq T^{1/4}\norm{f}_{L^4},
  \]
  with equality if and only if $|f|$ is constant. Let $\varphi_{m,0}$
  be the constant function $\varphi_{m,0}\equiv \sqrt{\frac{2m}{T}}$.
  For any $v\in H^1_{\mathrm{loc}}\cap P_T$ such that $\mathcal M(v)=m$ and
  $v\not\equiv e^{i\theta}\varphi_{m,0}$ ($\theta\in\R$) we have 
  \begin{gather*}
    0=\norm{\partial_x \varphi_{m,0}}_{L^2}^2< \norm{\partial_x v}_{L^2}^2,\\
    \norm{\varphi_{m,0}}_{L^4}^4=4T^{-1}\mathcal M^2(\varphi_{m,0})=4T^{-1}\mathcal M^2(v)\leq \norm{v}_{L^4}^4.
  \end{gather*}
  As a consequence, $\mathcal E(\varphi_{m,0})<\mathcal E(v)$ and this proves the proposition.
\end{proof}

\subsection{Minimization Among Half-Anti-Periodic Functions}

\subsubsection{The Focusing Case in $A_{T/2}$}

\begin{proposition}
\label{prop:focusing-A}
  Assume $b>0$.
  For all $m>0$,  the minimization problems~\eqref{eq:min-prob-m-A}
  and~\eqref{eq:min-prob-m-p-A} in $A_{T/2}$
  satisfy the following properties.
  \begin{itemize}
    \item[(i)] The minimizers for~\eqref{eq:min-prob-m-A}  and
     ~\eqref{eq:min-prob-m-p-A} are the same.  
  \item[(ii)]   There exists a unique (up to translations and phase shift)
    minimizer of~\eqref{eq:min-prob-m-A}. It is the
    rescaled function
    $\cn_{\alpha,\beta,k}=\frac1\alpha\cn\left(\frac\cdot\beta,k\right)$
    where the parameters $\alpha$, $\beta$ and $k$ are uniquely
    determined. Its fundamental period is $T$. The map from $m \in (0,\infty)$ to $k \in (0,1)$ is one-to-one, onto and increasing. 
\item[(iii)] In particular, given $k\in(0,1)$, $\cn=\cn(\cdot,k)$, if $b=2k^2$,
  $T=4K(k)$, and
  $m=\mathcal M(\cn)=2(E-(1-k^2)K)/k^2$,  then the unique (up to translations and  phase shift) minimizer of
 ~\eqref{eq:min-prob-m-A} is $\cn$.
  \end{itemize}
\end{proposition}

Before proving Proposition~\ref{prop:focusing-A}, we make the
following crucial observation.

\begin{lemma}\label{lem:stephen_trick}
  Let $v\in H^1_{\mathrm{loc}}\cap A_{T/2}$. Then there exists $\tilde v\in H^1_{\mathrm{loc}}\cap A_{T/2}$
  such that 
\[
\tilde v(x) \in\R,\quad \norm{\tilde v}_{L^2}= \norm{v}_{L^2}, \quad
\norm{\partial_x\tilde v}_{L^2}= \norm{\partial_x v}_{L^2},\quad
\norm{\tilde v}_{L^4}\geq  \norm{v}_{L^4}.
\]
\end{lemma}

\begin{proof}[Proof of Lemma~\ref{lem:stephen_trick}]
The proof relies on a combinatorial argument. 
Since $v \in H^1_{\mathrm{loc}} \cap A_{T/2}$, its Fourier series expansion contains 
only terms indexed by \emph{odd} integers:
\[
v(x)=\sum_{\substack{j\in\mathbb
      Z\\j\text{ odd }}} v_je^{ij\frac{2\pi}{T}x}.
\]
We define $\tilde v$ by its Fourier series expansion
\[
\tilde v(x)=\sum_{\substack{j\in\mathbb
      Z\\j\text{ odd }}} \tilde v_je^{ij\frac{2\pi}{T}x},\quad \tilde v_j:=\sqrt{\frac{|v_j|^2+|v_{-j}|^2}{2}}.
\]
It is clear that $\tilde{v}(x) \in \R$ for all $x\in\R$, and by
Plancherel formula,
\[
\norm{\tilde v}_{L^2}= \norm{v}_{L^2}, \quad
\norm{\partial_x\tilde v}_{L^2}= \norm{\partial_x v}_{L^2},
\]
so all we have to prove is that 
$\norm{\tilde v}_{L^4}\geq  \norm{v}_{L^4}$.
We have 
\[
|v(x)|^2=\sum_{\substack{j\in \mathbb
      Z\\j\text{ odd }}} |v_j|^2+\sum_{\substack{n\in2\mathbb
      N\\n\geq 2}} w_ne^{in\frac{2\pi}{T}x}+\bar w_ne^{-in\frac{2\pi}{T}x},
\]
where we have defined 
\[
w_n=\sum_{\substack{j>k,
      j+k=n\\j,k\text{ odd}}} v_j\bar v_{-k}+v_{k}\bar v_{-j}.
\]
Using the fact that for $n\in\mathbb N$, $n\neq 0$, the term
$e^{in\frac{2\pi}{T}x}$ integrates to $0$ due to periodicity,
\[
\int_{0}^T e^{in\frac{2\pi}{T}x}dx=0,
\]
we compute 
\[
  \frac{1}{T} \int_0^T|v|^4dx=\bigg(\sum_{\substack{j\in \mathbb
      Z\\j\text{ odd }}} |v_j|^2\bigg)^2+2\sum_{\substack{n\in2\mathbb
      N\\n\geq 2}} |w_n|^2.
\]
The first part is just 
\[
\bigg(\sum_{\substack{j\in \mathbb
      Z\\j\text{ odd }}}
  |v_j|^2\bigg)^2= \frac{1}{T^2} \norm{v}_{L^2}^4 = \frac{1}{T^2} \norm{\tilde v}_{L^2}^4.
\]
For the second part, we observe that
\begin{equation}
\label{eq:w_n}
w_n= \sum_{\substack{j>k,
      j+k=n\\j,k\text{ odd}}} \binom{v_j}{ \bar
  v_{-j}}\cdot\binom{v_{-k}}{\bar v_k}, 
\end{equation}
where the $\cdot$ denotes the complex vector scalar
product. Therefore, 
\[
\begin{split}
|w_n| &\leq \sum_{\substack{j>k,
      j+k=n\\j,k\text{ odd}}}  \left|\binom{v_j}{ \bar
  v_{-j}}\right|\left|\binom{v_{-k}}{\bar v_k}\right|
= \sum_{\substack{j>k,
      j+k=n\\j,k\text{ odd}}}  \sqrt{2 \tilde v_j^2}\sqrt{2 \tilde
    v_k^2} \\
&=2 \sum_{\substack{j>k,
      j+k=n\\j,k\text{ odd}}} \tilde v_j \tilde v_k=\tilde w_n, 
\end{split}
\]
where by $\tilde w_n$ we denote the quantity defined similarly as in
\eqref{eq:w_n} for $(\tilde v_j)$. As a consequence, 
\[
\norm{v}_{L^4}\leq \norm{\tilde v}_{L^4}
\]
and this finishes the proof of Lemma~\ref{lem:stephen_trick}.
\end{proof}

\begin{proof}[Proof of Proposition~\ref{prop:focusing-A}]
All functions are considered in $A_{T/2}$.
Consider a minimizing sequence $(u_n)$ for
\eqref{eq:min-prob-m-p-A}. By Lemma~\ref{lem:stephen_trick}, the
minimizing sequence can be chosen such that $u_n(x)\in\R$ for all
$x\in\R$ and this readily implies the equivalence between
\eqref{eq:min-prob-m-p-A} and~\eqref{eq:min-prob-m-A}, which is (i). 

Using the same arguments as in the proof of Proposition~\ref{prop:focusing-P_T}, we
infer that the minimizing sequence converges strongly in
$H^1_{\mathrm{loc}} \cap A_{T/2}$ to $u_\infty\in H^1_{\mathrm{loc}}\cap A_{T/2}$ verifying
for some ${a}\in\R$ the Euler-Lagrange equation
\[
\partial_{xx}u_\infty+bu_{\infty}^3+{a} u_{\infty}=0.
\]
Then,  since $u_\infty$ is real and in $A_{T/2}$, we may assume $\max u = u(0)>0$ and, by Lemma~\ref{th:2-1} (b), there exists a set of parameters
$\alpha,\beta\in(0,\infty)$, $ k\in(0,1)$ such that 
\[
u_\infty(x)=\frac1\alpha\cn\left(\frac x\beta, k\right),
\]
and the parameters $\alpha,\beta,k $  are determined by $T$, ${a}$, $b$ and $\max u$, with 
$2 k^2 \alpha^2 = b\beta^2$.

There exists an odd, positive integer $n$ so that the fundamental period of 
$u_\infty$ is $ {4K(k)}{\beta} = T/n$.
Since $u_\infty$ is a minimizer for~\eqref{eq:min-prob-m-A}, the operator 
\[
-\partial_{xx}- {a} - 3b u_\infty ^2 
\]
can have at most one negative eigenvalue in $L^2_{\mathrm{loc}} \cap A_{T/2}$. The function
$\partial_x u_\infty$ is in its kernel and has $2n$
zeros in $[0,T)$. By Sturm-Liouville theory, there are at least $n-1$ eigenvalues (with
eigenfunctions in $A_{T/2}$) below $0$. 
Hence, since $n$ is odd, $n=1$ and $ {4K(k)}{\beta} =T$.

The mass verifies, using $2 k^2 \alpha^2 = b\beta^2$ and ~\eqref{eq:2.1},
\[
m=\frac12\int_0^T
|\cn_{\alpha,\beta,k}(x)|^2dx=\frac{\beta}{\alpha^2}\frac12\int_0^{4K(k)}
|\cn(y,k)|^2dy = \frac 4{\beta b} (E(k) - (1-k^2)K(k)).
\]
Using $ {4K(k)}{\beta} =T$,
\begin{equation}\label{eq:mk-A}
m = M(k):=\frac {16}{bT} K(k)  (E(k) - (1-k^2)K(k)).
\end{equation}
Note all factors of $M(k)$ are positive, $\frac{\partial}{\partial k} K(k)>0$ and
\[
\frac{\partial}{\partial k} (E-(1-k^2)K )= \frac {E-K}k - \frac {E-(1-k^2)K}k + 2kK = kK >0.
\]
Thus~\eqref{eq:mk-A} defines $m$ as a strictly increasing function of $k\in (0,1)$ with range $(0,\infty)$ and hence has an inverse function.  For fixed $T,b,m$, the value $k \in (0,1)$ is uniquely determined by ~\eqref{eq:mk-A}.
We also have $\beta=\frac T{4K(k)}$ and $\alpha^2 = \frac {b\beta^2 }{2 k^2}$.
The above shows (ii).

In the case we are given  $k\in(0,1)$, 
$T=4K(k)$, $b=2 k^2$ and $m=\mathcal M(\cn(\cdot,k))$, we want to show that $u_\infty (x)= \cn(x,k)$. In this case, by Lemma~\ref{th:2-1} (b),
 $u_\infty = \cn_{\alpha,\beta,s}$ for some $\alpha,\beta>0$ and $s \in (0,1)$, up to translation and phase.
By the same Sturm-Liouville
theory argument, 
the fundamental period of $u_\infty$ is $T= {4K(s)}{\beta} $. The same calculation leading to~\eqref{eq:mk-A} shows
\[
m = M(s).
\]
Thus $M(s) = M(k)$. By the monotonicity of $M(k)$ in $k$, we have $k=s$. Thus $\alpha=\beta=1$ and  $u_\infty (x)= \cn(x,k)$. 
This shows (iii) and 
concludes the proof.
\end{proof}

\subsubsection{The Defocusing Case in $A_{T/2}$}

\begin{proposition}\label{prop:defocusing-A}
Assume $b<0$.
There exists a unique (up to phase shift and complex conjugate) minimizer
for ~\eqref{eq:min-prob-m-A}. It 
is the plane wave $u_{\min}\equiv \sqrt{\frac{2m}{T}} e^{\frac{2i\pi x}{T}}$.
\end{proposition}

\begin{proof}
  Denote the supposed minimizer by
  $w(x)=\sqrt{\frac{2m}{T}} e^{\pm\frac{2i\pi x}{T}}$. Let  $v\in
  H^1_{\mathrm{loc}}\cap A_{2K}$ such that $\mathcal M(v)=m$ and
  $v\not\equiv e^{i\theta}w$ ($\theta\in\R$).
  As in the proof of Proposition~\ref{prop:defocusing-P}, we have
  \[
    \norm{w}_{L^4}^4=4T^{-1}\mathcal M^2(w)=4T^{-1}\mathcal M^2(v)\leq \norm{v}_{L^4}^4.
    \]
   Since $v\in A_{2K}$, $v$ must have $0$ mean value. Recall that in
   that case $v$ verifies the
   Poincar\'e-Wirtinger inequality
   \[
   \norm{v}_{L^2}\leq \frac {T}{2\pi} \norm{v'}_{L^2},
   \]
   and that the optimizers of the Poincar\'e-Wirtinger inequality are
   of the form $C e^{\pm\frac{2i\pi}{T}x}$, $C\in\mathbb C$.
   This implies that
  \[
    \norm{\partial_xw}_{L^2}^2=\frac{8\pi^2}{T^2}\mathcal M(w)=\frac{8\pi^2}{T^2}\mathcal M(v)< \norm{\partial_xv}_{L^2}^2.
    \]
As a consequence, $\mathcal E(w)<\mathcal E(v)$ and this proves the lemma.
\end{proof}

As far as~\eqref{eq:min-prob-m-p-A} is concerned, we make the following conjecture

\begin{conjecture}\label{conj:minimizer-2}
Assume $b<0$.
The unique (up to translations and phase shift) minimizer of ~\eqref{eq:min-prob-m-p-A}
is the rescaled function
$\sn_{\alpha,\beta,k}=\frac1\alpha\sn\left(\frac\cdot\beta,k\right)$
where the parameters $\alpha$, $\beta$ and $k$ are uniquely determined.

In particular, given $k\in(0,1)$, $\sn=\sn(\cdot,k)$, if $b=-2k^2$,
  $T=4K(k)$, and
  $m=\mathcal M(\sn)$,  then the unique (up translations and to phase shift) minimizer of
 ~\eqref{eq:min-prob-m-p-A} is $\sn$.
\end{conjecture}

This conjecture is supported by numerical evidence, see Observation
\ref{obs:conjecture}.
The main difficulty in proving the conjecture is to show that the minimizer is real up to a phase.

\subsubsection{The Defocusing Case in $A_{T/2}^-$}

In light of our uncertainty about whether $\sn$ solves~\eqref{eq:min-prob-m-p-A},
let us settle for the simple observation that it is the energy minimizer among
\emph{odd}, half-anti-periodic functions: 
\begin{proposition}
\label{prop:defocusing-A-} 
Assume $b<0$. The unique (up to phase shift) minimizer of the problem
\begin{equation} \label{eq:min-prob-m-A-}
\min
\left\{
\mathcal E(u) \; | \; \mathcal M(u)=m ,u\in H_{\mathrm{loc}}^1\cap A_{T/2}^-
\right\}, 
\end{equation}
is the rescaled function
$\sn_{\alpha,\beta,k}=\frac1\alpha\sn\left(\frac\cdot\beta,k\right)$
where the parameters $\alpha$, $\beta$ and $k$ are uniquely determined.
Its fundamental period is $T$. The map from $m \in (0,\infty)$
  to $k \in (0,1)$ is one-to-one, onto and increasing. 

In particular, given $k\in(0,1)$, $\sn=\sn(\cdot,k)$, if $b=-2k^2$,
  $T=4K(k)$, and
  $m=\mathcal M(\sn)$,  then the unique (up to phase shift) minimizer of
 ~\eqref{eq:min-prob-m-A-} is $\sn$.
\end{proposition}
\begin{proof}
If $u \in A_{T/2}^-$, then $0 = u(0) = u(T/2)$, and since $u$ is completely
determined by its values on $[0,T/2]$, we may replace~\eqref{eq:min-prob-m-A-} by
\[
\min \left\{
\mathcal \int_0^{T/2}  \left( |u_x|^2 - \frac{b}{2} |u|^4 \right) dx 
\; \big| \; \int_0^{T/2} |u(x)|^2 dx = m, \; u\in H^1_0([0,T]) \right\},
\]
for which the map $u \mapsto |u|$ is admissible, showing that minimizers
are non-negative (up to phase), and in particular real-valued, hence 
a (rescaled) $\sn$ function by Lemma~\ref{th:2-2}. The remaining statements follow as in the proof
of Proposition~\ref{prop:focusing-A}. 
In particular,
the mass verifies, using $2 k^2 \alpha^2 = |b|\beta^2$,~\eqref{eq:2.1}, and $ {4K(k)}{\beta} =T$,
\[
\begin{split}
m &=\frac12\int_0^T
|\sn_{\alpha,\beta,k}(x)|^2dx=\frac{\beta}{\alpha^2}\frac12\int_0^{4K(k)}
|\sn(y,k)|^2dy 
\\
&= \frac 4{\beta |b|} (K(k) - E(k))
=\frac {16}{|b|T} K(k)  (K(k) - E(k)),
\end{split}
\]
which is 
a strictly increasing function of $k\in (0,1)$ with range $(0,\infty)$ and hence has an inverse function.
\end{proof}

\subsection{Orbital Stability}
\label{sec:orbital-stability}

Recall that we say that a standing wave $\psi(t,x) = e^{-ia t} u(x)$ is
orbitally stable for the flow of~\eqref{eq:nls} in the function space
$X$ if for all $\eps>0$ there exists $\delta>0$ such that the
following holds: if $\psi_0\in X$ verifies
\[
\norm{\psi_0-u}_X\leq \delta
\]
then the solution $\psi$ of~\eqref{eq:nls} with initial data $\psi(0,x)=\psi_0$
verifies for all $t\in\R$ the estimate
\[
\inf_{\theta\in\R,y\in\R}\norm{\psi(t,\cdot)-e^{i\theta}u(\cdot-y)}_X<\eps.
\]
As an immediate corollary of the variational characterizations above, we have the
following orbital stability statements:
\begin{corollary}\label{cor:orbital}
The standing wave $\psi(t,x)=e^{-iat}u(x)$ is a solution of~\eqref{eq:nls}, 
and is orbitally stable in $X$ in the following cases.   
For Jacobi elliptic functions: for any $k\in(0,1)$, 
\begin{align*}
  a&=1+k^2,  &b& =-2k^2, & u&=\sn(\cdot,k), &X&=H^1_{\mathrm{loc}}\cap A_{2K}^-;\\ 
  a&=1-2k^2, &b&=2k^2, & u&=\cn(\cdot,k),&X&=H^1_{\mathrm{loc}}\cap A_{2K};\\ 
  a&=-(2-k^2), &b&=2, & u&=\dn(\cdot,k),&X&=H^1_{\mathrm{loc}}\cap P_{2K}.\\
\intertext{For constants and plane waves: $(b \not =0)$}
a&=-\frac{2bm}{T}, &-\infty<b&\leq\frac{\pi^2}{Tm},
  &u&=\sqrt{\frac{2m}{T}},&X&=H^1_{\mathrm{loc}}\cap P_{T};\\ 
a&= \frac{4\pi^2}{T^2} -\frac{2bm}{T}, &b&<0,
  &u&=e^{\pm\frac{2i\pi x}{T}}\sqrt{\frac{2m}{T}},&X&=H^1_{\mathrm{loc}}\cap A_{T/2}.
\end{align*}
\end{corollary}

The proof of this corollary uses the variational characterizations from
Propositions~\ref{prop:focusing-P_T},~\ref{prop:defocusing-P},
\ref{prop:focusing-A},~\ref{prop:defocusing-A}, and~\ref{prop:defocusing-A-}. 
Note that for all the minimization problems considered we have the compactness of minimizing sequences. The proof follows the standard line introduced by Cazenave
and Lions \cite{CaLi82}, we omit the details here.

\begin{remark}
The orbital stability of $\sn$ \cite{GaHa07-2} in 
$H^1_{\mathrm{loc}} \cap A_{T/2}$ was proved using the Grillakis-Shatah-Strauss \cite{GrShSt87,GrShSt90}
approach, which amounts to identifying the periodic wave as a \emph{local} constrained
minimizer in this subspace. So the above may be considered an alternate proof, using
\emph{global} variational information. In the case of $\sn$, without 
Conjecture~\ref{conj:minimizer-2}, some additional spectral information in the 
subspace $A_{T/2}^+$ is needed to obtain orbital stability in 
$H^1_{\mathrm{loc}} \cap A_{T/2}$ (rather than just $H^1_{\mathrm{loc}} \cap A_{T/2}^-$)
-- see Corollary~\ref{cor:sn-orbital} in the next section for this.

Orbital stability of $\cn$ was obtained in \cite{GaHa07-2} only for small
amplitude $\cn$. We extend this result to all possible values of
$k\in(0,1)$. 
\end{remark}
\begin{remark}
Using the complete integrability of~\eqref{eq:nls}, Bottman, Deconinck and Nivala \cite{BoDeNi11}, and Gallay and Pelinovsky 
\cite{GalPel15}
showed that $\sn$ is in fact a minimizer of a higher-order functional
in $H^2_{\mathrm{loc}} \cap P_{nT}$ for any $n \in \N$, and thus showed it is
orbitally stable in these spaces.
\end{remark}


\section{Spectral Stability}
\label{sec:stability}

Given a standing wave $\psi(t,x) = e^{-iat} u(x)$ 
solution of~\eqref{eq:nls},
we consider the linearization of~\eqref{eq:nls} around this solution:
if $\psi(t,x)=e^{-iat}(u(x)+h)$, then $h$ verifies
\[
i\partial_th - Lh + N(h) = 0,
\]
where $L$ denotes the linear part and $N$ the nonlinear part. Assuming
$u$ is real-valued, we separate $h$ into real and imaginary parts to get the equation
\[
\partial_t\binom{\Re(h)}{\Im(h)}=J\mathcal L \binom{\Re(h)}{\Im(h)}+\binom{-\Im(N(h))}{\Re(N(h))},
\]
where 
\[
\mathcal L=
\begin{pmatrix}
L_+&0\\
0&L_-
\end{pmatrix},
\quad
J=
\begin{pmatrix}
0&1\\
-1&0
\end{pmatrix},
\quad
\begin{aligned}
L_+&=-\partial_{xx}-a-3b \,u^2,\\
L_-&=-\partial_{xx}-a-b \, u^2.
\end{aligned}
\]
We call
\begin{equation} \label{eq:lin-form}
  J \mathcal L = \begin{pmatrix} 0 & L_- \\ -L_+ & 0 \end{pmatrix}
\end{equation}
the \emph{linearized operator} of~\eqref{eq:nls} about the standing wave
$e^{-i a t} u(x)$.

Now suppose $u \in H^1_{\mathrm{loc}} \cap P_T$ is a (period $T$) periodic wave, and 
consider its linearized operator $J \mathcal L$ as an operator
on the Hilbert space $( P_T)^2$, with domain $(H^2_{\mathrm{loc}} \cap P_T)^2$.
The main structural properties of $J \mathcal L$ are:
\begin{itemize}
\item
since $L_{\pm}$ are self-adjoint operators on $ P_T$,
$\mathcal L$ is self-adjoint on $(P_T)^2$, 
while $J$ is skew-adjoint and unitary
\begin{equation} \label{eq:sa}
  \mathcal L^* = \mathcal L, \quad J^* = -J = J^{-1},
\end{equation} 
\item
$J \mathcal L$ commutes with complex conjugation,
\begin{equation} \label{eq:real}
  \overline{ J \mathcal L \; f} = J \mathcal L \bar{f},  
\end{equation}
\item
$J \mathcal L$ is antisymmetric under conjugation by the matrix
\[
  C = \begin{pmatrix} 1 & 0 \\ 0 & -1 \end{pmatrix} 
\]
(which corresponds to the operation of complex conjugation
\emph{before} complexification), 
\begin{equation} \label{eq:conjugate} 
  J \mathcal L C = - C J \mathcal L.
\end{equation}
\end{itemize}

At the \emph{linear} level, the stability of the periodic wave is determined by 
the location of the spectrum $\sigma(J \mathcal L)$, which in this periodic 
setting consists of isolated eigenvalues of finite multiplicity \cite{ReSi78}.
We first make the standard observation that as a result 
of~\eqref{eq:real} and~\eqref{eq:conjugate}, the spectrum of $J \mathcal L$ 
is invariant under reflection about the real and imaginary axes:
\[
  \lambda \in \sigma(J \mathcal L) \; \implies \; 
  \pm \lambda, \; \pm \bar{\lambda} \in \sigma(J \mathcal L).
\]
Indeed, if $J \mathcal L f = \lambda f$, then 
\begin{gather*}
 ~\eqref{eq:real} \implies J \mathcal L \bar{f} = \bar{\lambda} \bar{f}, 
  \quad\eqref{eq:conjugate} \implies J \mathcal L C f = -\lambda Cf, \\
 ~\eqref{eq:real} \mbox{ and }\eqref{eq:conjugate} \implies J \mathcal L C \bar{f}
  = -\bar{\lambda} C \bar{f}.
\end{gather*}

We are interested in whether the entire spectrum of 
$J \mathcal L$ lies on the imaginary axis, denoted 
$\sigma(J \mathcal L |_{P_T}) \subset i \R$,
in which case we say the periodic wave $u$ is \emph{spectrally stable in $P_T$}. 
Moreover, if $S \subset P_T$ is an invariant subspace -- more precisely,
$J \mathcal L : (H^2_{\mathrm{loc}} \cap S)^2 \to  (S)^2$ --
then we will say that the periodic wave $u$ is \emph{spectrally stable in $S$}
if the entire $( S)^2$ spectrum of $J \mathcal L$ lies on 
the imaginary axis, denoted $\sigma( J \mathcal L |_{S} ) \subset i \R$.
In particular, for $k \in (0,1)$ and $K = K(k)$,
since $\sn^2$, $\cn^2$, $\dn^2$ $\in P_{2K}^+$, the 
corresponding linearized operators respect the decomposition~\eqref{eq:direct2}, 
and we may consider $\sigma( J \mathcal L |_S)$ for 
$S = P_{2K}^{\pm}, A_{2K}^{\pm} \subset P_{4K}$, with
\begin{equation} \label{eq:spec-decomp}
\begin{split}
\sigma\left(J\mathcal L |_{P_{4K}}\right) &= 
\sigma\left(J\mathcal L |_{P_{2K}}\right) \cup 
\sigma\left(J\mathcal L |_{A_{2K}}\right) \\ &= 
\sigma\left(J\mathcal L |_{P_{2K}^+}\right) \cup 
\sigma\left(J\mathcal L |_{P_{2K}^-}\right)  \cup 
\sigma\left(J\mathcal L |_{A_{2K}^+}\right) \cup 
\sigma\left(J\mathcal L |_{A_{2K}^-}\right).
\end{split}
\end{equation}

Of course, spectral stability (which is purely linear) is a weaker notion than orbital 
stability (which is nonlinear). Indeed, the latter implies the former -- see
Proposition~\ref{prop:orbital-spectral} and the remarks preceding it.

The main result of this section is the following.
\begin{theorem}\label{thm:linear-stability}
Spectral stability in $P_T$, $T = 4K(k)$, holds for:
\begin{itemize}
\item $u=\sn$, $k \in (0,1)$,
\item $u=\cn$ and $k\in(0,k_c)$, where $k_c$ is the unique $k\in (0,1)$ so that $K(k)=2E(k)$, $k_c\approx 0.908$.
\end{itemize}
\end{theorem}

\begin{remark}
 The function $f(k)=K(k)-2E(k)$ is strictly increasing in $k\in
 (0,1)$, (since $K(k)$ is increasing while $E(k)$ is decreasing in
 $k$), with $f(0)=-\frac \pi2 $ and $f(1)=\infty$.
\end{remark}

\begin{remark}
Using Evans function techniques, it was proved in \cite{IvLa08} that $\sigma(J\mathcal L^{\cn}) \subset i\R$ also
for $k\in[k_c,1)$. This fact is also supported by
numerical evidence (see Section~\ref{sec:num-exp}).
\end{remark}

\begin{remark}
In the case of $\sn$, the $H^2_{\mathrm{loc}} \cap P_{nT}$ orbital stability
obtained in \cite{BoDeNi11,GalPel15} (using integrability) immediately
implies spectral stability in $P_{nT}$, and in particular in $P_T$.
So our result for $\sn$ could be considered an alternate, elementary proof,
not relying on the integrability. 
\end{remark}

\begin{remark}
The spectral stability of $\dn$ in $P_{2K}$ (its own fundamental period)
is an immediate consequence of its orbital stability in $H^1_{\mathrm{loc}} \cap P_{2K}$,
see Proposition~\ref{prop:orbital-spectral}.
\end{remark}

\subsection{Spectra of $L_+$ and $L_-$}

We assume now that we are given $k\in(0,1)$ and we describe the
spectrum of $L_+$ and $L_-$ in $ P_{4K}$ when
$\phi$ is $\cn$, $\dn$ or $\sn$. When $\phi=\sn$, we denote $L_+$ by
$L_+^{\sn}$, and we use similar notations for $L_-$ and $\cn, \dn$. Due
to the algebraic relationships between $\cn$, $\dn$ and $\sn$, we
have
\begin{align*}
  L_+^{\sn}&=-\partial_{xx}-(1+k^2)+6k^2\sn^2,\\
  L_+^{\cn}&=-\partial_{xx}-(1-2k^2)-6k^2\cn^2=L_+^{\sn}-3k^2,\\
  L_+^{\dn}&=-\partial_{xx}+(2-k^2)-6\dn^2=L_+^{\sn}-3.
\end{align*}
Similarly for $L_-$, we obtain
\begin{align*}
  L_-^{\sn}&=-\partial_{xx}-(1+k^2)+2k^2\sn^2,\\
  L_-^{\cn}&=-\partial_{xx}-(1-2k^2)-2k^2\cn^2=L_-^{\sn}+k^2,\\
  L_-^{\dn}&=-\partial_{xx}+(2-k^2)-2\dn^2=L_-^{\sn}+1.
\end{align*}
As a consequence, $L_\pm^{\sn}$, $L_\pm^{\cn}$, and $L_\pm^{\dn}$ share
the same eigenvectors. Moreover, these operators enter in the framework
of Schr\"odinger operators with periodic potentials and much can
be said about their spectrum (see e.g. \cite{Ea73,ReSi78}). Recall in
particular that given a Schr\"odinger operator $L=-\partial_{xx}+V$
with periodic potential $V$ of period $T$, the eigenvalues $\lambda_n$
of $L$ on $P_T$ satisfy
\[
\lambda_0<\lambda_1\leq \lambda_2<\lambda_3\leq \lambda_4<\cdots,
\]
with corresponding eigenfunctions $\psi_n$ such that
$\psi_0$ has no zeros, $\psi_{2m+1}$ and $\psi_{2m+2}$ have exactly
$2m+2$ zeros in $[0,T)$ (\cite[p. 39]{Ea73}). From the equations
satisfied by $\cn$, $\dn$, $\sn$, we directly infer that
\begin{align*}
L_-^{\sn}\dn&=-\dn,
& L_-^{\sn}\cn&=-k^2\cn,
& L_-^{\sn}\sn&=0.
\end{align*}
Taking the derivative with respect to $x$ of the equations
satisfied by $\cn$, $\dn$, $\sn$, we obtain
\begin{align*}
 L_+^{\sn}\partial_x\sn&=0,
& L_+^{\sn}\partial_x\cn&=3k^2\partial_x\cn,
& L_+^{\sn}\partial_x\dn&=3\partial_x\dn.
\end{align*}
Looking for eigenfunctions in the form $\chi=1-A\sn^2$ for $A\in\R$, we
find two other eigenfunctions:
\begin{align*}
L_+^{\sn}\chi_-&=e_-\chi_-,
& L_+^{\sn}\chi_+&=e_+\chi_+,
\end{align*}
where 
\begin{gather*}
\chi_\pm=1-\left(k^2+1\pm\sqrt{k^4-k^2+1}\right)\sn^2,\\
\pm e_\pm=\pm\left(k^2+1\pm 2\sqrt{k^4-k^2+1}\right)>0.
\end{gather*}
In the interval $[0,4K)$, $\chi_-$ has no zero, $\sn_x$ and $\cn_x$ have two zeros 
each, while $\dn_x$ and $\chi_+$ have 4 zeros each. By Sturm-Liouville theory,
they are the first 5 eigenvectors of $L_+$ for each of $\sn$, $\cn$, and $\dn$, 
and all other eigenfunctions have strictly greater eigenvalues.
Similarly, $\dn > 0$ has no zeros, while $\cn$ and $\sn$ have two each, so
these are the first $3$ eigenfunctions of $L_-$ for each of $\sn$, $\cn$, and $\dn$, 
and all other eigenfunctions have strictly greater eigenvalues.

The spectra of $L_\pm^{\sn}$, $L_\pm^{\cn}$, and $L_\pm^{\dn}$ are
represented in Figure~\ref{fig:spectra}, where the eigenfunctions
are also classified with respect to the subspaces of decomposition 
\eqref{eq:direct2}.


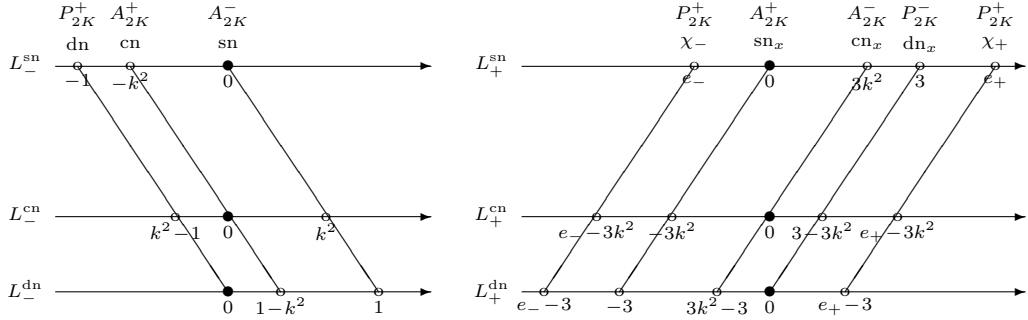
\begin{figure}[htpb!]
\setlength{\unitlength}{1mm}
\centering
\begin{picture}(133,44)(0,0)

\put (4,35){\vector(1,0){50}}
\put (4,15){\vector(1,0){50}}
\put (4,5){\vector(1,0){50}}
\put (66,35){\vector(1,0){67}}
\put (66,15){\vector(1,0){67}}
\put (66,5){\vector(1,0){67}}

\put (0,35){\makebox(0,0)[c]{\scriptsize $L_-^{\sn}$}}
\put (0,15){\makebox(0,0)[c]{\scriptsize $L_-^{\cn}$}}
\put (0,5){\makebox(0,0)[c]{\scriptsize $L_-^{\dn}$}}
\put (62,35){\makebox(0,0)[c]{\scriptsize $L_+^{\sn}$}}
\put (62,15){\makebox(0,0)[c]{\scriptsize $L_+^{\cn}$}}
\put (62,5){\makebox(0,0)[c]{\scriptsize $L_+^{\dn}$}}

\put(7,35) {\line(2,-3){20}} 
\put(14,35) {\line(2,-3){20}}
\put(27,35) {\line(2,-3){20}}
\put(89,35) {\line(-2,-3){20}}
\put(99,35) {\line(-2,-3){20}}
\put(112,35) {\line(-2,-3){20}}
\put(119,35) {\line(-2,-3){20}}
\put(129,35) {\line(-2,-3){20}}

\put(7,38) {\makebox(0,0)[c]{\scriptsize $\dn$}}
\put(14,38) {\makebox(0,0)[c]{\scriptsize $\cn$}}
\put(27,38) {\makebox(0,0)[c]{\scriptsize $\sn$}}
\put(89,38)  {\makebox(0,0)[c]{\scriptsize $\chi_-$}}
\put(99,38)  {\makebox(0,0)[c]{\scriptsize $\sn_x$}}
\put(112,38)  {\makebox(0,0)[c]{\scriptsize $\cn_x$}}
\put(119,38) {\makebox(0,0)[c]{\scriptsize $\dn_x$}}
\put(129,38)  {\makebox(0,0)[c]{\scriptsize $\chi_+$}}

\put(7,42) {\makebox(0,0)[c]{\scriptsize $P_{2K}^+$}}
\put(14,42) {\makebox(0,0)[c]{\scriptsize $A_{2K}^+$}}
\put(27,42) {\makebox(0,0)[c]{\scriptsize $A_{2K}^-$}}
\put(89,42)  {\makebox(0,0)[c]{\scriptsize $P_{2K}^+$}}
\put(99,42)  {\makebox(0,0)[c]{\scriptsize $A_{2K}^+$}}
\put(112,42)  {\makebox(0,0)[c]{\scriptsize $A_{2K}^-$}}
\put(119,42) {\makebox(0,0)[c]{\scriptsize $P_{2K}^-$}}
\put(129,42)  {\makebox(0,0)[c]{\scriptsize $P_{2K}^+$}}

\put(7,33) {\makebox(0,0)[c]{\scriptsize $-1$}}
\put(14,33) {\makebox(0,0)[c]{\scriptsize $-k^2$}}
\put(27,33) {\makebox(0,0)[c]{\scriptsize $0$}}
\put(89,33)  {\makebox(0,0)[c]{\scriptsize $e_-$}}
\put(99,33)  {\makebox(0,0)[c]{\scriptsize $0$}}
\put(112,33)  {\makebox(0,0)[c]{\scriptsize $3k^2$}}
\put(119,33) {\makebox(0,0)[c]{\scriptsize $3$}}
\put(129,33) {\makebox(0,0)[c]{\scriptsize $e_+$}}

\put(20,13) {\makebox(0,0)[c]{\scriptsize $k^2\!-\!1$}}
\put(27,13) {\makebox(0,0)[c]{\scriptsize $0$}}
\put(40,13) {\makebox(0,0)[c]{\scriptsize $k^2$}}
\put(76,13) {\makebox(0,0)[c]{\scriptsize $e_- \!-\!3k^2$}}
\put(86,13)  {\makebox(0,0)[c]{\scriptsize $-3k^2$}}
\put(99,13)  {\makebox(0,0)[c]{\scriptsize $0$}}
\put(106,13)  {\makebox(0,0)[c]{\scriptsize $3\!-\!3k^2$}}
\put(116,13)  {\makebox(0,0)[c]{\scriptsize $e_+\!-\!3k^2$}}

\put(27,3) {\makebox(0,0)[c]{\scriptsize $0$}}
\put(34,3) {\makebox(0,0)[c]{\scriptsize $1\!-\!k^2$}}
\put(47,3) {\makebox(0,0)[c]{\scriptsize $1$}}
\put(69,3)  {\makebox(0,0)[c]{\scriptsize $e_- \!-\!3$}}
\put(79,3) {\makebox(0,0)[c]{\scriptsize $-3$}}
\put(92,3)  {\makebox(0,0)[c]{\scriptsize $3k^2\!-\!3$}}
\put(99,3)  {\makebox(0,0)[c]{\scriptsize $0$}}
\put(109,3)  {\makebox(0,0)[c]{\scriptsize $e_+\!-\!3$}}

\put(7,35) {\makebox(0,0)[c]{\scriptsize o}}
\put(14,35) {\makebox(0,0)[c]{\scriptsize o}}
\put(27,35) {\makebox(0,0)[c]{$\bullet$}}
\put(89,35)  {\makebox(0,0)[c]{\scriptsize o}}
\put(99,35)  {\makebox(0,0)[c]{$\bullet$}}
\put(112,35)  {\makebox(0,0)[c]{\scriptsize o}}
\put(119,35) {\makebox(0,0)[c]{\scriptsize o}}
\put(129,35) {\makebox(0,0)[c]{\scriptsize o}}

\put(20,15) {\makebox(0,0)[c]{\scriptsize o}}
\put(27,15) {\makebox(0,0)[c]{$\bullet$}}
\put(40,15) {\makebox(0,0)[c]{\scriptsize o}}
\put(76,15) {\makebox(0,0)[c]{\scriptsize o}}
\put(86,15)  {\makebox(0,0)[c]{\scriptsize o}}
\put(99,15)  {\makebox(0,0)[c]{$\bullet$}}
\put(106,15)  {\makebox(0,0)[c]{\scriptsize o}}
\put(116,15)  {\makebox(0,0)[c]{\scriptsize o}}

\put(27,5) {\makebox(0,0)[c]{$\bullet$}}
\put(34,5) {\makebox(0,0)[c]{\scriptsize o}}
\put(47,5) {\makebox(0,0)[c]{\scriptsize o}}
\put(69,5)  {\makebox(0,0)[c]{\scriptsize o}}
\put(79,5) {\makebox(0,0)[c]{\scriptsize o}}
\put(92,5)  {\makebox(0,0)[c]{\scriptsize o}}
\put(99,5)  {\makebox(0,0)[c]{$\bullet$}}
\put(109,5)  {\makebox(0,0)[c]{\scriptsize o}}

\end{picture}
\setlength{\unitlength}{1pt}
\caption{Eigenvalues for $L_-$ and $L_+$ in $P_{4K}$. 
}
\label{fig:spectra}

\end{figure}

We may now recover the result of \cite{GaHa07-2} that $\sn$ is orbitally stable in 
$H^1_{\mathrm{loc}} \cap A_{2K}$, using the following simple consequences of the spectral
information above:
\begin{lemma} \label{lemma:sn-coercivity}
There exists $\delta>0$ such that
the following coercivity properties hold.
\begin{enumerate}
\item
$L_+^{\sn} |_{A_{2K}^-} > \delta$,
\item
$L_-^{\sn} |_{A_{2K}^- \cap \{ \sn \}^{\perp}} > \delta$,
\item
$L_+^{\sn} |_{A_{2K}^+ \cap \{ (\sn)_x \}^{\perp}} > \delta$,
\item
$L_-^{\sn} |_{A_{2K}^+ \cap \{ (\sn)_x \}^{\perp}} > \delta$.
\end{enumerate}
\end{lemma}

\begin{proof}
The first three are immediate from figure~\ref{fig:spectra}
(note the first two also follow from the minimization property
Proposition~\ref{prop:defocusing-A-}), while we see that in $A_{2K}^+$, 
$L_+^{\sn} |_{ \{ (\sn)_x \}^\perp } > e_+$, so since $\sn^2(x) \leq 1$,
\[
  L_-^{\sn} |_{ \{ (\sn)_x \}^\perp }  =  
  \left( L_+^{\sn} - 4k^2 \sn^2 \right) |_{ \{ (\sn)_x \}^\perp }
  > e_+ - 4 k^2 > 0
\]
where the last inequality is easily verified.
\end{proof}

\begin{corollary} \label{cor:sn-orbital}
For all $k\in(0,1)$, the standing wave $\psi(t,x)=e^{-i(1+k^2)t}\sn(x,k)$ is orbitally stable in $H^1_{\mathrm{loc}} \cap A_{2K}$.
\end{corollary}

\begin{proof}
Lemma~\ref{lemma:sn-coercivity} shows that $\sn$ is a non-degenerate
(up to phase and translation) \emph{local} minimizer of the energy with fixed 
mass and momentum.  So the classical 
Cazenave-Lions \cite{CaLi82} argument yields the orbital stability.
\end{proof}

Finally, we also record here the following computations concerning
$L_{\pm}^{\cn}$, used in analyzing the
generalized kernel of $J \mathcal L^{\cn}$ in the next subsection:

\begin{lemma}\label{lem:explicit}
Define $\hat E(x,k)=E(\phi,k)|_{\sin \phi = \sn(x,k)}$.
Let $\phi_1$ and $\xi_1$ be given by the following expressions.
\begin{align*}
\phi_1&=\frac{\left(\hat E(x,k)-\frac{
    E}{K}x\right)\cn_x-k^2\cn^3+{\frac {K{k}^{2}-
        E}{K}}\cn}{2(2k^2-1)\frac {
    E}{K}+2(1-k^2)},\\
\xi_1&=\frac{\left(\hat E(x,k)-\frac{
    E}{K}x\right)\cn+\cn_x }{-2(1-k^2) +\frac{2
    E}{K}}.
\end{align*}
The denominators are positive and we have
\[
L_+^{\cn}\phi_1=\cn,\qquad L_-^{\cn} \xi_1=\cn_x.
\]
Note $\hat E$ and $\xi_1$ are odd while $\phi_1$ is even. In particular $(\phi_1,\cn_x)=0= (\xi_1, \cn)$.
Moreover, $L_+^{\cn} (\frac12\cn-(1-2k^2)\phi_1)=\cn_{xx}$.
\end{lemma}

\begin{proof}
Recall that the elliptic integral of the second kind $\hat E(x,k)$ is not periodic. 
In fact, it is asymptotically linear in $x$ and verifies
\[
\hat E(x+2K,k)=\hat E(x,k)+2E(k).
\]
By~\eqref{eq:2.2}, $\partial_x \hat E(x,k) = \dn^2(x,k)$.
Denote $L_\pm = L_\pm^{\cn}$ in this proof. Using~\eqref{sn.sn} and~\eqref{dxsn}, 
we have
\begin{align*}
L_+\cn&=-4k^2\cn^3,\\
L_+(x\cn_x)&=4k^2\cn^3-2(2k^2-1)\cn,\\
L_+\cn^3&=6k^2\cn^5-8(2k^2-1)\cn^3-6(1-k^2)\cn,\\
  L_+(\hat E(x,k)\cn_x)&=6k^4\cn^5-4k^2(3k^2-2)\cn^3+2(1-4k^2+3k^4)\cn.
\end{align*}
Define
\[
\tilde\phi_1=\left(\hat E(x,k)-\frac{
    E(k)}{K(k)}x\right)\cn_x-k^2\cn^3+{\frac {K(k){k}^{2}- E(k)}{K(k)}}\cn.
\]
Then $\tilde\phi_1$ is periodic (of period $4K$) and verifies
\[
L_+\tilde\phi_1=\left(2(2k^2-1)\frac { E(k)}{K(k)}+2(1-k^2)\right)\cn.
\]
The factor is positive if $2k^2\ge 1$. If $2k^2<1$, it is greater than $2(2k^2-1)+2(1-k^2)=2k^2$.
Define, 
\[
\phi_1=\left(2(2k^2-1)\frac {
    E(k)}{K(k)}+2(1-k^2)\right)^{-1}\tilde \phi_1.
\]
Then
\[
L_+\phi_1=\cn.
\]
As for $L_-$, we have
\begin{align*}
L_-(\cn_x)&=4k^2\cn^2\cn_x,\\
L_-(x\cn)&=-2\cn_x,\\
L_-(\hat E(x,k)\cn)&=-2(1-k^2)\cn_x-4k^2\cn^2\cn_x.
\end{align*}
Define 
\[
\tilde \xi_1=\left(\hat E(x,k)-\frac{
    E(k)}{K(k)}x\right)\cn+\cn_x.
\]
Then $\tilde\xi_1$ is periodic (of period $4K$) and verifies
\[
L_-\tilde\xi_1=\left(-2(1-k^2) +\frac{2
    E(k)}{K(k)}\right)\cn_x.
\]
The factor is positive by~\eqref{eq:2.1}.
Defining 
\[
\xi_1=\left(-2(1-k^2) +\frac{2
    E(k)}{K(k)}\right)^{-1}\tilde\xi_1
\]
we get $L_- \xi_1 = \cn_x$. The last statement of the lemma follows from~\eqref{eq:a-b-cn}.
\end{proof}

\subsection{Orthogonality Properties}

The following lemma records some standard properties
of eigenvalues and eigenfunctions of the linearized operator
$J \mathcal L$, which follow only from the structural 
properties~\eqref{eq:sa} and~\eqref{eq:conjugate}:

\begin{lemma} \label{lemma:basics}
The following properties hold.
\begin{enumerate}
\item (symplectic orthogonality of eigenfunctions)
Let $ f=(f_1,f_2)^T$ and $ g=(g_1,g_2)^T$ be two eigenvectors of $J\mathcal L$
corresponding to eigenvalues $\lambda,\mu\in\mathbb C$. 
Then~\eqref{eq:sa} implies
\[
  \lambda + \bar{\mu} \not=0 \; \implies \;  
\ps{f}{ J g}=\ps{f}{\mathcal L g} = 0,
\]
while~\eqref{eq:conjugate} implies
\[
  \lambda - \bar{\mu} \not=0 \; \implies \;  
\ps{Cf}{ J g} = \ps{Cf}{ \mathcal L g}=0,
\]
so that
\[
  \lambda \pm \bar{\mu} \not=0 \; \implies \;
\ps{f_1}{g_2}=\ps{f_2}{g_1}=0.
\]
\item (unstable eigenvalues have zero energy)
If $J \mathcal L f = \lambda f$, $\lambda \notin i\R$, then~\eqref{eq:sa} implies
\[
  \ps{f}{ \mathcal L f }= 0.
\]
\end{enumerate}
\end{lemma}

\begin{proof} 
We first prove (1). We have
\[
\lambda\ps{f}{Jg}=\ps{\lambda f}{Jg}
 =\ps{J\mathcal L f}{ Jg}
= \ps{\mathcal L f}{g}=\ps{f}{ \mathcal L g}=
  -\ps{f}{ \mu Jg} 
=-\bar\mu\ps{f}{ Jg},
\]
so $(\lambda + \bar{\mu}) \ps{f}{ J g} = 0$ which gives the first
statement.
The second statement follows from the same argument with $f$ replaced by $Cf$, 
while the third statement is a consequence of $\ps{f}{ Jg} = \ps{Cf}{ Jg} = 0$.  

Item (2) is a special case of the first statement of (1), with $g=f$.
\end{proof}

\subsection{Spectral Stability of $\sn$ and $\cn$} 

Our goal in this section is to establish Theorem
\ref{thm:linear-stability}, i.e. to prove the spectral stability of
$\sn$ in $P_{4K}$ for all $k\in(0,1)$, and the spectral stability of
$\cn$ in $P_{4K}$ for all $k\in(0,k_c)$.

We first recall the standard fact that 
\[
  \mbox{orbital stability } \implies \mbox{ spectral stability.}
\] 
Indeed, an eigenvalue $\lambda = \alpha + i \beta$ of $J \mathcal L$ with
$\alpha > 0$ produces a solution of the linearized equation whose magnitude
grows at the exponential rate $e^{\alpha t}$, and this linear growing mode 
(together with its orthogonality properties from Lemma~\ref{lemma:basics}) 
can be used to contradict orbital stability. Rather than go through the nonlinear
dynamics, however, we will give a simple direct proof of spectral stability
in the symmetry subspaces where we have the orbital stability -- that is, in 
$P_{2K}$ for $\dn$, and in $A_{2K}$ for $\cn$ and $\sn$ --
using just the spectral consequences for $L_{\pm}$ implied by the (local) 
minimization properties of these elliptic functions: 
\begin{proposition} \label{prop:orbital-spectral}
For $0 < k < 1$, $K = K(k)$,  $\dn$ is spectrally stable in $P_{2K}$ while $\cn$ and
$\sn$ are spectrally stable in $A_{2K}$. Precisely, 
we have
\[
  \sigma(J \mathcal L^{\dn} |_{P_{2K}}) \subset i \R, \quad
  \sigma(J \mathcal L^{\cn} |_{A_{2K}}) \subset i \R, \quad 
  \sigma(J \mathcal L^{\sn} |_{A_{2K}}) \subset i \R.
\]  
\end{proposition}
\begin{proof}
Begin with $\dn$ in $P_{2K}$. From Figure~\ref{fig:spectra}, we see 
$L_-^{\dn} |_{\dn^\perp} > 0$, and thus $(L_-^{\dn})^{\pm 1/2}$ exist 
on $\dn^\perp$. It follows from the minimization property Proposition~\ref{prop:focusing-P_T} that on $\dn^{\perp}$, $L_+^{\dn} \geq 0$ (otherwise
there is a perturbation of $\dn$ lowering the energy while preserving the mass).
Suppose $J \mathcal L^{\dn} f = \lambda f$, $\lambda \not\in i\R$.
Then $L_-^{\dn} L_+^{\dn} f_1 = -\lambda^2 f_1$. 
Since
  $(\dn,0)^T$ is an eigenvector of $J\mathcal L$ for the eigenvalue
  $0$, Lemma~\ref{lemma:basics} implies
$f_1 \perp \dn$. Therefore, we have 
\[
  (L_-^{\dn})^{1/2} L_+^{\dn} (L_-^{\dn})^{1/2} \left( (L_-^{\dn})^{-1/2} f_1 \right)
  = -\lambda^2 \left( (L_-^{\dn})^{-1/2} f_1 \right)
\]
and on $\dn^{\perp}$,
\[
  L_+ \geq 0 \; \implies \; L_-^{1/2} L_+ L_-^{1/2} \geq 0 \; \implies \;
  \lambda^2 \leq 0
\]
contradicting $\lambda \notin i\R$.

Next, consider $\cn$ in $A_{2K}$. Again from Figure~\ref{fig:spectra}, we 
see $L_-^{\cn} |_{\cn^\perp} > 0$, while the minimization property
Proposition~\ref{prop:focusing-A} implies that $L_+^{\cn} \geq 0$ on $\cn^\perp$,
and so the spectral stability follows just as for $\dn$ above.

Finally, consider $\sn$ in $A_{2K}$. 
By Lemma~\ref{lemma:sn-coercivity}, $L_+^{\sn} > 0$
on $\{(\sn)_x\}^\perp$, while $L_-^{\sn} \geq 0$ on 
$\{(\sn)_x\}^\perp$, and so the spectral stability follows from the 
same argument as above, with the roles of $L_+$ and $L_-$ reversed.
\end{proof}

Moreover, both $\sn$ and $\cn$ are spectrally stable in $P_{2K}^-$:
\begin{lemma}
For $0 < k < 1$, $K = K(k)$,
\[
  \sigma(J \mathcal L^{\cn} |_{P_{2K}^-}) \subset i \R, \quad 
  \sigma(J \mathcal L^{\sn} |_{P_{2K}^-}) \subset i \R.
\]
\end{lemma}
\begin{proof}
This is an immediate consequence of the positivity of $\mathcal L^{\sn}$
and $\mathcal L^{\cn}$ on $P_{2K}^-$ (see Figure~\ref{fig:spectra}),
and Lemma~\ref{lemma:basics}.
\end{proof}
So in light of~\eqref{eq:spec-decomp}, to prove Theorem~\ref{thm:linear-stability},
it remains only to show $\sigma(J \mathcal L |_{P_{2K}^+}) \subset i \R$
for each of $\cn$ and $\sn$.

This will follow from a simplified version of a general result for infinite dimensional
Hamiltonian systems (see \cite{HaKa08,KaKeSa04,KaKeSa05})
relating coercivity of the linearized energy with the number of
eigenvalues with negative Krein signature of the linearized operator 
$J\mathcal L$ of the form~\eqref{eq:lin-form}:

\begin{lemma}[coercivity lemma] \label{lem:abstract-coercivity}
Consider $J\mathcal L$ on $S \times S$ for some invariant
subspace $S \subset P_T$, and suppose it has an eigenvalue whose 
eigenfunction $\xi = (\xi_1, \xi_2)^T$ has negative (linearized) energy:
\[
  J \mathcal L \xi = \mu \xi, \quad \scalar{\xi}{\mathcal L\xi} < 0.
\]
Then the following results hold.
\begin{enumerate}
\item
If $L_+$ has a one-dimensional negative subspace (in $S$):
\begin{equation} \label{eq:L+neg}
  L_+ f = -\lambda f, \quad \lambda > 0, \quad L_+ |_{f^\perp} > 0
\end{equation}
Then $L_+ |_{\xi_2^\perp} > 0$.
\item
If $L_-$ has a one-dimensional negative subspace (in $S$):
\begin{equation} \label{eq:L-neg}
  L_- g = -\nu g, \quad \nu > 0, \quad L_- |_{g^\perp} > 0
\end{equation}
Then $L_- |_{\xi_1^\perp} > 0$.
\item
If both~\eqref{eq:L+neg} and~\eqref{eq:L-neg} hold, then
$\sigma(J\mathcal L|_{S\times S}) \subset i \R$. 
\end{enumerate}
\end{lemma}\begin{proof}
First note that by Lemma~\ref{lemma:basics} (2), $0 \not= \mu \in i\R$,
and writing $\mu = i \gamma$, $0 \not= \gamma \in \R$, we have
$L_- \xi_2 = i \gamma \xi_1$, $L_+ \xi_1 = - i \gamma \xi_2$.

Moreover,
\[
\ps{\xi_1}{ L_+ \xi_1} = -\gamma \ps{\xi_1}{i\xi_2} = 
  \gamma \ps{i\xi_1}{\xi_2}  =\ps{L_-\xi_2}{\xi_2},
\] 
so by assumption $\ps{\xi_1}{ L_+ \xi_1}= \ps{\xi_2}{ L_- \xi_2}< 0$.

We prove (1). For any $h \perp \xi_2$, decompose
\[
  h = \alpha f + h_+, \quad  \xi_1 = \beta f + \xi_+, \quad 
  h_+ \perp f, \quad  \xi_+ \perp f,
\]
where we may assume $\alpha \ge 0$ and $\beta \ge 0$. We have
\[
  0 = i \gamma \ps{h}{\xi_2} = \ps{h}{-i \gamma \xi_2} = \ps{h}{L_+ \xi_1}
  = -\lambda \alpha \beta  + \ps{h_+}{L_+ \xi_+}.
\]
Thus, using $L_+ |_{f^\perp} >0 $, $L_+^{1/2} = (L_+ |_{f^\perp})^{1/2}$ 
is well defined on $f^\perp$ and
\begin{align*}
(\alpha\beta \lambda )^2 &= \ps{h_+}{L_+ \xi_+}^2 
= \ps{L_+^{1/2} h_+}{L_+^{1/2} \xi_+}^2  \\
& \le \ps{h_+}{L_+ h_+} \ps{\xi_+ }{L_+ \xi_+} \\
&=  (\ps{h}{L_+ h} + \alpha^2 \lambda) (\ps{\xi_1}{L_+ \xi_1} +\beta^2 \lambda)
\end{align*}
with both factors on the right $> 0$.  
Since $\ps{\xi_1}{L_+ \xi_1} < 0$, we must have $\ps{h}{L_+ h}>0$.

Statement (2) follows in exactly the same way, with the roles
of $L_+$ and $L_-$ reversed, the roles of $\xi_1$ and $\xi_2$ reversed,
and with $g$ and $\nu$ replacing $f$ and $\lambda$.  

Finally, for (3), suppose $J \mathcal L \eta = \zeta \eta$. 
If $\zeta \notin i\R$, then by Lemma~\ref{lemma:basics} (1), 
$\ps{\xi_1}{ \eta_2} = \ps{\xi_2} {\eta_1} = 0$, and so by parts (1) and (2),
\[
  \ps{\eta_1}{ L_+ \eta_1} > 0, \quad
  \ps{\eta_2}{ L_+ \eta_2} > 0, \quad
  \implies \ps{\eta }{ \mathcal L \eta} > 0,
\] 
contradicting Lemma~\ref{lemma:basics} (2). Thus $\zeta \in i\R$.
\end{proof}

\begin{proof}[Proof of Theorem~\ref{thm:linear-stability}]
Begin with $\sn$ in $P_{2K}^+$. From Figure~\ref{fig:spectra} it is clear 
that in $P_{2K}^+$, condition~\eqref{eq:L+neg} holds for $L_+^{\sn}$
and~\eqref{eq:L-neg} holds for $L_-^{\sn}$. 
Explicit computation yields 
\[
 L_+^{\sn}(\dn^2+k^2\cn^2)=-(1-k^2)^2,\quad L_-^{\sn} 1=-(\dn^2+k^2\cn^2),
\]
which implies 
\[
J\mathcal L^{\sn}\binom {\dn^2+k^2\cn^2}{i(1-k^2)}=i(1-k^2)\binom {\dn^2+k^2\cn^2}{i(1-k^2)}.
\]
Moreover, 
\begin{multline*}
\dual{\mathcal L^{\sn}\binom {\dn^2+k^2\cn^2}{i(1-k^2)}}{\binom
  {\dn^2+k^2\cn^2}{i(1-k^2)}}\\=
\dual{L_+^{\sn} (\dn^2+k^2\cn^2) }{\dn^2+k^2\cn^2}+
 (1-k^2)\dual{L_-^{\sn} 1}{1}\\=
-((1-k^2)^2+(1-k^2)) \dual{1 }{(\dn^2+k^2\cn^2)}\\
=-((1-k^2)^2+(1-k^2))  (4E(k)-2(1-k^2)K(k))<0,
\end{multline*}
by \eqref{eq:2.1}.
Hence all the conditions of Lemma~\ref{lem:abstract-coercivity} are
verified for $\sn$ in $P_{2K}^+$, and so we conclude
$\sigma(J \mathcal L^{\sn} |_{P_{2K}^+}) \subset i \R$, as required.

Next we turn to $\cn$. Again from Figure~\ref{fig:spectra} it is clear 
that in $P_{2K}^+$, condition~\eqref{eq:L+neg} holds for $L_+^{\cn}$
and~\eqref{eq:L-neg} holds for $L_-^{\cn}$. Explicit computation yields
\[
 L_+^{\cn}(-\dn^2+k^2\sn^2)=1,\quad L_-^{\cn}1=-\dn^2+k^2\sn^2,
\]
which implies 
\[
J\mathcal L^{\cn}\binom {-\dn^2+k^2\sn^2}{i}=i\binom {-\dn^2+k^2\sn^2}{i}.
\]
Moreover, when $k<k_c$, we have
\[
\dual{\mathcal L^{\cn}\binom {-\dn^2+k^2\sn^2}{i}}{\binom
  {-\dn^2+k^2\sn^2}{i}}
=
2\dual{L_-^{\sn} 1}{1}
=
4K(k)-8E(k)<0.
\]
Hence the conditions of Lemma~\ref{lem:abstract-coercivity} are
verified for $\cn$ in $P_{2K}^+$ when $k<k_c$, yielding
$\sigma(J \mathcal L^{\cn} |_{P_{2K}^+}) \subset i \R$, as required.
\end{proof}


\section{Linear Instability}
\label{sec:instability}

Theorem~\ref{thm:linear-stability} (and Proposition~\ref{prop:orbital-spectral})
give the spectral
stability of the periodic waves
$\dn$, $\sn$, and $\cn$ (at least for $k < k_c$) against perturbations which
are periodic with their fundamental period. It is also natural to ask if this stability
is maintained against perturbations whose period is a \emph{multiple} of the 
fundamental period. In light of Bloch-Floquet theory, this question is also relevant
for stability against \emph{localized} perturbations in $L^2(\R)$.

\subsection{Theoretical Analysis}

It is a simple observation that $\dn$ immediately becomes unstable against 
perturbations with twice its fundamental period:
\begin{proposition}
Both $\sigma(J \mathcal L^{\dn} |_{A^+_{2K}})$ and 
$\sigma(J \mathcal L^{\dn} |_{A^-_{2K}})$ contain a pair of non-zero real eigenvalues.
In particular $\dn$ is linearly
unstable against perturbations in $P_{4K}$.
\end{proposition}  
\begin{proof}
In each of $A_{2K}^+$ and $A_{2K}^-$, $L^{\dn}_- > 0$ while
$L^{\dn}_+$ has a negative eigenvalue: $L^{\dn}_+ f = -\lambda f$,
$\lambda > 0$.
So the self-adjoint operator $(L^{\dn}_-)^{1/2} L_+^{\dn} (L_-^{\dn})^{1/2}$
has a negative direction,
\[
  \left( (L_-^{\dn})^{-1/2} f, \; ((L_-^{\dn})^{1/2} L_+^{\dn} (L_-^{\dn})^{1/2}) 
  (L_-^{\dn})^{-1/2} f \right) = -\lambda ( f, \; f) < 0,
\]
hence a negative eigenvalue 
$(L_-^{\dn})^{1/2} L_+^{\dn} (L_-^{\dn})^{1/2} g = - \mu^2 g$, $\mu > 0$.
Setting $h := (L_-^{\dn})^{-1/2} g$, 
$h \in A_{2K}^+$ ($A_{2K}^-$),
we see 
\[
  L_+^{\dn} L_-^{\dn} h = - \mu^2 h \quad \implies \quad
 J \mathcal L^{\dn} \binom{L_-^{\dn} h}{\pm \mu h} =  
 \pm \mu \binom{L_-^{\dn} h}{\pm \mu h}.
\] 
Hence $\mu,-\mu\in\R$ are eigenvalues of $J\mathcal L$ in $A_{2K}^+$
($A_{2K}^-$).
\end{proof}

\begin{remark}
The proof shows $\dn$ is unstable in $P_{2nK}$ for every even $n$
since $h \in P_{2nK}$. In fact, $\dn$ 
is unstable in any
    $P_{2nK}$, $n \ge 2$. Indeed, we always have $L_-\dn=0$, thus by
    Sturm-Liouville Theory (see e.g.~\cite[Theorem 3.1.2]{Ea73}), $0$
    is always the first simple eigenvalue of $L_-$ in
    $P_{2nK}$. Moreover, $L_+\dn_x=0$, and $\dn_x$ has $2n$ zeros in
    $P_{2nK}$. Hence there are at least $2n-2$ negative eigenvalues
    for $L_+$ in $P_{2nK}$. With the above argument, this proves
  linear instability in $P_{2nK}$ for any $n \ge 2$.
\end{remark}

For $\sn$, the $H^2(\R)$ orbital stability result of \cite{BoDeNi11,GalPel15} implies 
spectral stability against perturbations which are periodic with \emph{any} multiple 
of the fundamental period.

Using formal perturbation theory, \cite{Row74} showed that
$\cn$ becomes unstable against perturbations which are periodic with
period a sufficiently large multiple of the fundamental period. Our main goal 
in this section is to make this rigorous:
\begin{theorem}\label{thm:instability}
For $0<k<1$,
there exists $n_1=n_1(k)\in\mathbb N$ such that $\cn$ is linearly unstable in $P_{4nK}$ for $n \ge n_1$, i.e.,
the spectrum of $J\mathcal L^{\cn}$
as an operator on $ P_{4nK}$ contains an eigenvalue with
positive real part. 
\end{theorem}
We will in fact prove a slightly more general result, which is the
existence of a branch strictly contained in the first quadrant  for
the spectrum of $J\mathcal L^{\cn}$ considered as an operator
on $L^2(\R)$. 
Theorem~\ref{thm:instability} will be a consequence of a more general
perturbation result applying to all real periodic waves (see Proposition~\ref{prop:perturbation}), and in particular not relying on any integrable
structure.

We start with some preliminaries. Let 
\[
J\mathcal L=
\begin{pmatrix}
0 & L_-\\-L_+&0
\end{pmatrix}
\]
with 
\[
L_-=-\partial_{xx}-a-b u^2,\quad L_+=-\partial_{xx}-a-3b u^2
\]
where $u$ a periodic solution to 
\begin{equation}
\label{eq:assump-1}
u_{xx}+a u+b|u|^2\ u=0.
\end{equation}
We assume that  $u(x) \in \R$ 
and let $T$ denote a period of $u^2$. 
The spectrum of $J\mathcal L$ as an operator on $L^2(\R)$
can be analyzed using Bloch-Floquet decomposition. For $\theta\in [0,2\pi/T)$, define
\[
J\mathcal L^\theta=
\begin{pmatrix}
0 & L_-^\theta\\-L_+^\theta&0
\end{pmatrix}
\]
where $L_\pm^\theta$ is the operator obtained when formally replacing
$\partial_x$ by $\partial_x+i\theta=e^{-i\theta
  x}\partial_x\left(e^{i\theta x}\cdot\right)$ in the expression of
$L_\pm$. If we let $(M^\theta f )(x)= e^{i \theta x} f(x)$, then $L^\theta _\pm = M^{-\theta}L_\pm M^\theta$.
Then we have
\begin{equation}\label{eq:bloch}
\sigma\left(J\mathcal
  L |_{L^2(\R)}\right)=\bigcup_{\theta\in[0,\frac{2\pi}{T}]}\sigma\left(J\mathcal
  L^\theta |_{P_T}\right).
\end{equation}
In what follows, all operators are considered on $ P_T$
unless otherwise mentioned.

Let us consider the case $\theta=\frac\pi T$. Denote 
\[
D=\partial_x+i\frac\pi T.
\]
Since $u$ is a real valued periodic solution to \eqref{eq:assump-1},
by Lemmas \ref{th:2-1} and \ref{th:2-2}, $u$ is a rescaled $\cn$,
$\dn$ or $\sn$. In any case, the following holds:
\begin{equation} \label{eq:assumption}
\begin{split}  & \varphi=e^{-i\frac\pi Tx} u, \;\;
  \psi=D\varphi=e^{-i\frac\pi Tx} u_x \; \in 
  H_{\mathrm{loc}}^1\cap P_T\setminus\{0\} \\
  & \mbox{ are such that }
  \ker(L^{\frac{\pi}{T}}_-)=\langle  \varphi\rangle,
  \; \; \ker(L^{\frac{\pi}{T}}_+)=\langle  \psi\rangle.
\end{split}
\end{equation} 
Note that for any $f,g\in H^1_{\mathrm{loc}}\cap P_T$, we can integrate by parts with $D$:
\[
\int_0^T Df \bar g \; dx = -\int_0^T  f \overline{Dg} \; dx.
\]
Remark that
\[
\ps{\varphi}{\psi} = \int_0^T\varphi \bar\psi\; dx=\int_0^T uu_x \; dx=0.
\]
Therefore there exist $\varphi_1,\psi_1$ such that
\[
L_-^{\frac \pi T}\psi_1=\psi,\quad L_+^{\frac \pi T}\varphi_1=\varphi,
\quad \varphi _1\perp \psi, \quad \psi_1 \perp \varphi.
\]
The kernel of the operator $J\mathcal L^{\frac \pi T}$ is generated
by $\binom{\psi}{0},$ $\binom{0}{\varphi}$. On top of that, the generalized kernel of
$J\mathcal L^{\frac \pi T}$ contains (at least) $\binom{0}{\psi_1},$ $\binom{\varphi_1}{0}$.

Our goal is to analyze the spectrum of the operator $J\mathcal L^{\frac
  \pi T -\eps}$ when $|\eps|$ is small. In particular, we want to locate the eigenvalues generated by perturbation of the generalized kernel of $J \mathcal L^{\frac \pi T}$. 
For the sake of simplicity in notation, when $\theta=\frac \pi T$, we
use a tilde to replace the exponent $\frac \pi T$. In particular, we write
\[
J\mathcal L^{\frac \pi T}=
J\tilde{\mathcal L},\quad 
L_\pm^{\frac \pi T}=
 \tilde L_\pm.
\]

\begin{proposition}\label{prop:perturbation}
Assume the condition~\eqref{eq:ass2} stated below.
There exist $\lambda_1\in\mathbb C$ with 
$\Re(\lambda_1) > 0$, $\Im(\lambda_1) > 0$; $b_0\in\C$;
and $\eps_0>0$, such that for all $0\leq \eps<\eps_0$ there exist
$\lambda_2(\eps)\in\mathbb C$, $b_1(\eps)\in\mathbb C$, 
$v_2(\eps),w_2(\eps)\in H_{\mathrm{loc}}^2 \cap P_T$,
\begin{gather}
\abs{b_1(\eps)}+\abs{\lambda_2(\eps)}+\norm{v_2(\eps)}_{H_{\mathrm{loc}}^2\cap P_T}+\norm{w_2(\eps)}_{H_{\mathrm{loc}}^2\cap P_T}\lesssim
1
\nonumber
\\
\label{v2w2orthogonal}
v_2(\eps)\perp \psi,\quad w_2(\eps)\perp \vp,
\end{gather}
verifying the following property. Set 
\begin{align}
v_0&=b_0\psi,&v_1&=b_1(\eps)\psi-2ib_0 \tilde L_+^{-1}D\psi-\lambda_1\varphi_1,\label{eq:def_v}\\
w_0&=\varphi,&w_1&=(b_0\lambda_1-2i)\psi_1. \label{eq:def_w}
\end{align}
Here, $\tilde L_+^{-1}$ is taken from $\psi^\perp$ to $\psi^\perp$.
Define
\[
\begin{array}{lllllll}
\lambda&=&&&\eps\lambda_1&+&\eps^2\lambda_2(\eps),\\
v&=&v_0&+&\eps v_1(\eps)&+&\eps^2v_2(\eps),\\
w&=&w_0&+&\eps w_1&+&\eps^2w_2(\eps).
\end{array}
\]
Then 
\[
J\mathcal L^{\frac \pi T-\eps}\binom{v}{w}=\lambda \binom{v}{w}.
\]
\end{proposition}

Note that the orthogonality conditions in \eqref{v2w2orthogonal} are reasonable: The eigenvector is normalized by $P_\vp w = w_0 = \vp$, and hence $w_2 \perp \vp$. To impose $v_2 \perp \psi$,  we allow $b_1(\eps)\psi$ in $v_1$ to be $\eps$-dependent to absorb $P_\psi (v-v_0)$.

\begin{proof}[Proof of Proposition~\ref{prop:perturbation}]
Let us write the expansion of the operators in $\eps$.  We have
\[
L_\pm^{\frac \pi T -\eps}=L_\pm^{\frac \pi T}+2i\eps D +\eps^2,
\]
Therefore
\[
J\mathcal L^{\frac \pi T -\eps}=
J\mathcal L^{\frac \pi T}+\eps
\begin{pmatrix}
0 &2iD\\
-2iD&0
    \end{pmatrix}
+\eps^2
\begin{pmatrix}
  0&1\\-1&0
\end{pmatrix}
\]
We expand in $\eps$ the equation $\left(J\mathcal L^{\frac \pi T-\eps}-\lambda\mathcal
  I\right)\binom{v}{w}=0$ and show that it can be satisfied at each order
of $\eps$. 

At order $\mathcal O(1)$, we have
\[
J\tilde{\mathcal L} \binom{v_0}{w_0}=0,
\]
which is satisfied because $\binom{v_0}{w_0}\in\ker(J\tilde{\mathcal L})$ by
definition.

At order $\mathcal O(\eps)$, we have
\[
J\tilde{\mathcal L} \binom{v_1}{w_1}+\begin{pmatrix}
-\lambda_1 &2iD\\
-2iD&-\lambda_1
    \end{pmatrix}\binom{v_0}{w_0}=0,
\]
which can be rewritten, using the expression of $v_0$, $w_0$, and
$D\varphi=\psi$, as
\begin{align}
\tilde L_-w_1&=(b_0\lambda_1-2i)\psi,\label{eq:w_1}\\
\tilde L_+v_1&=-2ib_0D\psi-\lambda_1\varphi. \label{eq:v_1}
\end{align}
It is clear that the functions $v_1(\eps)$ and $w_1$ defined in
\eqref{eq:def_v}-\eqref{eq:def_w} satisfy~\eqref{eq:w_1}-\eqref{eq:v_1}.

At order $\mathcal O(\eps^2)$, we consider the equation as a whole, 
involving also the higher orders of $\eps$. We have
\begin{multline*}
J\tilde{\mathcal L} \binom{v_2}{w_2}+\begin{pmatrix}
-\lambda_1 &2iD\\
-2iD&-\lambda_1
    \end{pmatrix}\binom{v_1}{w_1}
+\begin{pmatrix}
  -\lambda_2&1\\-1&-\lambda_2
\end{pmatrix}\binom{v_0}{w_0}
\\
+\eps \left(\begin{pmatrix}
-\lambda_1 &2iD\\
-2iD&-\lambda_1
    \end{pmatrix}\binom{v_2}{w_2}+\begin{pmatrix}
  -\lambda_2&1\\-1&-\lambda_2
\end{pmatrix}\binom{v_1}{w_1}
\right)
\\
+
\eps^2 \begin{pmatrix}
  -\lambda_2&1\\-1&-\lambda_2
\end{pmatrix}\binom{v_2}{w_2}
=0,
\end{multline*}
in other words
\begin{align}
\tilde L_-w_2&=W_2+\eps W_3+\eps^2W_4\label{eq:w_2}\\
\tilde L_+v_2&=V_2+\eps V_3+\eps^2V_4\label{eq:v_2}
\end{align}
where
\begin{equation} \label{eq:V,W}
\begin{array}{ll}
W_2 =\lambda_1 v_1-2iDw_1+\lambda_2v_0-w_0, &
V_2 =-2iDv_1-\lambda_1w_1-v_0-\lambda_2w_0,\\
W_3 =\lambda_1 v_2-2iDw_2+\lambda_2v_1-w_1, &
V_3 =-2iDv_2-\lambda_1w_2-v_1-\lambda_2w_1,\\
W_4 =\lambda_2v_2-w_2, & V_4=-v_2-\lambda_2w_2.
\end{array}
\end{equation}
Note that $V_2$ and $W_2$ depend on $b_0,\lambda_1$ and
$b_1,\lambda_2$, whereas $V_3,V_4$ and $W_3,W_4$ also depend on $v_2$
and $w_2$. 
Our strategy to solve the system~\eqref{eq:w_2}-\eqref{eq:v_2} is
divided into two steps. We first ensure that it can be solved at the main order by
ensuring that the compatibility conditions 
\begin{equation}
\label{eq:compatibility_V2_W2}
\ps{W_2}{\varphi}=\ps{V_2}{\psi}=0
\end{equation}
are satisfied. This is achieved by making a suitable choice of
$b_0,\lambda_1$. Then we solve for $b_1,\lambda_2,v_2,w_2$ by using a
Lyapunov-Schmidt argument. 

We rewrite the compatibility conditions
\eqref{eq:compatibility_V2_W2} in the following form, using the expressions
for $v_0,$ $w_0$, $v_1$ and $w_1$, and the properties of $\varphi$ and $\psi$:
\begin{align*}
&\phantom{b_0}\ps{\varphi_1}{\varphi}\lambda_1^2+b_02i\left(\ps{
  D\psi}{\varphi_1}-\ps{\psi_1}{\psi}\right)\lambda_1+\phantom{b_0}\left(\ps{\varphi}{\varphi}-4\ps{\psi_1}{\psi}\right)=0\\
&b_0\ps{\psi_1}{\psi}\lambda_1^2
+\phantom{b_0} 2i \left(
  \ps{\varphi_1}{D\psi}
-\ps{\psi_1}{\psi}\right)\lambda_1
+b_0\left(\ps{\psi}{\psi}-4 \ps{\tilde L_+^{-1}D\psi}{D\psi}\right)=0
\end{align*}
These equations do not depend on $b_1$ or $\la_2$ although $W_2 $ and $V_2$ do.
For a moment, we write these equations as 
\begin{align}
 \phantom{b_0} A_1\lambda_1^2+b_0B\lambda_1+\phantom{b_0} C_1&=0,\label{eq:first}\\
b_0A_2\lambda_1^2+\phantom{b_0} B\lambda_1+b_0C_2&=0,\label{eq:second}
\end{align}
where
\[
\begin{split}
  &A_1 := \ps{\varphi_1}{ \varphi} \quad \in \R, \\
  &A_2 := \ps{\psi_1}{ \psi} \quad \in \R, \\
  &B := 2i\left(\ps{D\psi}{\varphi_1}-\ps{\psi_1}{\psi}\right) \quad \in i \R, \\
  &C_1 := \ps{\varphi}{\varphi}-4\ps{\psi_1}{\psi} \quad \in \R, \\
  &C_2 := \ps{\psi}{\psi}-4 \ps{\tilde L_+^{-1}D\psi}{D\psi} \quad \in \R.
\end{split}
\]
Multiplying~\eqref{eq:first} by $C_2+A_2\lambda_1^2$,~\eqref{eq:second} by $B\lambda_1$, and subtracting gives 
\begin{equation} \label{eq:quadratic}
 A_1A_2\lambda_1^4+
  (A_1C_2 +A_2C_1-B^2)\lambda_1^2+C_1C_2=0,
\end{equation}
a quadratic equation in $\lambda_1^2$ with real coefficients. 
If $A_1A_2\neq 0$, the roots of~\eqref{eq:quadratic} are given by 
\[
 \lambda_1^2= \frac{- (A_1C_2 +A_2C_1-B^2) \pm\sqrt{(A_1C_2 +A_2C_1-B^2)^2-4 A_1A_2 C_1C_2}}{2A_1A_2}
\]
We now assume the discriminant of this quadratic is negative:
\begin{equation} \label{eq:ass2}
  (A_1 C_2 + A_2 C_1-B^2)^2 -4 A_1A_2 C_1C_2 < 0
\end{equation}
which implies that $A_1 A_2 \not= 0$, and moreover
guarantees the existence of a root $\lambda_1$ of~\eqref{eq:quadratic}
strictly contained in the first quadrant: 
$\Re \lambda_1 > 0$ and $\Im \lambda_1 > 0$ 
(the other roots being $-\lambda_1$, $\pm \bar{\lambda}_1$). 
It follows from~\eqref{eq:ass2} that $B \not= 0$, and so we may 
solve~\eqref{eq:first} and set
\[
  b_0 := -\frac{(A_1 \lambda_1^2 + C_1)}{B \lambda_1},
\]
so that both~\eqref{eq:first} and \eqref{eq:second} are satisfied.

We now solve for $b_1,\lambda_2,v_2,w_2$ 
using a Lyapunov-Schmidt argument. The first step is to
solve, given $(b_1, \la_2)$, projected versions 
of~\eqref{eq:w_2}-\eqref{eq:v_2},
\begin{equation} \label{eq:projected}
\begin{split}
  \tilde L_-w_2 &= W_2 + P_{\vp^{\perp}} \left[ 
  \eps W_3+\eps^2W_4 \right]  \\
  \tilde L_+v_2 &= V_2 + P_{\psi^{\perp}} \left[
  \eps V_3+\eps^2V_4 \right]
\end{split}
\end{equation}
to obtain $v_2 = v_2(b_1,\la_2) \in \psi^\perp$, 
$w_2 = w_2(b_1,\la_2) \in \vp^\perp$:
\begin{lemma}
Given any $b_1 \in \C$, $\la_2 \in \C$
with $|b_1| + |\la_2| \leq M$, there is a unique solution 
\[
  (v_2, w_2) = (v_2(b_1,\la_2),w_2(b_1,\la_2)) \in 
(H^2_{loc} \cap P_T \cap \psi^{\perp}) \times (H^2_{loc} \cap P_T \cap \vp^{\perp})
\]
of~\eqref{eq:projected}, with $\|v_2\|_{H^2} + \| w_2 \|_{H^2} \leq C(M)$.
\end{lemma}
\begin{proof}
By the expressions~\eqref{eq:V,W}, we may rewrite system~\eqref{eq:projected} as a linear system of $v_2$ and $w_2$,
\[
  \tilde {\mathcal L}_\eps  \begin{pmatrix} v_2 \\ w_2 \end{pmatrix} = 
  \begin{pmatrix} S_v \\ S_w \end{pmatrix} + 
  \eps \begin{pmatrix} R_v \\ R_w \end{pmatrix} ,
\]
where
\[
  \tilde {\mathcal L}_\eps =
  \begin{pmatrix} \tilde L_+ + P_{\psi^\perp} 2i \eps D + \eps^2 & 
  (\eps \la_1  + \eps^2 \la_2)P_{\psi^\perp}  \\ 
  -(\eps \la_1  + \eps^2 \la_2)P_{\vp^\perp}& 
  \tilde L_- + 
  P_{\vp^\perp} 2i \eps D + \eps^2     
  \end{pmatrix},
\]
\begin{equation} \label{eq:sources}
\begin{split}
  & S_v = P_{\psi^\perp} ( -2iD[b_1\psi - 2ib_0 \tilde L_+^{-1} D \psi - \la_1 \vp_1]
  - \la_1(b_0 \la_1 - 2i) \psi_1 - \la_2 \vp )   \\
  & S_w = P_{\vp^\perp} ( \la_1[ b_1\psi - 2ib_0 \tilde L_+^{-1} D \psi - \la_1 \vp_1]
  - 2iD(b_0 \la_1 - 2i) \psi_1 + \la_2 b_0 \psi ) \\
  & R_v = P_{\psi^\perp} ( \la_2(2i-b_0 \la_1) \psi_1 - 
  [- 2ib_0 \tilde L_+^{-1} D \psi - \la_1 \vp_1] ) \\
  & R_w = P_{\vp^\perp} ( \la_2 [b_1\psi - 2ib_0 \tilde L_+^{-1} D \psi - \la_1 \vp_1]
  + (2i-b_0 \la_1) \psi_1 ).
  \end{split}
\end{equation}
Note that $S_v,S_w,R_v$ and $R_w$ do not contain $v_2,w_2$ or $\eps$.
Recalling the definition
\[
  \tilde {\mathcal L} =
  \begin{pmatrix} \tilde L_+ & 0 \\ 0 & \tilde L_- 
  \end{pmatrix}, 
\]
it follows from~\eqref{eq:assumption} that
\[
  \tilde {\mathcal L}^{-1} :(P_T \cap \psi^\perp) \times (P_T \cap \vp^\perp)
  \to (P_T \cap H^2_{loc} \cap \psi^\perp) \times (P_T \cap H^2_{loc}
  \cap \vp^\perp)
\]
is bounded, and hence so is $\tilde {\mathcal L}_\eps^{-1}$, uniformly in
$\eps$ for $\eps$ sufficiently small, with
\[
  \| \tilde {\mathcal L}_\eps^{-1} - \tilde {\mathcal L}^{-1} 
  \|_{(L^2 \times L^2 \to H^2 \times H^2)} \lesssim \eps. 
\]
Thus
\begin{equation} \label{eq:v2,w2}
  \begin{pmatrix} v_2 \\ w_2 \end{pmatrix}  = 
  \tilde {\mathcal L}_\eps^{-1} \left(
  \begin{pmatrix} S_v \\ S_w \end{pmatrix} + 
  \eps \begin{pmatrix} R_v \\ R_w \end{pmatrix} \right) 
  = \begin{pmatrix} \tilde L_+^{-1} S_v \\ \tilde L_-^{-1} S_w \end{pmatrix}
  + O_{H^2 \times H^2}(\eps)
\end{equation}
gives $(v_2(b_1,\la_2), w_2(b_1,\la_2))$ as desired.
\end{proof} 

The second step is to plug $(v_2(b_1,\la_2), w_2(b_1,\la_2))$ back into 
$V_3$, $V_4$, $W_3$, $W_4$, and solve, for $(b_1, \la_2)$, 
the remaining compatibility conditions  
\begin{equation} \label{eq:compatibility}
 (V_3+\eps V_4,\psi)=  (W_3+\eps W_4,\varphi)=0
\end{equation}
which, together with~\eqref{eq:projected}, complete the solution
of the eigenvalue problem. Using~\eqref{eq:v2,w2} and~\eqref{eq:V,W},
we may write~\eqref{eq:compatibility} as the system
\begin{equation*} 
\begin{split}
  0 &= (-2iD[\tilde L_+^{-1} S_v + O (\eps)] -\la_1 
  [\tilde L_-^{-1} S_w + O (\eps)]
  -v_1 - \la_2 w_1 + \eps V_4, \psi) \\
  0 &= (  \la_1[\tilde L_+^{-1} S_v + O (\eps)] -2iD[
  \tilde L_-^{-1} S_w + O (\eps)] + \la_2 v_1 - w_1 
  + \eps W_4, \vp) 
\end{split}
\end{equation*}
and then by the expressions~\eqref{eq:def_v}-\eqref{eq:def_w}
and~\eqref{eq:sources}, we may further rewrite as
\begin{equation} \label{eq:compat2}
\Phi(b_1,\la_2,\eps) = 
  (M + O(\eps)) \begin{pmatrix} b_1 \\ \la_2 \end{pmatrix}
  +F+O(\eps) =0
\end{equation}
where $\Phi$ is a rational vector function of $b_1,\la_2$ and $\eps$;
$F$ is a fixed (independent of $(b_1,\la_2)$) vector
with $|F| \lesssim 1$; and $M = \frac {\partial \Phi}{\partial (b_1,\la_2)}|_{\eps=0}$ is the matrix
\[
\begin{split}
  &M = \begin{pmatrix}
  ( -4D \tilde L_+^{-1}D \psi - \la_1^2 \psi_1 -\psi, \psi) &  
  (2iD \vp_1 - 2(\la_1 b_0 - i) \psi_1, \psi)
  \\ (-2i \la_1 \tilde L_+^{-1} D\psi  -2i \la_1 D \psi_1, \vp) 
  & (-\la_1 \vp_1  -2i b_0 D \psi_1 -2ib_0 \tilde L_+^{-1} D \psi - \la_1 \vp_1, \vp)
  \end{pmatrix} \\
  &= \begin{pmatrix}
  4(\tilde L_+^{-1}D \psi, D \psi) - \la_1^2 (\psi_1,\psi) - (\psi, \psi) &  
  -2i(\vp_1,D\psi)  - 2(\la_1 b_0 - i) (\psi_1, \psi) \\ 
  -2i \la_1 (D\psi, \vp_1)  +2i \la_1 (\psi_1, \psi) 
  & -2\la_1 (\vp_1,\vp)  +2i b_0 (\psi_1,\psi)  
  -2ib_0 (D \psi, \vp_1)
  \end{pmatrix} \\
  &= \begin{pmatrix}
  -C_2 -A_2 \la_1^2 &  
  -B - 2\la_1 b_0 A_2 \\ 
  -\la_1 B  & -b_0 B - 2 \la_1 A_1
   \end{pmatrix} = \begin{pmatrix}
  \frac{B}{b_0}\la_1 &  
  -B - 2\la_1 b_0  A_2 \\ 
  -\la_1 B  & -b_0 B - 2 \la_1 A_1
   \end{pmatrix}.
\end{split}
\]
where in the last step we used~\eqref{eq:second}. The determinant of $M$ is, using \eqref{eq:second} and \eqref{eq:first} to eliminate $b_0$,
\[
\begin{split}
\det M &= -2 \la_1  B  \left(B+  \frac{\la_1}{b_0} A_1+\la_1 b_0 A_2 \right) 
\\
& = -2 \la_1  B  \left(B-  \frac{A_2 \la_1^2+C_2}B A_1- \frac {A_1 \la_1^2+C_1}B A_2 \right) 
\\
& = -2 \la_1    \left(B^2-  2 A_1 A_2 \la_1^2- C_2 A_1- C_1 A_2 \right) .
\end{split}
\]
Since $A_1,A_2,C_1,C_2,B^2$ are real, and $\la_1A_1A_2\not =0$, we have $\det M \not =0$, otherwise $\la_1^2 \in \R$.

Thus $(b_1,\la_2)$ may be solved from~\eqref{eq:compat2} for $\eps$ sufficiently small by the implicit function theorem, providing the required
solution to~\eqref{eq:compatibility}, and so completing the 
proof of Proposition~\ref{prop:perturbation}. 
\end{proof}

\begin{proof}[Proof of Theorem~\ref{thm:instability}]
We need only verify the assumptions of Proposition~\ref{prop:perturbation}
for the case of $u(x) = \cn(x;k)$, $T = 2K(k)$.
Since $u = \cn \in A_{2K}$, we have $u^2 = \cn^2 \in P_T$.
Moreover, \eqref{eq:assumption} holds (see Figure~\ref{fig:spectra}). 
It remains to verify the condition~\eqref{eq:ass2}. 
The values of the
coefficients for the equations of $b_0$ and $\lambda_1$ are given by
the following formulas, obtained by using the
equation verified by $\cn$ and the explicit expressions given by 
Lemma~\ref{lem:explicit}.  Due to the complicated nature of the
expressions, the dependence of $E$ and $K$ on $k$ will be left implicit. 
\begin{align*}
A_1&=\ps{\varphi_1}{\varphi}=\ps{\phi_1}{u}=
\frac{k^2K(2E-K)+(E-K)^2}{2k^2(E(1-2k^2)-K(1-k^2))},\\
A_2&=\ps{\psi_1}{\psi}=\ps{\xi_1}{u_x}=\frac{k^2K(2E-K)+(E-K)^2}{2k^2(E-(1-k^2)K)},\\
B
   &=2i\left(\ps{\varphi_1}{D\psi}-\ps{\psi_1}{\psi}\right)=2i\left(\ps{\phi_1}{u_{xx}}-A_2\right)\\
&=-i\frac{2EK(k-1)(k+1)(K-E)}{(E(1-2k^2)-K(1-k^2))(E-(1-k^2)K)},
 \\
C_1&=\ps{\varphi}{\varphi}-4\ps{\psi_1}{\psi}=\ps{u}{u}-4A_2=\frac{2K^2(k-1)(k+1)}{E-(1-k^2)K},\\
C_2&=\ps{\psi}{\psi}-4\dual{\tilde
     L_+^{-1}D\psi}{D\psi}=\ps{u_x}{u_x}-4\dual{
     L_+^{-1}u_{xx}}{u_{xx}}\\
&=\frac{2K^2(k-1)(k+1)}{E(1-2k^2)-K(1-k^2)}.
\end{align*}
Therefore, 
\begin{multline*}
(A_1C_2 +A_2C_1-B^2)^2-4 A_1A_2 C_1C_2\\=
-\frac{16K^4E^2(1-k)^3(1+k)^3(K-E)^2}{k^2(E-(1-k^2)K)^2(E(1-2k^2)-(1-k^2)K
)^2}<0,
\end{multline*}
Thus Proposition~\eqref{prop:perturbation} applies, providing an unstable eigenvalue
of $J (\mathcal L^{\cn})^\theta$ for $\theta = \frac{\pi}{2K} - \eps$, and all 
$0 < \eps \leq \eps_0$. It follows in particular that $\cn$ is unstable against 
perturbations with period $4nK$, where $n$ is the smallest even integer
$\geq \frac{\pi}{K \eps_0}$.  
This concludes the proof of Theorem~\ref{thm:instability}.
\end{proof}

\subsection{Numerical Spectra}

We have tested numerically the spectra of the different operators
involved. To this aim, we used a fourth order centered finite difference
discretization of the second derivative operator. Unless otherwise
specified, we have used $2^{10}$ grid points. The spectra are then
obtained using the built in function of our scientific computing
software (Scilab). Whenever the spectra can be theoretically
described, the theoretical description and our numerical computations
are in good agreement. 

We start by the presentation of the spectra of $J\mathcal L^{\pq}$, for
$\pq=\cn,\dn,\sn$ on $P_{4K}$.

\begin{observation}
 On $P_{4K}$, the spectrum of  $J\mathcal L^{\pq}$ is such that
 \begin{itemize}
 \item if $\pq=\sn$ then $\sigma(J\mathcal L^{\sn})\subset i\R$ for all $k\in(0,1)$,
\item if $\pq=\cn$, then $\sigma(J\mathcal L^{\cn})\subset i\R$ for all
  $k\in(0,1)$, including when $k>k_c$,
\item if $\pq=\dn$, then $J\mathcal L^{\dn}$ admits two double
  eigenvalues $\pm\lambda$ with $\lambda>0$ and the rest of the
  spectrum verifies
$(\sigma(J\mathcal L^{\dn})\setminus\{\pm\lambda\})\subset i\R$ for all
  $k\in(0,1)$. 
 \end{itemize}
The numerical observations for $\cn$ and $\dn$ at $k=0.95$ are
represented in Figure~\ref{fig:sfig4}.
\end{observation}

\begin{figure}[htpb!]
  \centering
  \includegraphics[width=.5\linewidth]{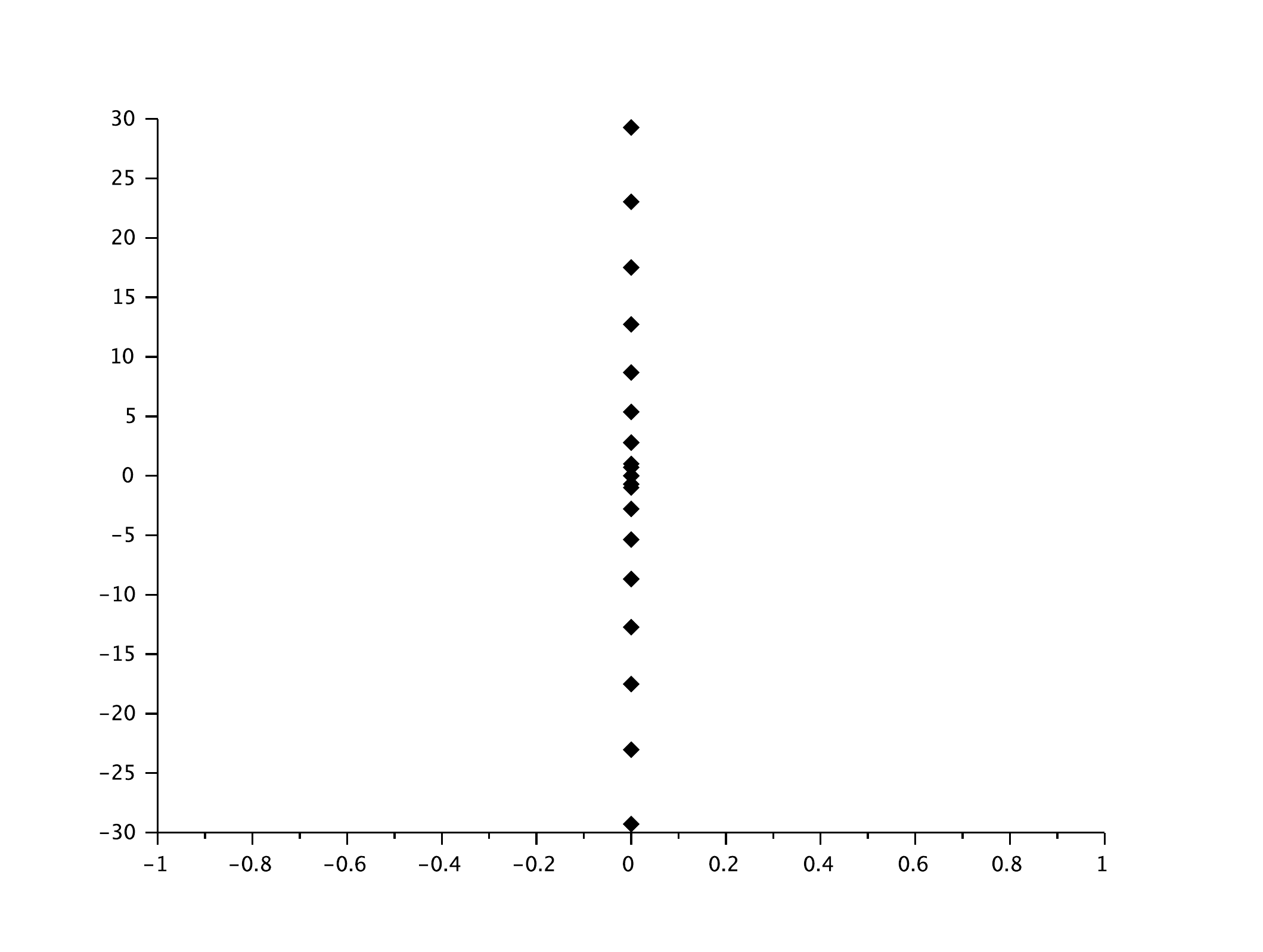}\hfill
  \includegraphics[width=.5\linewidth]{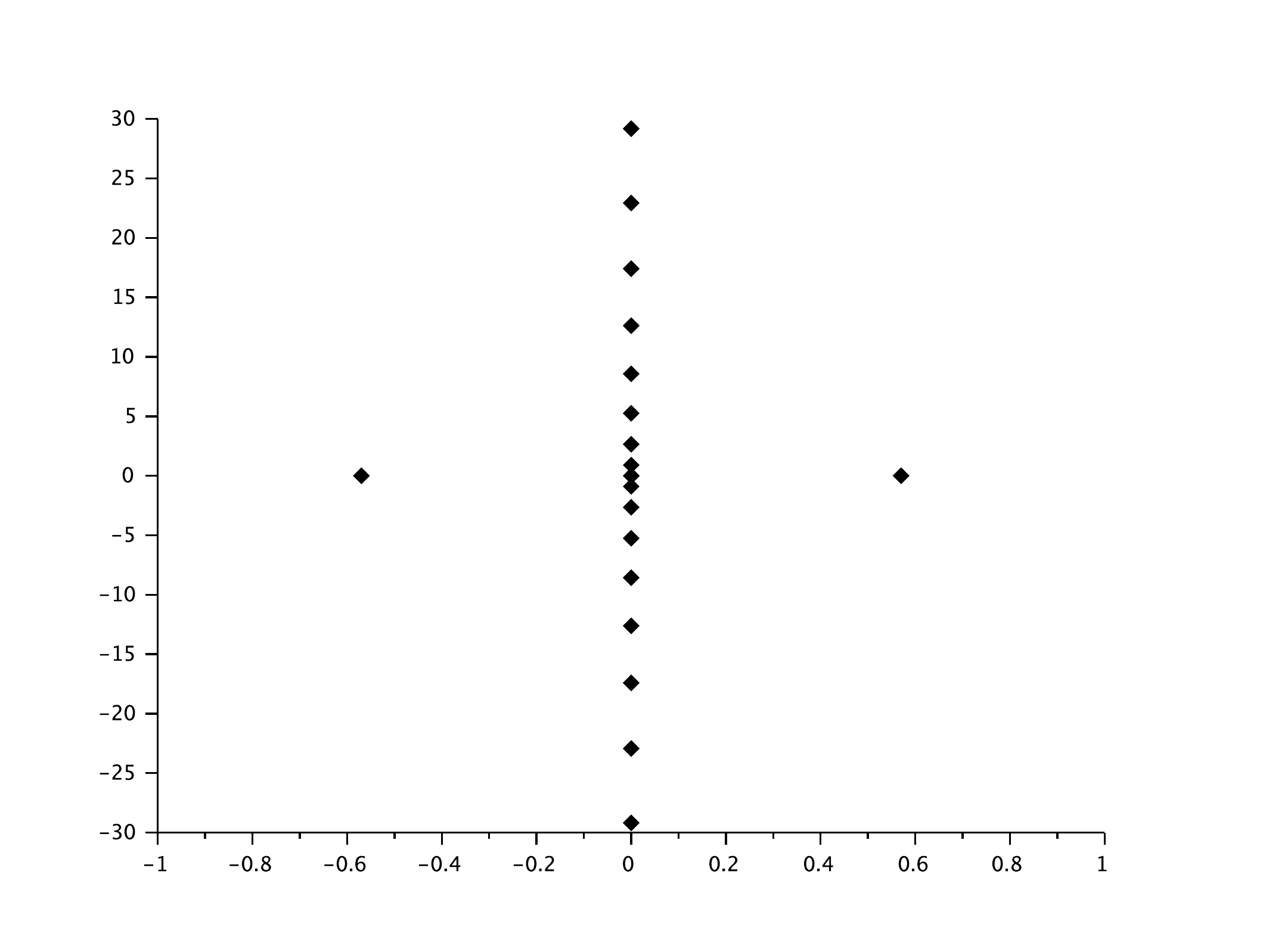}
  \caption{$\sigma(J\mathcal L^{\cn})$ (left) and $\sigma(J\mathcal
    L^{\dn})$ (right) on $P_{4K}$ for $k=0.95$}
  \label{fig:sfig4}
\end{figure}

\FloatBarrier

We then compare the results of Theorem~\ref{thm:instability} with the
numerical results. In Figure~\ref{fig:sfig5}, we have drawn the numerical spectrum of
$J\mathcal L ^{\cn}$ as an operator on $L^2(\R)$. To this aim, we have
used the Bloch decomposition of the spectrum of  $J\mathcal L ^{\cn}$
given in~\eqref{eq:bloch}: we computed the spectrum of $J(\mathcal L
^{\cn})^\theta$ for $\theta$ in a discretization of
$(0,\frac{\pi}{2K}]$ and we have interpolated between the values
obtained to get the curve in plain (blue) line. In order to keep the
computation time reasonable, we have dropped the number of space
points from $2^{10}$ to $2^8$.
We then have drawn in
dashed (red) the
straight lines passing through the origin and the points whose
coordinates are given in the complex plane by
$\pm\lambda_1,\pm\bar\lambda_1$, $\lambda_1$  given in the proof of
Proposition~\ref{prop:perturbation}. The picture shows that the dashed
(red) line are tangent to the plain
(blue) curve, thus confirming $\lambda_1$ as the first order in the
expansion for the eigenvalue emerging from $0$ performed in
Proposition~\ref{prop:perturbation}.

\begin{figure}[htpb!]
  \centering
  \includegraphics[width=.7\linewidth]{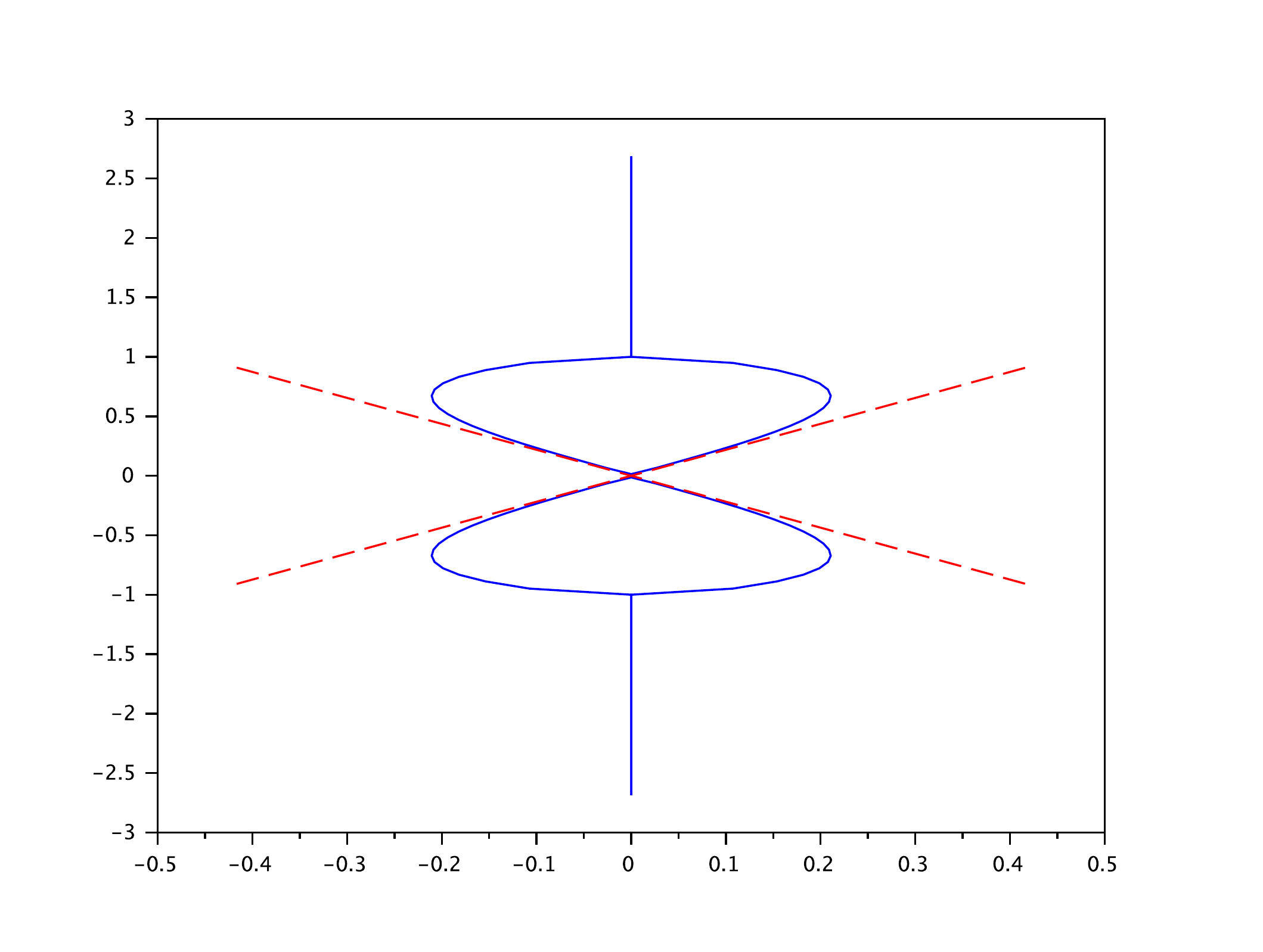}
  \caption{$\sigma(J\mathcal L^{\cn})$ on $L^2(\R)$ for $k=0.9$ (plain (blue)
    curve), first order asymptotic around $0$ (dashed (red) lines) }
  \label{fig:sfig5}
\end{figure}
\FloatBarrier

Numerically, eigenvalues on the number 8 curve in Figure~\ref{fig:sfig5} are simple, and move from the origin toward the intersection points of the number 8 curve with the imaginary axis, when $\theta$ is decreased from $\pi/(2K)$ to $0^+$.

These eigenvalues are simple because we did the Block decomposition \eqref{eq:bloch} in $P_{2K}$ with $\theta \in  [0,2\pi/T)=[0,\pi/K)$, and $\cn$ is only in $P_{2K}(-1)$, not in $P_{2K}$. Thus it is in the kernel of $L_+^\theta$ only for $\theta = \pi/(2K)$. The bifurcation occurs only near $\theta=\pi/(2K)$, not at $\theta=0$.

In contrast, Rowlands \cite{Row74} did the Block decomposition in $P_{4K}$ with $\theta \in [0,\pi/(2K))$. We have $\cn \in P_{4K}$, and $\cn$ is in the kernel of $L_+^\theta$ only for $\theta = 0$.  The bifurcation occurs only near $\theta=0$.

These two approaches are essentially the same, and our approach does not give a new instability branch.

\section{Numerics}
\label{sec:numerics}

We describe here the numerical experiments performed to understand
better the nature of the Jacobi elliptic functions as constrained minimizers of some functionals. 
To this aim, we use a normalized gradient flow approach related to the
minimization problem~\eqref{eq:min-prob-m-p}.

\subsection{Gradient Flow With Discrete Normalization}

It is relatively natural when dealing with constrained minimization problems like
\eqref{eq:min-prob-m-p}-\eqref{eq:min-prob-m-A} to use the following construction. 
Define an
increasing sequence of time $0=t_0<\dots<t_n$ and take an initial data
$u_0$. Between each time step,
let $u(t,x)$ evolve along the gradient flow
\[
\left\{
\begin{aligned}
u_t&=-\mathcal E'(u)=u_{xx}+b|u|^2u, \\
u(t_n,x)&=u_n(x),
\end{aligned}
\right.
\quad x\in\R,\;t_n<t<t_{n+1},\;n\geq 0.
\]
At each time step $t_n$, the function is renormalized so as to have
the desired mass and momentum. The renormalization for the mass is
obtained by a straightforward scaling:
\begin{equation}
  \label{eq:renormalization}
u_{n+1}(x):=u(t_{n+1},x)\sqrt{\frac{m}{\mathcal M(u(t_{n+1},x))}}. 
\end{equation}
When there is no momentum,  like in the minimization problems
\eqref{eq:min-prob-m},~\eqref{eq:min-prob-m-A}, and only real-valued functions are considered, such approach to compute the minimizers was
developed by Bao and Du \cite{BaDu04}.

However, dealing with complex valued solutions and with an additional
momentum constraint as in problems~\eqref{eq:min-prob-m-p},~\eqref{eq:min-prob-m-p-A} turns out to make the problem more challenging
and to our knowledge little is known about the strategies that one can
use to deal with this situation (see \cite{ChSc16} for an approach on
a related problem). 

To construct $u_n$ in such a way that $\mathcal P(u_n)=p$, a simple scaling is
not possible for at least two reasons. First of all, if $p=0$, a
scaling would obviously lead to failure of our strategy. Second, even
if $p\neq0$, as we are already using a scaling to get the correct
mass, making a different scaling to obtain the momentum constraint
will result into a modification of the mass. To overcome
these difficulties, we propose the following approach. 

Recall that, as noted in \cite{BaDu04}, the renormalizing step
\eqref{eq:renormalization} is equivalent to solving exactly the
following ordinary differential equation
\begin{equation}
\label{eq:2}
u_t=\mu_n u,
\quad t_n<t<t_{n+1},
\quad n\geq 0,
\quad \mu_n=\frac{1}{t_{n+1}-t_n}\ln\left(\frac{\sqrt{2m}}{\norm{u(t_n)}_{L^2}}\right).
\end{equation}
Inspired by this remark, we consider the following problem,
which we see as the equivalent of~\eqref{eq:2} for the
momentum renormalization. 
\begin{equation}
\label{eq:3}
u_t=i\varpi_n u_x,
\quad x\in\R,
\quad t_n<t<t_{n+1},
\quad n\geq 0,
\end{equation}
where we want to choose the values of $\varpi_n$ in such a way that
\(
\mathcal P(u(t_{n+1}))=p.
\)
To this aim, we need to solve~\eqref{eq:3}. Note that~\eqref{eq:3} is
a partial differential equation, whereas~\eqref{eq:2} was only an
ordinary differential equation. We make the following formal
computations, which can be justified if the functions involved are
regular enough. As we work with periodic functions, we consider
the Fourier series representation of $u$, that is
\[
u(t,x)=\sum_{j=-\infty}^{\infty}c_j (t)e^{i \frac{2\pi}{T}jx}
\]
with the Fourier coefficients
\[
c_j(t) = \frac{1}{T} \int_{-T/2}^{T/2} u(t,x) e^{-i \frac{2\pi}{T}jx} dx.
\]
Then~\eqref{eq:3} becomes
\begin{equation*}
\partial_tc_j=-\frac{2\pi}{T}j\varpi_n c_j, \quad j\in\mathbb Z,
\quad t_n<t<t_{n+1},\quad n\geq 0.
\end{equation*}
For each $j\in\mathbb Z$ and for any $t_n< t<t_{n+1}$ the solution is
\[
c_j(t)=\exp\left( -\frac{2\pi}{T}j\varpi_n (t-t_n)\right)c_j(t_n),
\]
 and
therefore the solution of~\eqref{eq:3} is
\[
u(t,x)=\sum_{j=-\infty}^{\infty}\exp\left( -\frac{2\pi}{T}j\varpi_n (t-t_n)\right)c_j(t_n)e^{i \frac{2\pi}{T}j x}.
\]
Using this Fourier series expansion of $u$, we have 
\[
\mathcal P(u(t_{n+1}))=-\sum_{j=-\infty}^{\infty} \pi j \exp\left( -\frac{4\pi}{T}j\varpi_n (t_{n+1}-t_n)\right)\abs{c_j(t_n)}^2.
\]
We determine implicitly the value of $\varpi_n$, by requiring that
$\varpi_n$ is such that
\[
\mathcal P(u(t_{n+1}))=p.
\]
In practice, it
might not be so easy to compute $\varpi_n$ and therefore we shall use
the following approximation. We replace the exponential by its first order
Maclaurin polynomial. We get 
\[
\mathcal P(u(t_{n+1}))=-\sum_{j=-\infty}^{\infty} \pi j \left(1-
  \frac{4\pi}{T}j\varpi_n (t_{n+1}-t_n)\right) \abs{c_j(t_n)}^2+\mathcal
O(\varpi_n^2 (t_{n+1}-t_n)^2).
\]
Therefore, an approximation for $\varpi_n$ is given by
$\tilde\varpi_n$, which is defined implicitly by
\[
p=-\sum_{j=-\infty}^{\infty} \pi j \left(1-
  \frac{4\pi}{T}j\tilde\varpi_n (t_{n+1}-t_n)\right) \abs{c_j(t_n)}^2.
\]
Solving for $\tilde\varpi_n $, we obtain
\[
\tilde\varpi_n=\bigg(p+\sum_{j=-\infty}^{\infty} \pi j \abs{c_j(t_n)}^2\bigg)\bigg((t_{n+1}-t_n) \frac{4\pi^2}{T}\sum_{j=-\infty}^{\infty} j^2 \abs{c_j(t_n)}^2\bigg)^{-1}.
\]
We can further simplify the expression of $\tilde\varpi_n $ by
remarking that
\[
\mathcal P(u(t_n))=-\sum_{j=-\infty}^{\infty} \pi j \abs{c_j(t_n)}^2,\quad
\int_{-T/2}^{T/2}|\partial_xu(t_n)|^2dx=\frac{4
  \pi^2}{T}\sum_{j=-\infty}^{\infty} j^2 \abs{c_j(t_n)}^2.
\]
This gives 
\[
\tilde\varpi_n=\frac{p-\mathcal P(u(t_n))}{(t_{n+1}-t_n) \norm{\partial_x u(t_n)}_{L^2}^2}.
\]
This is the value we will use in practice.  

\subsection{Discretization}

Let us now further discretize our problem. We first present a
semi-implicit time discretization, given by the following scheme. 
\begin{gather*}
  \frac{\tilde u_{n+1}- u_n}{\delta t}=\partial_{xx}\tilde
  u_{n+1}+b|u_n|^2\tilde u_{n+1}, 
\quad\tilde u_{n+1}\in P_{T},
\\ 
\hat u_{n+1}=\sum_{j=-\infty}^{\infty}  c_j(\tilde u_{n+1})
\left(1-\frac{2\pi}{T}\delta t\tilde\varpi_nj\right)e^{i\frac{2\pi}{T}jx},
\\
 u_{n+1}=\hat u_{n+1} \sqrt{\frac{m}{\mathcal M(\hat
    u_{n+1})}},
\end{gather*}
where $\tilde\varpi_n$ is given by 
\[
\tilde\varpi_n=\frac{p-\mathcal P(u_n)}{\delta t \norm{\partial_x u_n}_{L^2}^2},
\]
and $(c_j(\tilde u_{n+1}))$ are the Fourier coefficients of $\tilde
u_{n+1}$. 
Note that the system is linear. 

\begin{remark}
  If $p=0$, at the end of each step, $u_{n+1}$ has the desired mass and
  momentum. If $p\neq 0$, then $u_{n+1}$ only has the desired
  mass and it is unclear if the algorithm will still give
  convergence toward the desired mass-momentum constraint
  minimizer. We plan to investigate this question in further works. 
\end{remark}

Finally, we present the fully discretized problem. We discretize the
space interval $\left[-\frac T2,\frac T2\right]$ by
setting
\[
x^0=-\frac T2,\quad x^l=x^0+l\delta x, \quad \delta x=\frac TL, \quad L\in
2\mathbb N. 
\]
We denote by $u_n^l$ the numerical approximation of
$u(t_n,x^l)$. Using the (backward Euler) semi-implicit scheme for time
discretization and second-order centered finite difference for spatial
derivatives, we obtain the following scheme. 

\begin{gather}
  \label{eq:333}
  \frac{\tilde u_{n+1}^l- u_n^l}{\delta t}=\frac{\tilde
  u_{n+1}^{l-1}-2\tilde
  u_{n+1}^l+\tilde
  u_{n+1}^{l+1}}{\delta x^2}+b|u_n^l|^2\tilde u_{n+1}^l, \quad
 u_{n+1}^0=u_{n+1}^L,
\\
\hat u_{n+1}^l=
\sum_{j=-L/2}^{L/2}  c_j(\tilde u_{n+1})
 \left(1-\frac{2\pi}{T}\delta t\tilde\varpi_nj\right)e^{i\frac{2\pi}{L}jl \delta x},  \label{eq:333-2}
\\
\tilde u_{n+1}^l=\hat  u_{n+1}^l \sqrt{\frac{m}{\mathcal M(\hat    u_{n+1})}},\label{eq:333-3}
\end{gather}
where $c_j(\tilde u_{n+1})=\frac{1}{L+1}\sum_{l=0}^{L}  
\tilde u_{n+1}^l e^{i\frac{2\pi}{L}jl \delta x}.$

As the system~\eqref{eq:333} is linear, we can solve it using a Thomas algorithm for tridiagonal matrix modified to take into
account the periodic boundary conditions. The discrete Fourier
transform and its inverse are computed using the built in Fast Fourier
Transform algorithm. 

We have not gone further in the analysis of the scheme presented
above. As shown in the next section, the outcome of the numerical experiments are in good
agreement with the theoretical results. We plan to further analyze and generalize our approach in future works.

\section{Numerical Solutions of Minimization Problems}
\label{sec:num-exp}
Before presenting the numerical experiments, we introduce some
notation for particular plane waves. Define
\[
\varphi_{\mu ,\rho}=\sqrt{ \frac{2\mu }{T}}e^{-i\frac{\rho}{\mu}
  x},\text{ the
  plane wave with }\mathcal M(\varphi_{\mu ,\rho})=\mu \text{ and
}\mathcal P(\varphi_{\mu ,\rho})=\rho.
\]

In the numerical experiments, we have chosen to fix $k=0.9$. The
period will be either $T=2K(k)$ or $T=4K(k)$. We use $2^{10}$
grid points for the interval $[-\frac T2,\frac T2]$. The time step
will be set to $1$. We decided to run the algorithm until a maximal difference of
$10^{-3}$ between the absolute values of the moduli of
$u_j^l$ and the expected minimizer has been reached.

We made the tests with the following initial data:
\begin{equation}\label{eq:initial-data}
(a)\, u_0(x)=5,\quad
(b)\, u_0(x)=\exp(2i\pi x/T),\quad
(c)\, u_0(x)=1+\cos(2i\pi x/T)+i.
\end{equation}
Depending on the expected profile, we may have shifted $u_j$ so that a minimum
or a maximum of its modulus is at
the boundary. Since the problem is translation invariant, this causes
no loss of generality. 

Since the initial  data $u_0$ in  \eqref{eq:initial-data} do not match the required mass/momentum, $u_1$ are very different from $u_0$. Thus  \eqref{eq:initial-data} is a random choice, and this shows up in the rapid drop from $t_0$ to $t_1$ in Figure \ref{fig:focusing-P-small-mass}.
The idea is to show that the choice of initial data is not important
for the algorithm and that no matter from where the algorithm is starting, it
converges to the supposed minimizer (unless the initial
  data has some symmetry preserved by the algorithm).

\subsection{Minimization Among Periodic Functions}

Minimization among periodic functions is completely covered by the
theoretical results Propositions~\ref{prop:focusing-P_T} and
\ref{prop:defocusing-P}. We have performed different tests using the
scheme described in~\eqref{eq:333}-\eqref{eq:333-3} and we have found
that the numerical results are in good agreement with the theoretical
ones. 

\subsubsection{The Focusing Case}

In all the experiments performed in this case, we have tested the
scheme with and without the momentum renormalization step
\eqref{eq:333-3} and we have obtained the same result each time. This
confirms that in the periodic case the momentum
constraint plays no role (see (i)   in Proposition
\ref{prop:focusing-P_T}, and Proposition~\ref{prop:defocusing-P}). In
what follows, we present only the results obtained using the full
scheme with renormalization of mass and momentum. 

We fix $T=2K(k)$ and $b=2$. We first perform an experiment to verify
the agreement with case (ii) in Proposition
\ref{prop:focusing-P_T}. Let
$m=\frac{\pi^2}{8K}<\frac{\pi^2}{bT}$. With each initial data in
\eqref{eq:initial-data}, we observe convergence towards the constant
solution, hereby confirming case (ii) of  Proposition
\ref{prop:focusing-P_T}. The results are presented in Figure
\ref{fig:focusing-P-small-mass} for
initial data (c) of~\eqref{eq:initial-data}. The requested precision is
achieved after $12$ time steps. 
\begin{figure}[htpb!]
  \centering
  \includegraphics[width=0.49\textwidth]{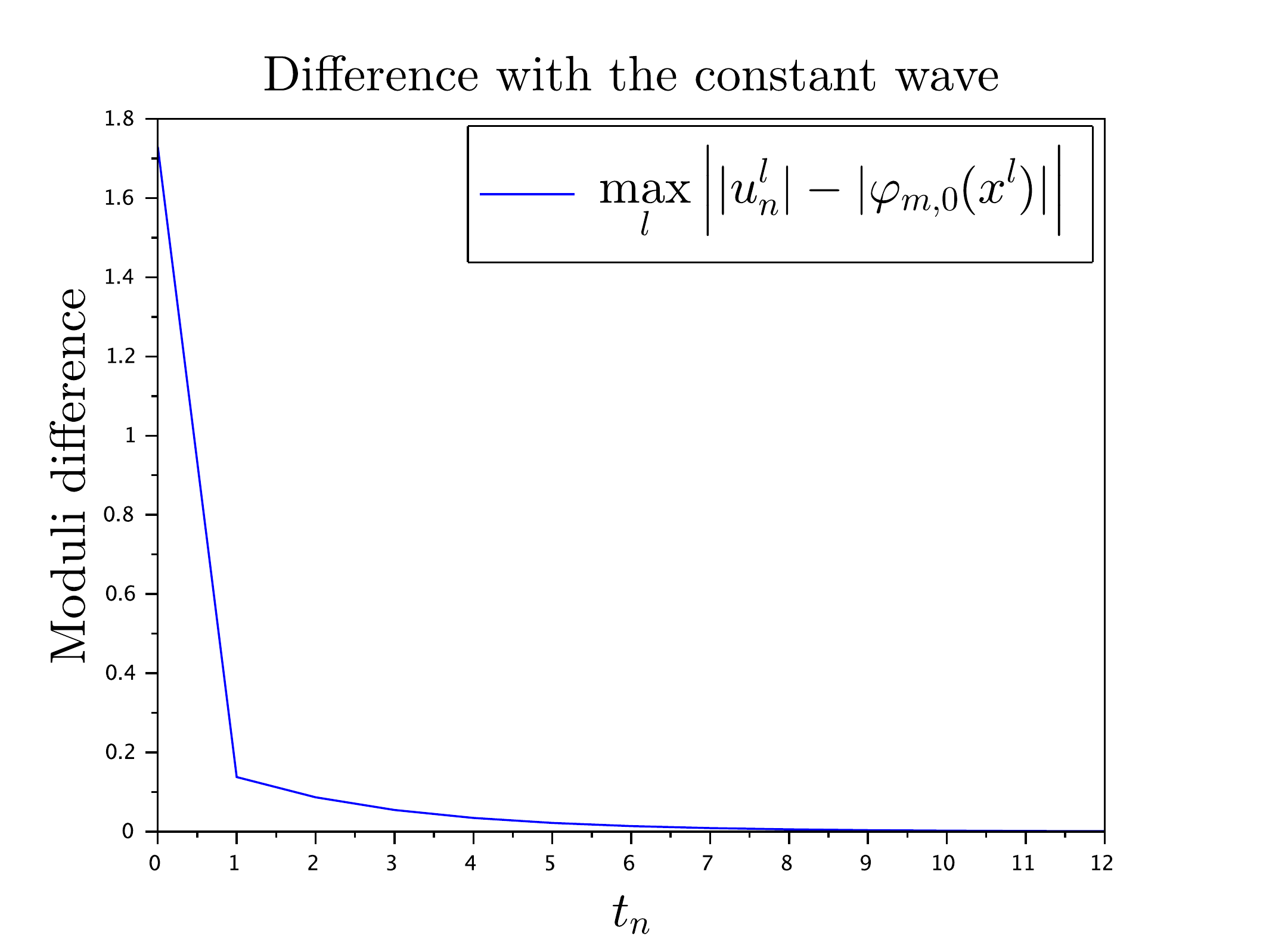}
  \hfill
  \includegraphics[width=0.49\textwidth]{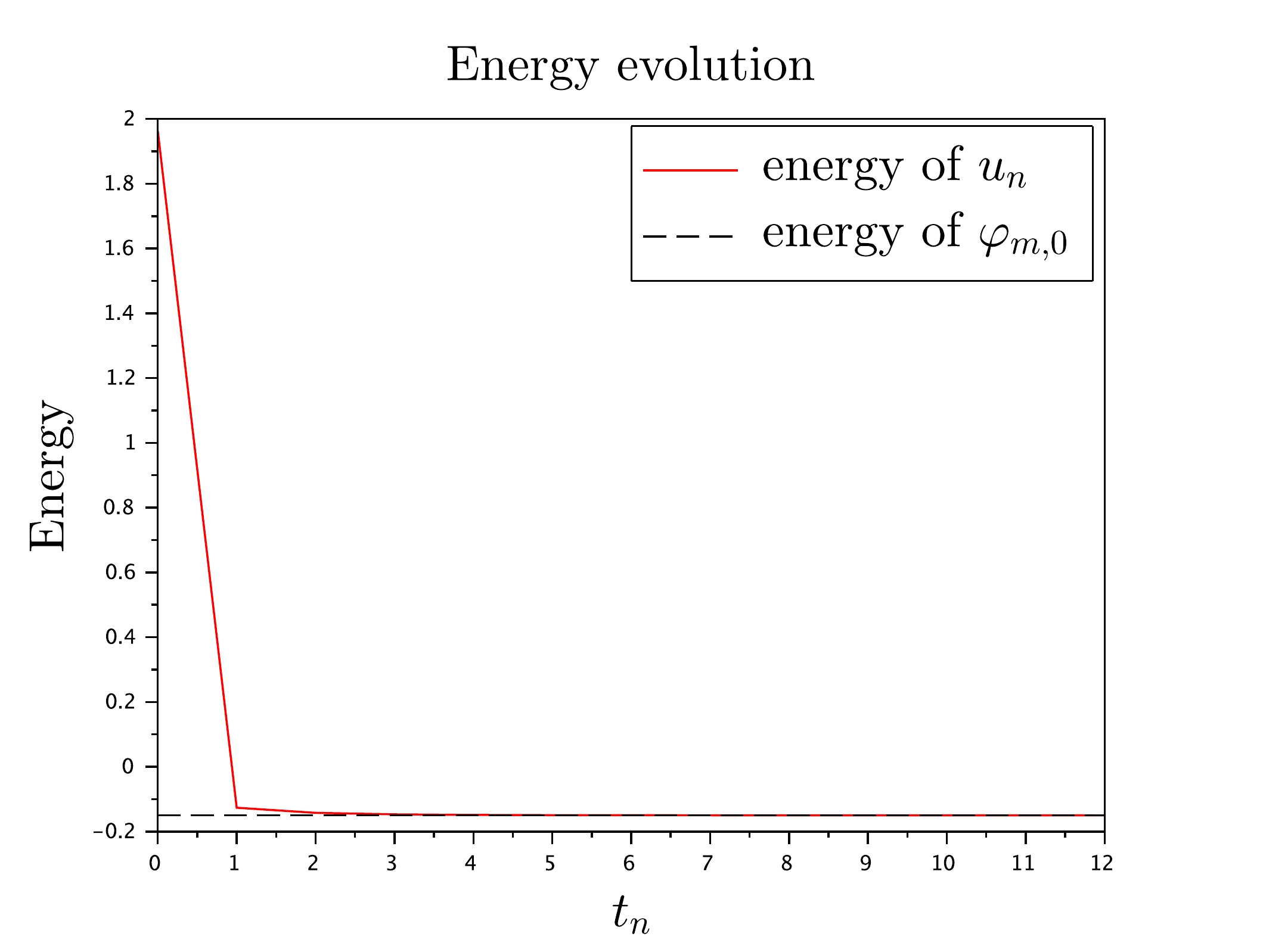}
  \caption{For $m=\frac{\pi^2}{8K}<\frac{\pi^2}{bT}$, focusing, periodic case}
  \label{fig:focusing-P-small-mass}
\end{figure}

The second experiment that we perform is aimed at testing case (iv) of
Proposition~\ref{prop:focusing-P_T}.  Let
$m=\mathcal M(\dn)=E(k)$. 
Once again we observe a good
agreement between the theoretical prediction and the numerical
experiment. The results are presented in Figure
\ref{fig:focusing-P-dn} for
initial data (c) of~\eqref{eq:initial-data}. The requested precision is
achieved after $14$ time steps. 
\begin{figure}[htpb!]
  \centering
  \includegraphics[width=0.49\textwidth]{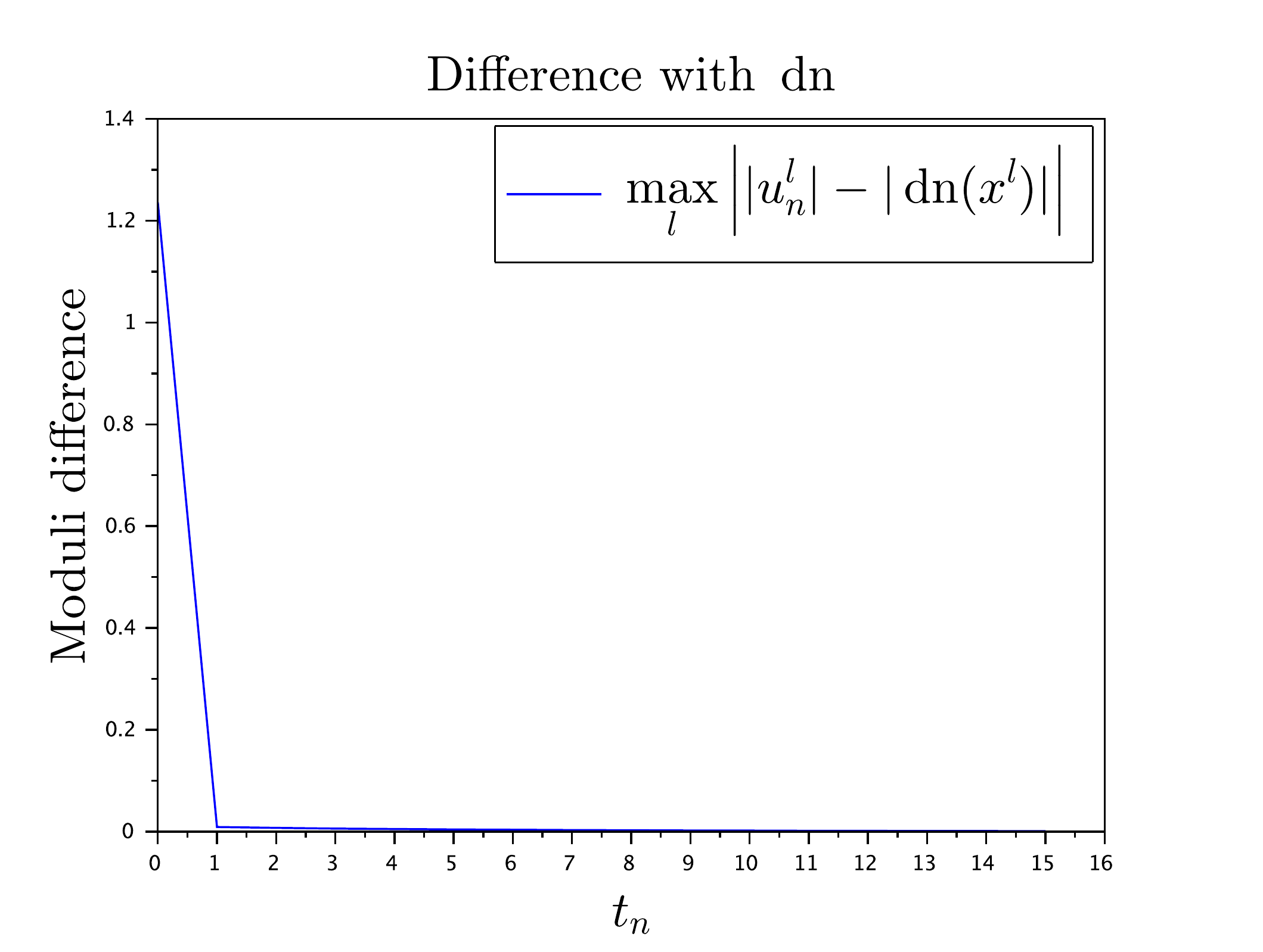}
  \hfill
  \includegraphics[width=0.49\textwidth]{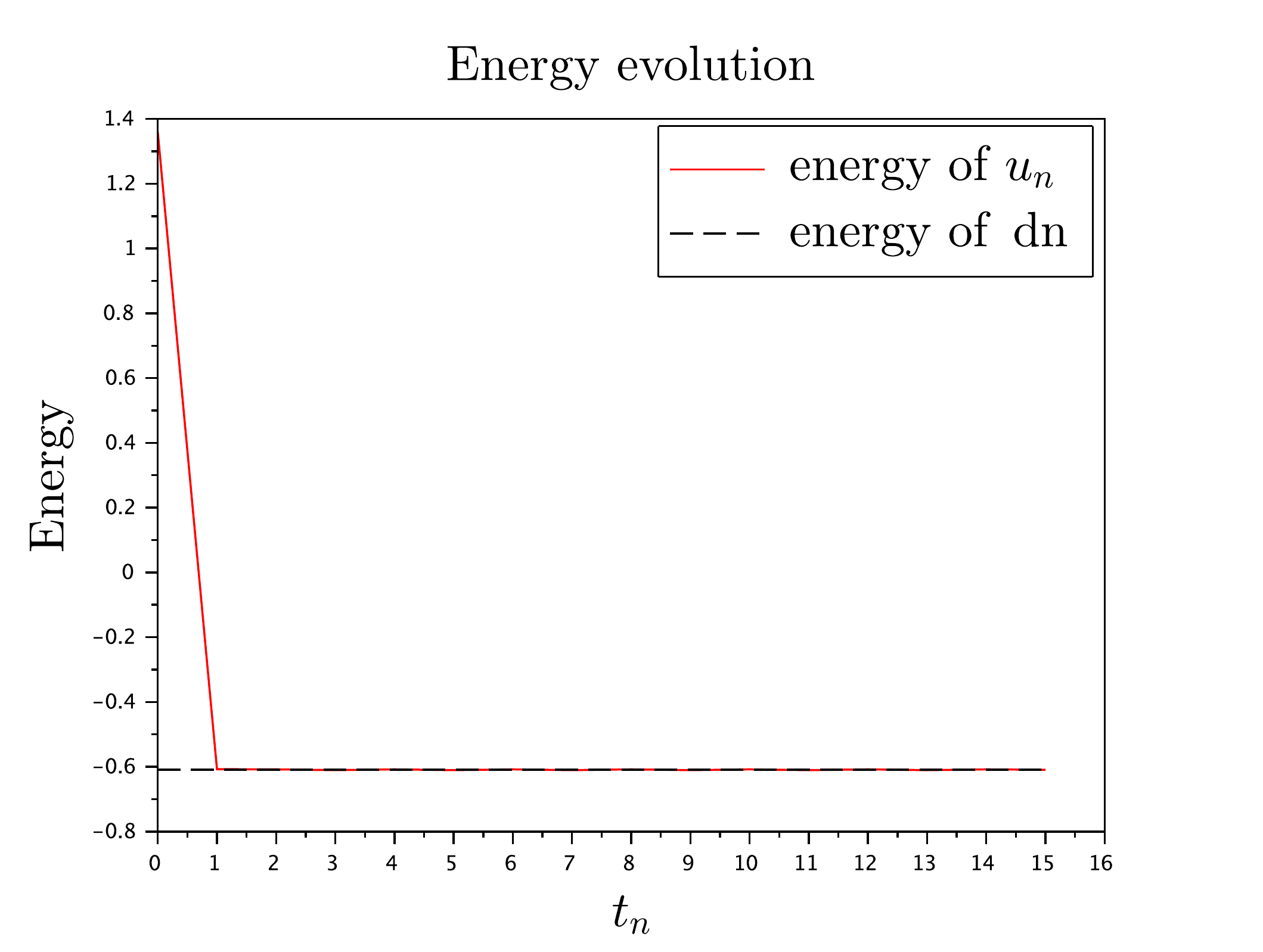}
  \caption{For $m=\mathcal M(\dn)=E(k)$, focusing, periodic case}
  \label{fig:focusing-P-dn}
\end{figure}

All the other experiments that we have performed show a good agreement
with the theoretical results  in the focusing case for minimization
among periodic functions. To avoid repetition, we give no further
details here. 

\subsubsection{The Defocusing Case}

We now present the experiment in the defocusing case. We have used
$b=-2k^2$ and $T=4K$. We have tested the algorithm with and without
the momentum renormalization step~\eqref{eq:333-3}, obtaining the same
results. The results are presented in Figure~\ref{fig:defocusing-P} for
initial data (c) of~\eqref{eq:initial-data} and mass constraint $m=\mathcal M(\sn)=\frac{2(K-E)}{k^2}$. 
The requested precision is
achieved after $6$ time steps. 

\begin{figure}[htpb!]
  \centering
  \includegraphics[width=0.49\textwidth]{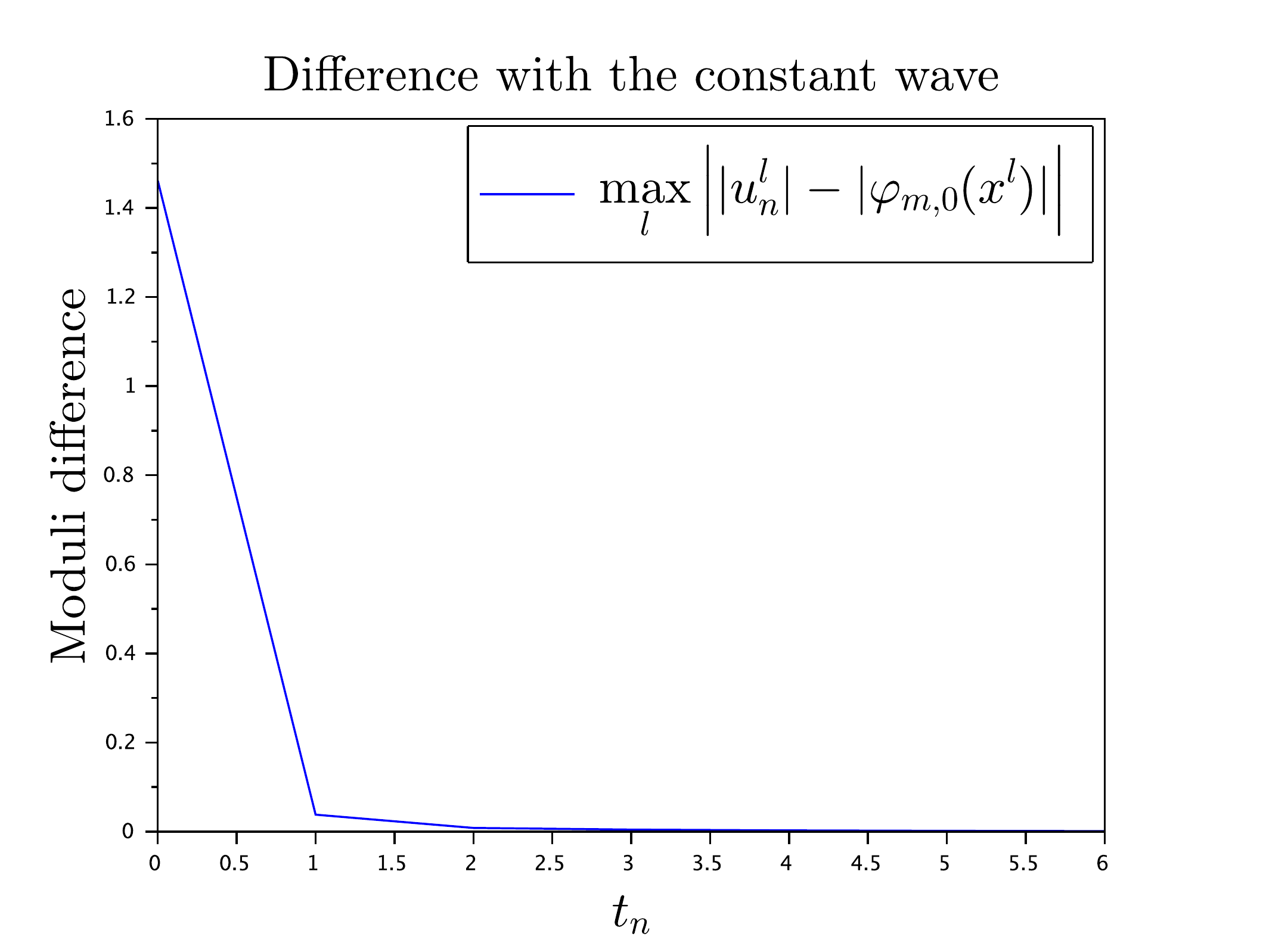}
  \hfill
  \includegraphics[width=0.49\textwidth]{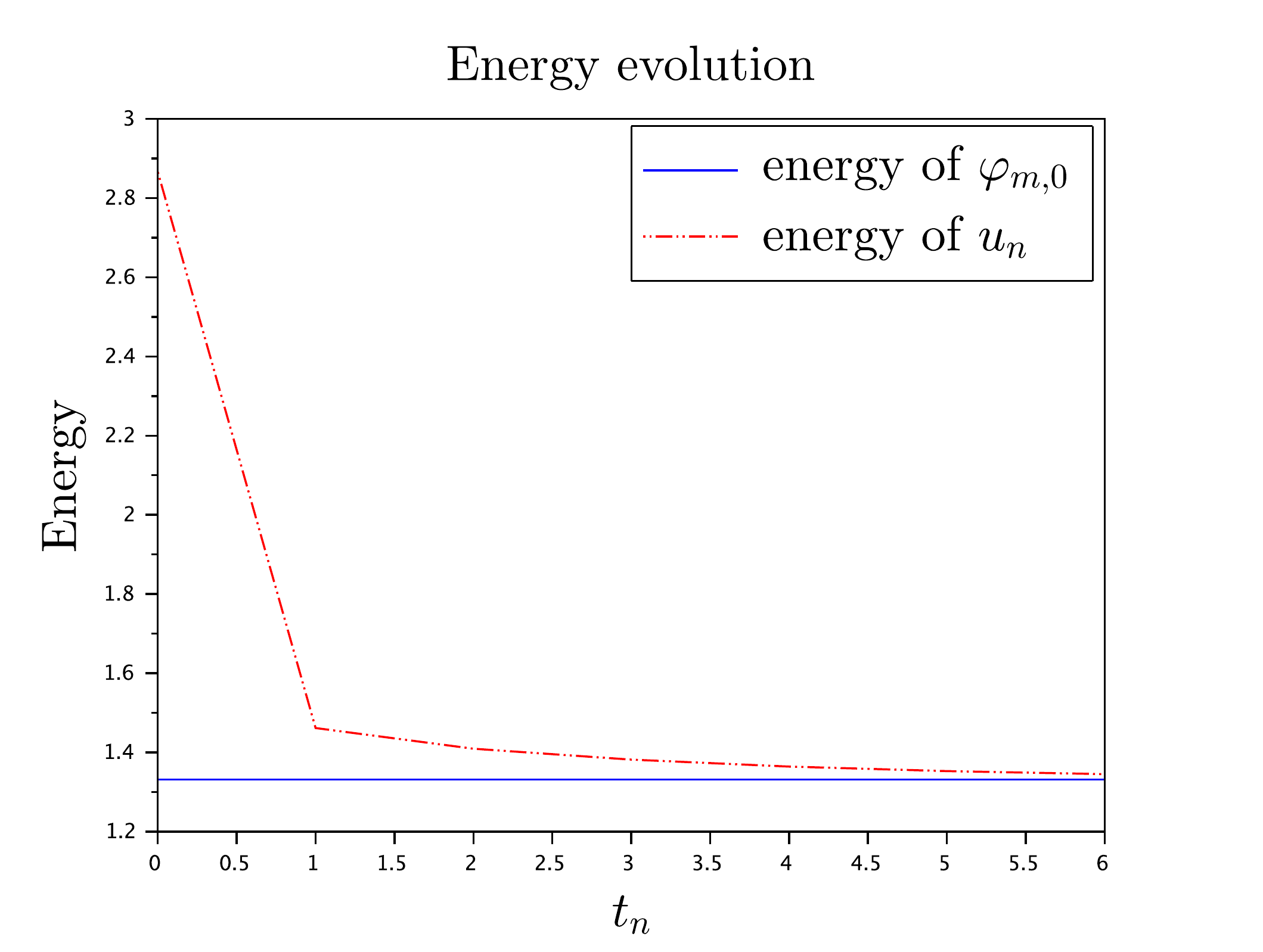}
  \caption{For $m=\mathcal M(\sn)=\frac{2(K-E)}{k^2}$, defocusing, periodic case}
  \label{fig:defocusing-P}
\end{figure}

\subsection{Minimization Among Half-Anti-Periodic Functions}

We will in that case
add an additional step in the algorithm in which we keep only the
anti-periodic part of the function. This way it will not matter wether
or not our initial data has the right anti-periodicity, since 
anti-periodicity will be forced at each iteration of the algorithm.

\subsubsection{The Focusing Case}
We compare in this section the numerical results with Proposition
\ref{prop:focusing-A}. We have used
$b=2k^2$ and $T=4K$. The tests performed show a good agreement
between the numerics and the theoretical result.  We present in Figure
\ref{fig:focusing-A} the result for
initial data (c) of~\eqref{eq:initial-data} and mass constraint $m=\mathcal M(\cn)=2(E-(1-k^2)K)/k^2$

\begin{figure}[htpb!]
  \centering
  \includegraphics[width=0.49\textwidth]{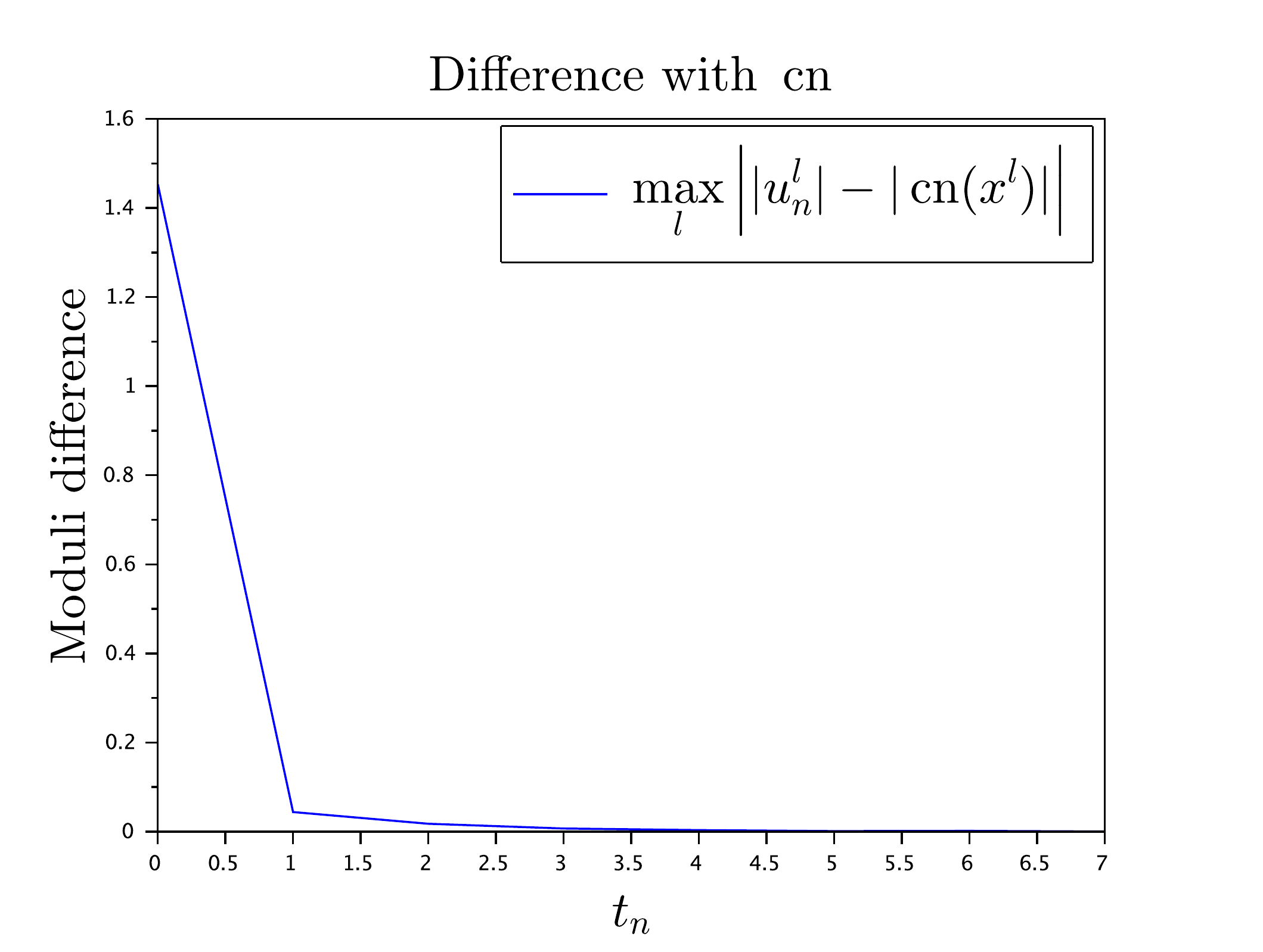}
  \hfill
  \includegraphics[width=0.49\textwidth]{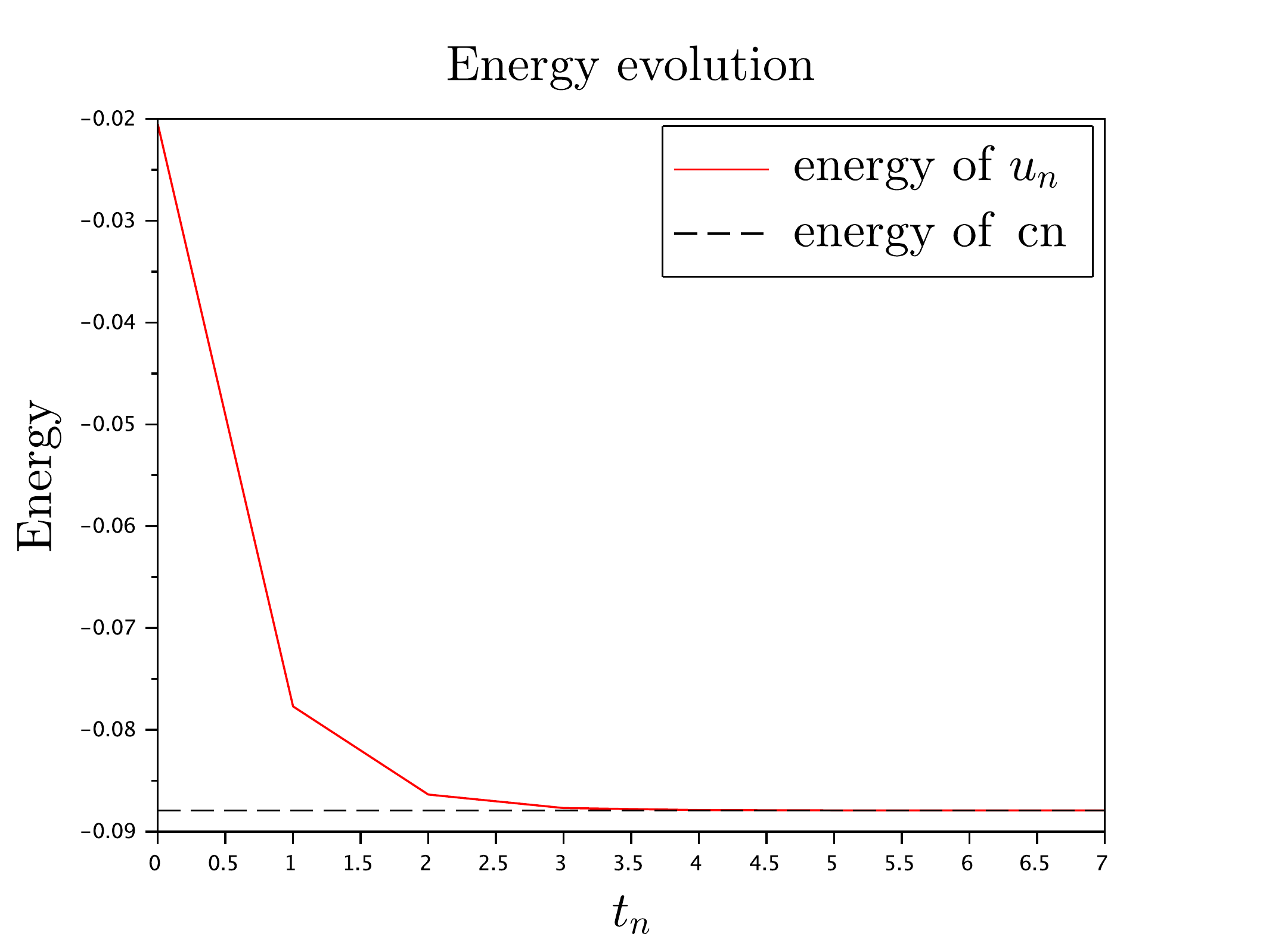}
  \caption{For $m=\mathcal M(\cn)=2(E-(1-k^2)K)/k^2$, focusing, anti-periodic case}
  \label{fig:focusing-A}
\end{figure}

\subsubsection{The Defocusing Case}

We finally turn out to the defocusing case, still imposing anti-periodicity. 
We have used
$b=-2k^2$ and $T=4K$. 

We have tested the algorithm without
the momentum renormalization step~\eqref{eq:333-3} and confirmed the
theoretical result Proposition~\ref{prop:defocusing-A}, which states
that a plane wave is the minimizer. We present the
result in Figure~\ref{fig:defocusing-A} for
initial data (c) of~\eqref{eq:initial-data} and mass constraint
$m=\mathcal M(\sn)=\frac{2(K-E)}{k^2}$. Note a plateau in the two graphs of
Figure~\ref{fig:defocusing-A}. This is due to the fact that the
sequence remains for some time close to $\sn$ (which is the expected
minimizer if we impose in addition the momentum constraint), before
eventually converging to the plane wave minimizer.

\begin{figure}[htpb!]
  \centering
  \includegraphics[width=0.49\textwidth]{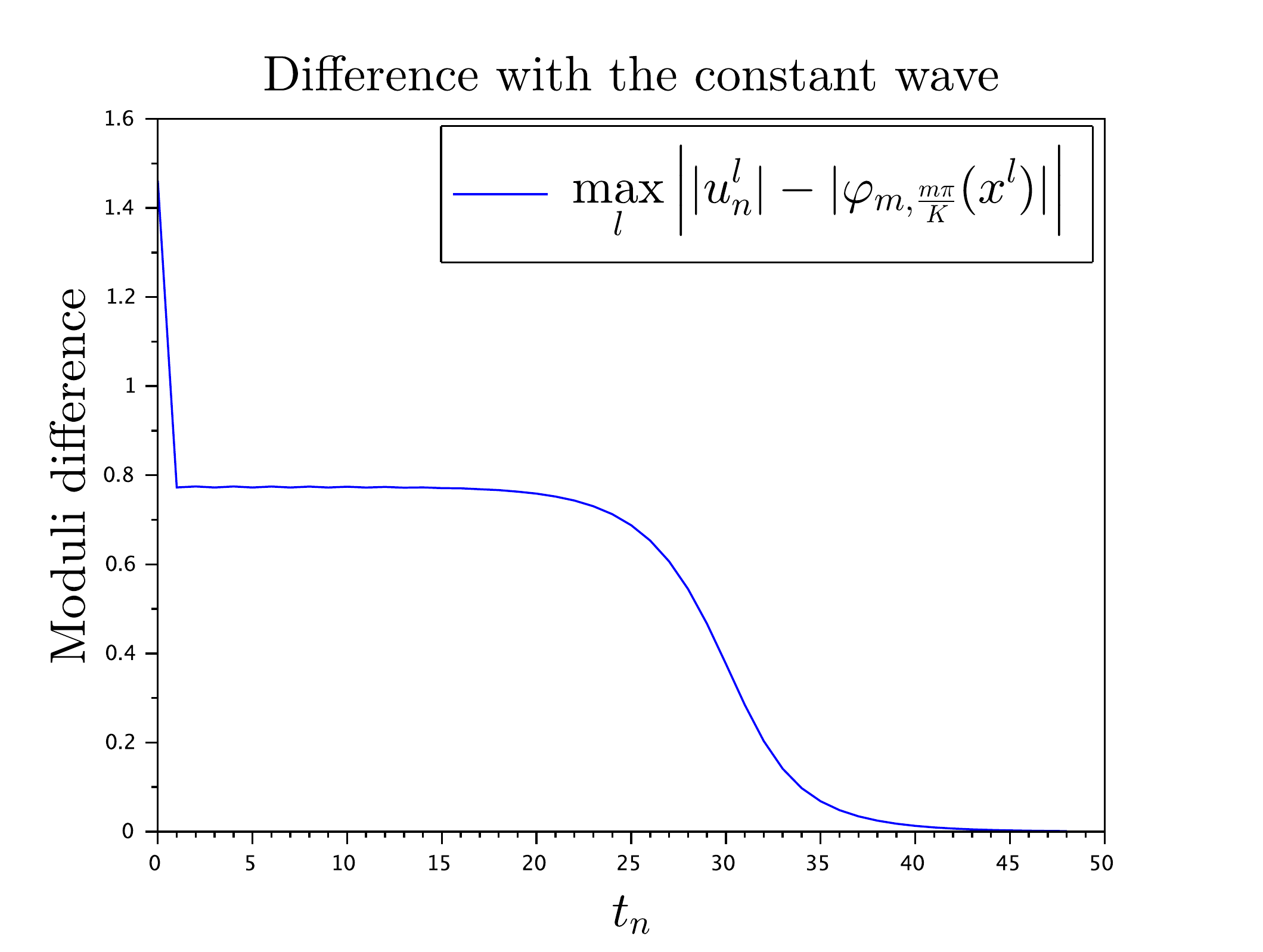}
  \hfill
  \includegraphics[width=0.49\textwidth]{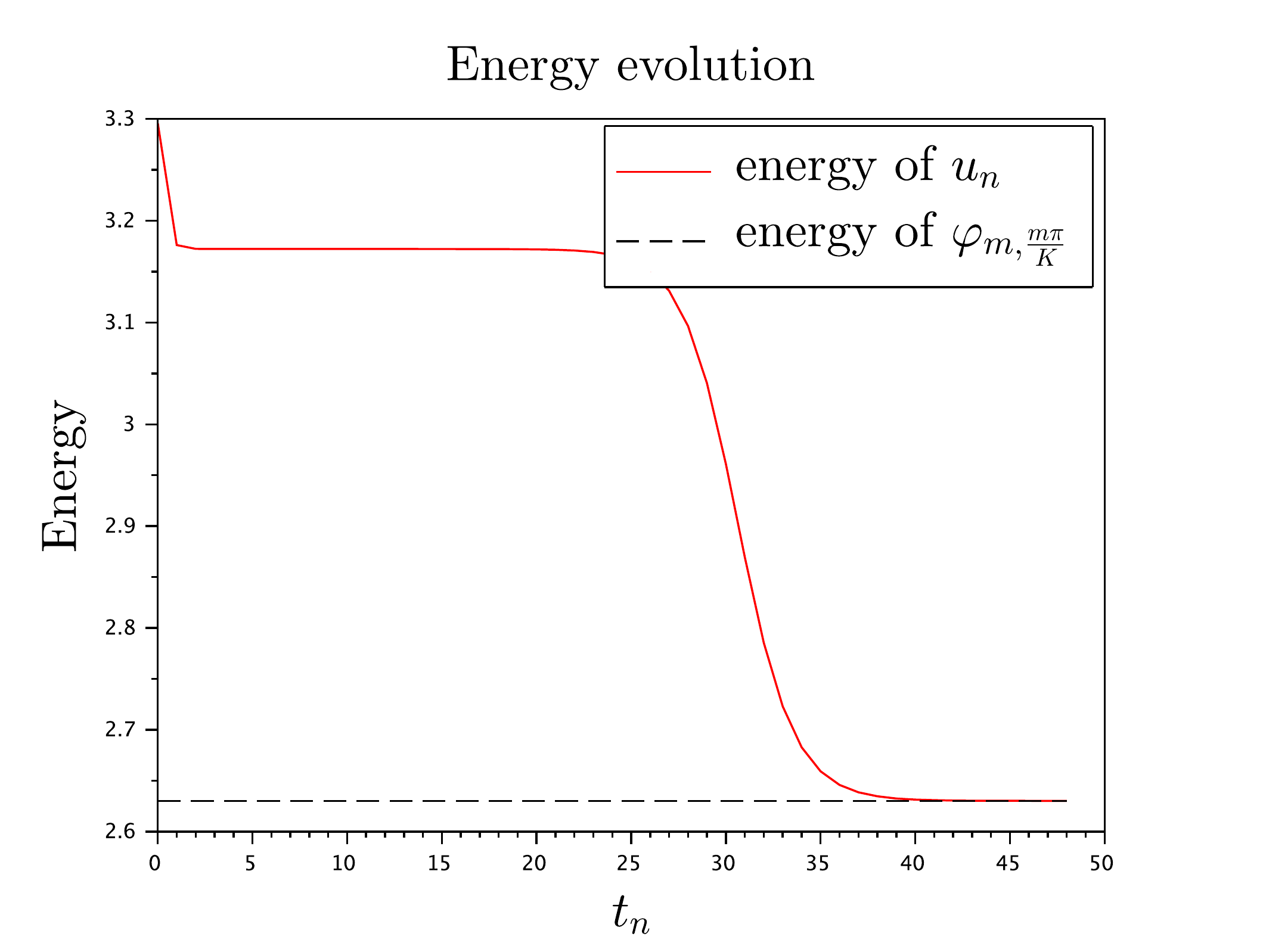}
  \caption{For $m=\mathcal M(\sn)=2(E(k)-K)/k^2$, defocusing,
    anti-periodic case without momentum constraint}
  \label{fig:defocusing-A}
\end{figure}

Finally, we run the full algorithm with mass and momentum
renormalization for  mass constraint
$m=\mathcal M(\sn)=\frac{2(K-E)}{k^2}$ and  $0$ momentum constraint. No
theoretical result is available in this case. We made the following
observation, which confirms Conjecture~\ref{conj:minimizer-2}.

\begin{observation}\label{obs:conjecture}
  The function $\sn$ is a minimizer for problem~\eqref{eq:min-prob-m-p-A}
  with $m=\mathcal M(\sn)$. 
\end{observation}

 We present in Figure~\ref{fig:defocusing-A-p} the result of the
 experiment with full algorithm for
initial data (c) of~\eqref{eq:initial-data} and mass constraint
$m=\mathcal M(\sn)=\frac{2(K-E)}{k^2}$. 

\begin{figure}[htpb!]
  \centering
  \includegraphics[width=0.49\textwidth]{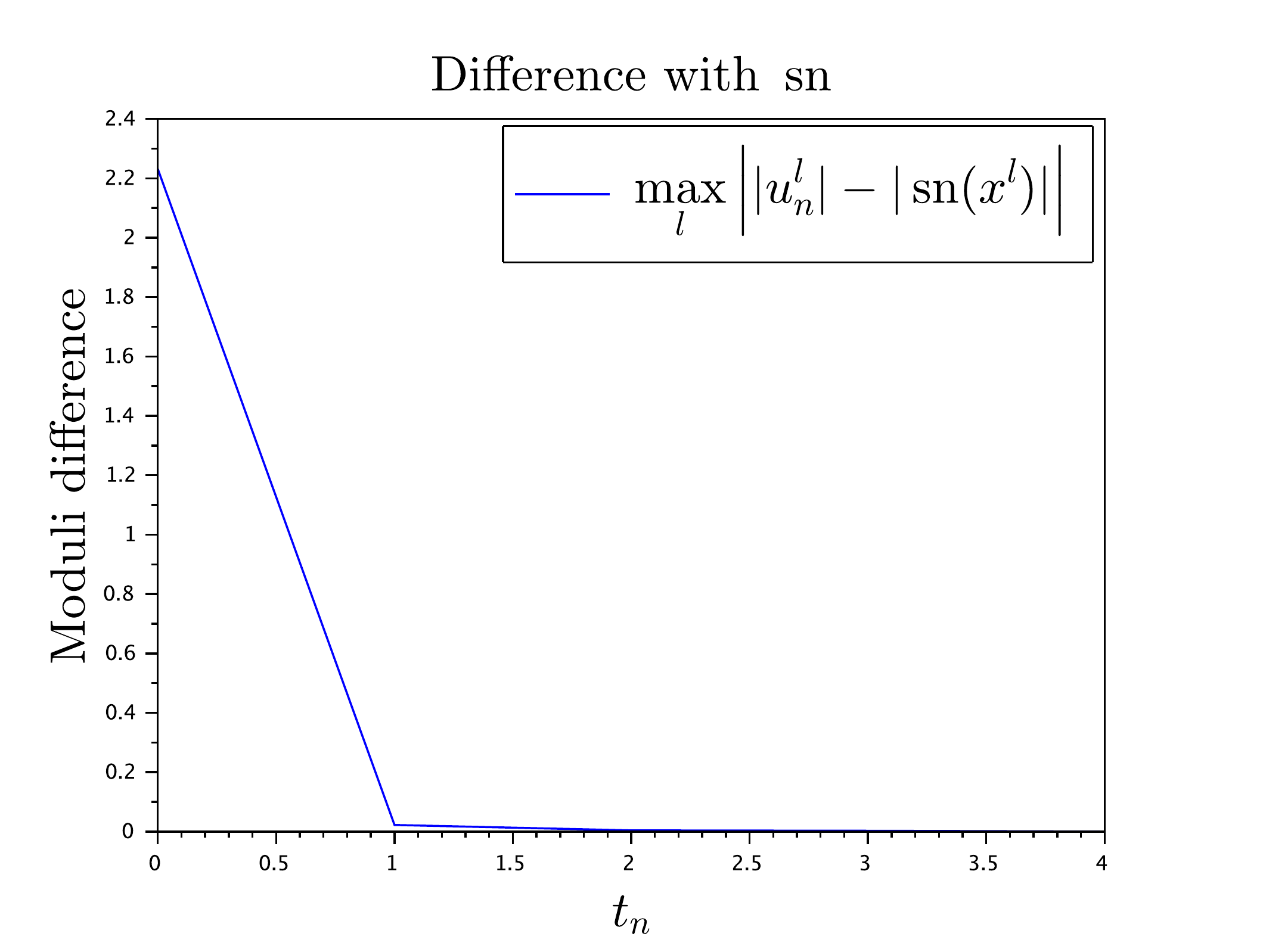}
  \hfill
  \includegraphics[width=0.49\textwidth]{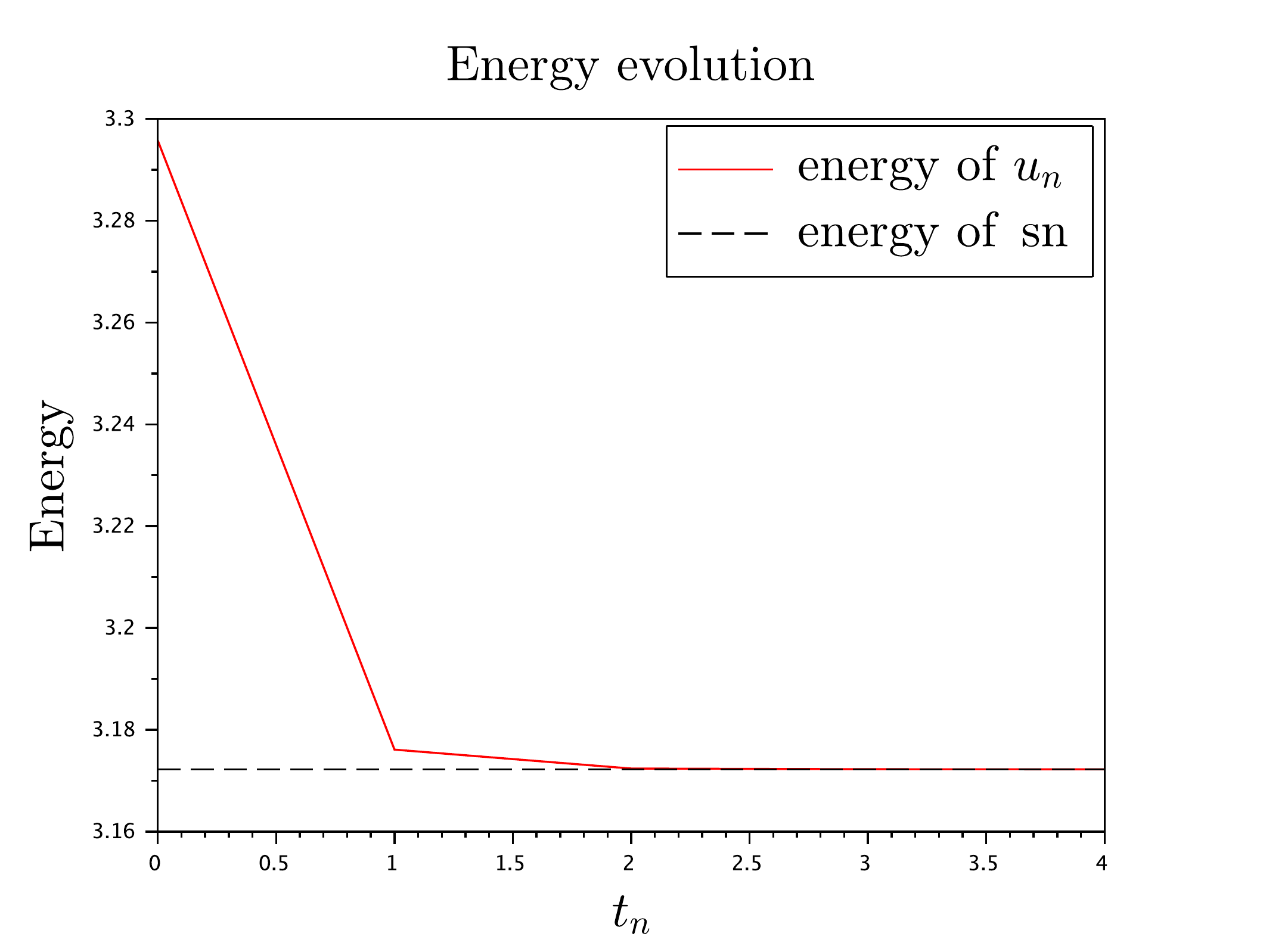}
  \caption{For $m=\mathcal M(\sn)=2(E(k)-K)/k^2$, defocusing,
    anti-periodic case with momentum constraint}
  \label{fig:defocusing-A-p}
\end{figure}

 \bibliographystyle{abbrv}
 \bibliography{gustafson-lecoz-tsai}
 
\end{document}